\documentclass[12pt]{amsart}
\usepackage{amsmath} 
\usepackage{amssymb}
\usepackage{amsfonts}

\newcommand{\bc}{\mathbf C}

\newcommand{\br}{\mathbf R}
\newcommand{\bz}{\mathbf Z}
\newcommand{\bn}{\mathbf N}
\newcommand{\Cal}{\mathcal}

\newcommand{\coker}{\operatorname{Coker}}

\renewcommand{\dim}{\operatorname{Dim}}

\newcommand{\id}{\operatorname{Id}}
\renewcommand{\iff}{\Leftrightarrow}
\newcommand{\im}{\operatorname{Im}}
\newcommand{\interior}{\overset{\circ}}
\renewcommand{\ker}{\operatorname{Ker}}

\newcommand{\lip}{\operatorname{Lip}}

\newcommand{\mn}[1]{\Vert#1\Vert}

\newcommand{\op}{\operatorname{Op}}
\newcommand{\ol}{\overline}

\newcommand{\ran}{\operatorname{Ran}}

\newcommand{\re}{\operatorname{Re}}
\renewcommand{\Re}{\operatorname{Re}}
\newcommand{\restr}[1]{\big|_{#1}}

\newcommand{\set}[1]{\left\{\,#1\,\right\}}

\newcommand{\supp}{\operatorname{\rm supp}}

\newcommand{\sw}[1]{\left ( #1\right ) }

\newcommand{\w}[1]{\langle #1\rangle }

\newcommand{\wf}{\operatorname{WF}}

\newcommand{\wt}{\widetilde}

\newcommand{\eig}{\operatorname{Eig}}
\newcommand{\var}{\operatorname{var}}
\newcommand{\rb}{{\rm b}}
\newcommand{\spec}{\operatorname{Spec}}

\numberwithin{equation}{section}

\gdef\theoremheaderfont#1{\gdef\theorem@headerfont{#1}}

\def\skipmath@b#1\(#2\){{#1}
  \ifx\skipmath@b#2\else\(#2\)\expandafter\skipmath@b\fi}

\theoremheaderfont{\sc}
\newtheorem{thm}{Theorem}[section]
\newtheorem{lem}[thm]{Lemma}

\newtheorem{prop}[thm]{Proposition}

\newtheorem{rem}[thm]{Remark}

\theoremstyle{definition}

\newtheorem{defn}[thm]{Definition}

\newtheorem{exe}[thm]{Example}

\theoremstyle{remark}

\begin{document}

\title[Pseudospectrum for systems]
{The Pseudospectrum of Systems\\of Semiclassical Operators}
\author[NILS DENCKER]{{\textsc Nils Dencker}}
\address{Centre for Mathematical Sciences, University of Lund, Box 118,
SE-221 00 Lund, Sweden}
\email{dencker@maths.lth.se}
\date{December 1, 2008} 

\baselineskip 18,5pt 
\lineskip 2pt
\lineskiplimit 2pt






\maketitle

\section{Introduction}

In this paper we shall study the pseudospectrum or spectral instability
of square non-selfadjoint semiclassical systems of principal type. Spectral instability
of non-selfadjoint operators is currently a topic of interest in
applied mathematics, see~\cite{Dav} and  \cite{Trebook}. It arises from the fact that, for 
non-selfadjoint operators, the resolvent could be very large in an open
set containing the spectrum. For semiclassical differential operators, this
is due to the bracket condition and is connected to the
problem of solvability. In applications where one needs to compute the
spectrum, the spectral instability has the consequence that discretization and 
round-off errors give false spectral values, so
called pseudospectrum, see ~\cite{Trebook} and references there.

We shall consider bounded systems $P(h)$ of semiclassical operators
given by~\eqref{Pdef}, and we shall generalize the results of the
scalar case in~\cite{dsz}. Actually, the study of unbounded operators can in many
cases be reduced to the bounded case, see Proposition~\ref{reduxlem} and
Remark~\ref{reduxrem}. We shall also study semiclassical  
operators with analytic symbols, in the case when the symbols
can be extended analytically to a tubular neighborhood of the phase
space satisfying~\eqref{analsymb}. The operators that we consider are of
principal type, which means 
that the principal symbol vanishes of first order on
the kernel, see Definition~\ref{princtype}.

The definition of {\em semiclassical} pseudospectrum in ~\cite{dsz} is
essentially the bracket condition, which is suitable for symbols
of principal type. By instead using the definition of
(injectivity) pseudospectrum by Pravda-Starov~\cite{p-s} we obtain
a more refined view of the spectral instability, see
Definition~\ref{psdef}. For example, $z$ is in the
pseudospectrum of infinite index for $P(h)$ if for any ~$N$ the
resolvent norm blows up faster than any power of the semiclassical parameter:
\begin{equation}\label{pspec}
 \mn{(P(h) - z\id)^{-1}} \ge C_N h^{-N}\qquad 0 < h \ll 1
\end{equation}

In ~\cite{dsz} it was proved that~\eqref{pspec} holds almost everywhere in the
{\em semiclassical} pseudospectrum. We shall generalize this to systems and
prove that for systems of principal type, except for a nowhere dense set
of degenerate values, the resolvent blows 
up as in the scalar case, see Theorem~\ref{bracketthm}. The complication is that 
the eigenvalues don't have constant multiplicity in general, only almost everywhere.

At the boundary of the semiclassical pseudospectrum, we obtained in~\cite{dsz} 
a bound on the norm of the semiclassical resolvent, under the
additional condition of having no unbounded (or closed)
bicharacteristics. In the systems case, the picture is more 
complicated and it seems to be difficult to get an estimate on the norm of the resolvent 
using only information about the eigenvalues, even in the principal
type case, see Example~\ref{saex}. In fact, the norm is
essentially preserved under multiplication with elliptic systems, but
the eigenvalues are changed. Also, the
multiplicities of the eigenvalues could be changing at all points 
on the boundary of the eigenvalues, see Example~\ref{princex}. We shall instead introduce
{\em quasi-symmetrizable} systems, which generalize
the normal forms of the scalar symbols at the boundary of the
eigenvalues, see Definition~\ref{QS}. 
Quasi-symmetrizable systems are of principal type and we obtain
estimates on the resolvent as in the scalar case, see Theorem~\ref{QSthm}.

For boundary points of {\em finite type} we obtained in~\cite{dsz}
subelliptic type of estimates on the semiclassical resolvent. This is
the case when one has non-vanishing higher order brackets. 
For systems the situation is less clear, 
there seems to be no general results on the subellipticity for
systems. In fact, the real and imaginary parts
do not commute in general, making the bracket condition meaningless.
Even when they do, Example ~\ref{ex1} shows that the bracket condition is not 
sufficient for subelliptic type of estimates. 
Instead we shall introduce invariant conditions on the order of vanishing of the symbol
along the bicharacteristics of the eigenvalues.
For systems, there could be several (limit) bicharacteristics  of the
eigenvalues going through a characteristic point, see Example~\ref{subex}. 
Therefore we introduce the {\em approximation} property
in Definition~\ref{apprdef} which gives that the all (limit)
bicharacteristics of the eigenvalues are parallell at the characteristics,
see Remark~\ref{condrem}. The general case presently looks too complicated to handle.
We shall generalize the property of being of finite type to systems,
introducing systems of {\em subelliptic type}. These are
quasi-symmetrizable systems satisfying the approximation property, such that
the imaginary part on the kernel vanishes of finite order along the
bicharacteristics of the real part of the eigenvalues.
This definition is invariant under multiplication with invertible systems and taking adjoints,
and for these systems we obtain subelliptic types of estimates on the
resolvent, see Theorem~\ref{subthm}.

As an example, we may look at
 \begin{equation*}
P(h) = h^2{\Delta}\id_N + i K(x)
 \end{equation*}
where ${\Delta} = -\sum_{j=1}^n \partial_{x_j}^2$ is the positive Laplacean,
and $K(x)\in C^\infty(\br^n)$ is a symmetric $N \times N$ system.
If we assume some conditions of ellipticity at infinity for
$K(x)$, we may reduce to the case of bounded symbols by Proposition~\ref{reduxlem} and
Remark~\ref{reduxrem}, see Example~\ref{sysex}. 
Then we obtain that $P(h)$
has discrete spectrum in the right half plane $\set{z:\ \re z \ge 0}$,
and in the first quadrant  if $K(x) \ge 0$, by Proposition~\ref{discretespec}.
We obtain from Theorem~\ref{bracketthm} that the $L^2$ operator norm of the resolvent
grows faster than any power of $h$ as $h \to 0$, thus~\eqref{pspec} holds
for almost all values $z$ such that $\re z > 0$ and 
$\im z$ is an eigenvalue of $K$, see Example~\ref{Qex1}. 

\noindent
For $\re z = 0$
and almost all eigenvalues $\im z$ of $K$,
we find from Theorem~\ref{subthm} that the norm of the resolvent is bounded by
$Ch^{-2/3}$, see Example~\ref{Qex3}. 
In the case $K(x) \ge 0$ and $K(x)$ is invertible at infinity,
we find from Theorem~\ref{QSthm} 
that the norm of the resolvent is 
bounded by $Ch^{-1}$ for  $\re z > 0$ and $\im z = 0$ by Example~\ref{Qex2a}.
The results in this paper are formulated for operators acting on the
trivial bundle over $\br^n$. But since our results are mainly local, they can be applied to
operators on sections of fiber bundles.

\section{The Definitions}

We shall consider  $N\times N$ systems of semiclassical
pseudo-differential operators, and use the Weyl quantization: 
\begin{equation}\label{weyl}
P ^w ( x , h D_x ) u  = \frac{1}{ (2 \pi )^n } 
\iint_{T^*\br^n}  P \left ( \frac{ x+ y } 2 , h \xi \right) e^{ {i} 
\langle x - y , \xi \rangle } u ( y ) d y d \xi 
\end{equation}
for matrix valued $P \in C^\infty ( T^* \br^n, \Cal L(\bc^N,
\bc^N )) $.
We shall also consider the semiclassical operators 
\begin{equation}\label{Pdef}
 P ( h ) \sim \sum_{ j
  =0}^\infty h ^j P_j^w ( x , h D )
\end{equation} 
with $P_j \in C^\infty_\rb (
T^* \br^n, \Cal L(\bc^N, \bc^N ))$. Here $C^\infty_\rb$ is the set of
$C^\infty$ functions having all derivatives in $L^\infty$ and
$P_0 = {\sigma}(P(h))$ is the principal symbol  of $ P ( h )$.
The operator is said to be  elliptic if the principal symbol $P_0$ is
invertible, and of principal type if $P_0$ vanishes of first order on
the kernel, see Definition~\ref{princtype}.
Since the results in the paper only depend on the principal symbol, one could also
have used the Kohn-Nirenberg quantization because the different quantizations only
differ in the lower order terms. 
We shall also consider operators with analytic symbols, then we shall assume
that  $P_j(w)$ are bounded and holomorphic in a tubular neighborhood of
$T^*\br^n$ satisfying
\begin{equation}\label{analsymb}
 \mn{P_j(z,{\zeta})} \le C_0C^jj^j\qquad| \im (z,{\zeta}) | \le 1/C\quad \forall\, j \ge 0
\end{equation}
which will give exponentially small errors in the calculus, here $\mn
A$ is the norm of the matrix $A$. 
But the results hold for more general analytic symbols, see
Remarks~\ref{intrem} and~\ref{QSrem}. In the following, we shall
use the notation $w = (x,{\xi}) \in T^* \br^n $.

We shall consider the spectrum $\spec P(h)$ which is the set of values
${\lambda}$ such that the resolvent
$
 (P(h) - {\lambda}\id_N)^{-1}
$
is a bounded operator, here $\id_N$ is the identity in $\bc^N$. The
spectrum of $P(h)$ is essentially
contained in the spectrum of the principal symbol $P(w)$, which is
given by
\begin{equation*}
 |P(w) - {\lambda}\id_N| = 0
\end{equation*}
where $|A|$ is the determinant of the matrix $A$. 
For example, if $P(w) = {\sigma}(P(h))$ is bounded and $z_1$ is not an eigenvalue of $P(w)$ for any
$w =(x,{\xi})$ (or a limit eigenvalue at infinity) then $P(h) - z_1\id_N$ is
invertible by Proposition~\ref{discretespec}. When $P(w)$ is an
unbounded symbol one needs additional conditions, see for example Proposition~\ref{reduxlem}.
We shall mostly restrict our study to bounded symbols, but we can
reduce to this case if $P(h) - z_1\id_N$ is invertible,
by considering 
$$ (P(h) - z_1\id_N)^{-1} (P(h) - z_2\id_N) \qquad z_2
\neq z_1 $$ 
see Remark~\ref{reduxrem}. But unless we have 
conditions on the eigenvalues at infinity, this does not always give a bounded operator.

\begin{exe}\label{infex}
Let 
\begin{equation*}
 P(\xi ) = 
\begin{pmatrix}
0 & \xi \\
0 & 0 
\end{pmatrix} \qquad {\xi} \in \br
\end{equation*}
then $0$ is the only eigenvalue of $P({\xi})$ but 
\begin{equation}\label{res}
(P(\xi ) - z \id_N)^{-1} = -{1}/{z}
\begin{pmatrix}
1 & {\xi}/z \\ 0 & 1 
\end{pmatrix}
\end{equation} 
and $(P^w - z\id_N)^{-1}P^{w} = -z^{-1} P^w$ is
unbounded for any $z \ne 0$.
\end{exe}

\begin{defn}  Let  $P \in C^\infty ( T^* \br^n)$ be an $N \times N$ system.
We denote the closure of the set of eigenvalues of $P$ by
\begin{equation}
{\Sigma}(P) = \overline {\set{\lambda \in \bc:  \exists\, w\in
    T^*\br^n,\  |P(w) - {\lambda}\id_N| = 0}}
\end{equation}
and the eigenvalues at infinity:
\begin{equation} 
  {\Sigma}_\infty(P) = \set{{\lambda} \in \bc: \exists\, w_j \to
   \infty \ \exists \,u_j \in \bc^N\setminus 0; \ |P(w_j)u_j-
   {\lambda}u_j|/|u_j|  \to 0, \ j \to \infty}
\end{equation}
which is closed in $\bc$.
\end{defn}

In fact, that ${\Sigma}_\infty(P)$ is closed follows by taking a
suitable diagonal sequence. Observe 
that as in the scalar case, we could have ${\Sigma}_\infty(P) =
{\Sigma}(P)$, for example if ~$P(w)$ is constant in one direction.  It
follows from the definition that ${\lambda} \notin
{\Sigma}_\infty(P)$ if and only if the resolvent is defined and
bounded when~ $|w|$ is large enough:
\begin{equation}\label{eqdef}
 \mn{(P(w) - {\lambda}\id_N)^{-1}} \le C\qquad |w| \gg 1
\end{equation}
In fact, if ~\eqref{eqdef} does not hold there would exist $w_j \to
\infty$ such that $\mn{(P(w_j) - {\lambda}\id_N)^{-1}} \to \infty$, $j
\to \infty$. Thus, there would exist $u_j \in \bc^N$ such that $|u_j| = 1$ and $P(w_j) u_j
- {\lambda} u_j \to 0$. On the contrary, if ~\eqref{eqdef} holds then 
$|P(w)u - {\lambda}u| \ge |u|/C$ for any~ $u \in \bc^N$ and $|w| \gg 1$.

It is clear from the definition that ${\Sigma}_\infty(P)$
contains all finite limits of eigenvalues of $P$ at infinity. In fact,
if $P(w_j)u_j = {\lambda}_ju_j$, $|u_j| = 1$, $w_j \to \infty$ and
${\lambda}_j \to {\lambda}$ then 
$$P(w_j)u_j - {\lambda}u_j =
({\lambda}_j - {\lambda})u_j \to 0$$
Example~\ref{infex} shows that in general ~${\Sigma}_\infty(P)$ could be
a larger set.

\begin{exe}\label{limeigex}
Let $P({\xi})$ be given by Example~\ref{infex}, then ${\Sigma}(P) =
\set{0}$ but ${\Sigma}_\infty(P) = \bc$ by~\eqref{res}
and~\eqref{eqdef}. In fact, for any ${\lambda} \in
\bc$ we find that 
$$|P({\xi})u_{\xi}- {\lambda}u_{\xi}| =
{\lambda}^2\quad\text{when}\quad u_{\xi} = {}^t({\xi},{\lambda})$$ 
We have that $|u_{\xi}| = \sqrt{|{\lambda}|^2 + {\xi}^2} \to \infty$ so
$|P({\xi})u_{\xi}- {\lambda}u_{\xi}|/|u_{\xi}|  \to 0$ when $|{\xi}| \to \infty$.
\end{exe}

For bounded symbols we get equality according to the 
following proposition.

\begin{prop}
If $P \in C^\infty_{\rm{b}} ( T^* \br^n)$ is an $N \times N$ system
then ${\Sigma}_\infty(P)$ is the set of all limits of  
the eigenvalues of $P$ at infinity.
\end{prop}

\begin{proof}
Since  ${\Sigma}_\infty(P)$ contains all
limits of eigenvalues of $P$ at infinity, we only have to 
prove the opposite inclusion.
Let ${\lambda} \in {\Sigma}_\infty(P)$ then by the
definition there exist $w_j \to \infty$
and $u_j \in \bc^N$ such that $|u_j| = 1$ and $|P(w_j)u_j -
{\lambda}u_j| = {\varepsilon}_j \to 0$. Then we may choose $N \times
N$ matrix $A_j$ such that $\mn{A_j} = {\varepsilon}_j$ and 
$P(w_j)u_j = {\lambda}u_j + A_ju_j$ thus ${\lambda}$
is an eigenvalue of $P(w_j)- A_j$. Now if~ $A$ and ~$B$ are $N\times N$
matrices and $d(\eig(A), \eig(B))$ is the minimal distance between the
sets of eigenvalues of $A$ and ~$B$ under permutations, then we have
that $ d(\eig(A), \eig(B)) \to 0$ when $\mn{A-B} \to 0$. In fact, a
theorem of Elsner  ~\cite{elsner} gives
\begin{equation}
 d(\eig(A), \eig(B)) \le N(2\max(\mn A,\mn B))^{1-1/N}\mn{A-B}^{1/N}
\end{equation}
Since the matrices $P(w_j)$ are uniformly bounded
we find that they have an eigenvalue ${\mu}_j$ such
that $|{\mu}_j - {\lambda}| \le C_N{\varepsilon}_j^{1/N} \to 0$ as $j
\to \infty$, thus
${\lambda} = \lim_{j \to \infty}
{\mu}_j$ is a limit of eigenvalues of $P(w)$ at infinity.  
\end{proof}

One problem with studying systems $P(w)$, is that the eigenvalues are
not very regular in the parameter $w$, generally they depend only
continuously (and eigenvectors measurably) on~$w$.

\begin{defn}\label{Omegadef} 
For  an $N \times N$ system $P \in C^\infty ( T^* \br^n)$ we define 
\begin{equation*}
 {\kappa}_P(w,{\lambda}) = \dim \ker (P(w) - {\lambda}\id_N)
\end{equation*}
and 
\begin{equation*}
 K_P(w,{\lambda}) = \max\set{k:\ \partial_{\lambda}^j
   p(w,{\lambda}) = 0 \text{ for}\ j < k}
\end{equation*}
where $ p(w,{\lambda}) = |P(w) - {\lambda}\id_N|$ is the
characteristic polynomial. We have ${\kappa}_P \le K_P$ with equality
for symmetric systems but in general we need not
have equality, see Example~\ref{varcharex}. Let
\begin{equation*}
 {\Omega}_k(P) = \set{(w,{\lambda}) \in T^*\br^n\times \bc:\
   K_P(w,{\lambda}) \ge k}\qquad k \ge 1
\end{equation*}
then $\emptyset = {\Omega}_{N+1}(P) \subseteq {\Omega}_{N}(P)
\subseteq \dots \subseteq {\Omega}_1(P)$ and
we may define
\begin{equation}
\Xi(P) = \bigcup_{j > 1} \partial {\Omega}_j(P)
\end{equation}
where  $\partial {\Omega}_j(P)$ is the
boundary of ~${\Omega}_j(P)$ in the relative topology of ${\Omega}_1(P)$.
\end{defn}

Clearly, ${\Omega}_j(P)$ is a closed set for any $j \ge 1$. By the
definition we find that the multiplicity $K_P$ of the 
zeros of $|P(w)-{\lambda}\id_N|$ is locally constant on ${\Omega}_1(P)
\setminus \Xi(P)$. If $P(w)$ is symmetric then ${\kappa}_P = \dim \ker (P(w) -
{\lambda}\id_N)$ also is constant on ${\Omega}_1(P)\setminus
{\Xi}(P)$ but this is not true in general, see Example~\ref{princex}.

\begin{rem}  \label{constcharlem}
We find that $\Xi(P) $ 
is closed and nowhere dense in ${\Omega}_1(P)$ since it is the union of
boundaries of closed sets. 
We also find that 
$$(w,{\lambda}) \in {\Xi}(P) \iff (w,\ol {\lambda}) \in {\Xi}(P^*)$$
since $|P^* - \ol {\lambda}\id_N| = \ol{|P - {\lambda}\id_N|}$.
\end{rem}

\begin{exe}\label{varcharex}
Let 
\begin{equation*}
 P(w) = 
\begin{pmatrix}
{\lambda}_1(w) & 1 \\ 0 & {\lambda}_2(w) 
\end{pmatrix}
\end{equation*}
where ${\lambda}_j(w) \in C^\infty$, $j=1$, $2$,
then ${\Omega}_1(P)  = \set{(w,{\lambda}):\ {\lambda} = {\lambda}_j(w),
  \ j= 1,\, 2}$, 
$${\Omega}_2(P) =  
\set{(w,{\lambda}):\ {\lambda} = {\lambda}_1(w)  =
  {\lambda}_2(w)}$$ but ${\kappa}_P \equiv 1$ on ${\Omega}_1(P)$.
\end{exe}

\begin{exe}\label{constgeocharex}
Let
$$P(t) = 
\begin{pmatrix}
0 & 1 \\ t & 0 
\end{pmatrix} \qquad t\in \br$$
then $P(t)$ has the eigenvalues $\pm \sqrt{t}$, and ${\kappa}_P
\equiv 1$ on ${\Omega}_1(P)$. 
\end{exe}

\begin{exe}\label{simpleex}
Let 
\begin{equation*}
 P = 
\begin{pmatrix}
w_1+ w_2 & w_3 \\ w_3 & w_1 - w_2 
\end{pmatrix}
\end{equation*}
then 
\begin{equation*}
 {\Omega}_1(P) = \set{(w; {\lambda}_j):\ {\lambda}_j = w_1 + (-1)^j \sqrt{w_2^2
 + w_3^2}, \ j=1,\ 2}
\end{equation*}
We have that ${\Omega}_2(P) = \set{(w_1,0,0; w_1):\ w_1 \in \br}$ and
${\kappa}_P = 2$ on ${\Omega}_2(P)$.  
\end{exe}

\begin{defn} 
Let ${\pi}_j$ be the projections
\begin{equation*}
 {\pi}_1(w,{\lambda}) = w\qquad {\pi}_2(w,{\lambda}) = {\lambda}
\end{equation*}
then we  define for ${\lambda}\in \bc$ the closed sets
\begin{equation*}
 {\Sigma}_{\lambda}(P) = {\pi}_1 \left({\Omega}_1(P)\bigcap
   {\pi}^{-1}_2({\lambda})\right) = \set{w:\ |P(w) - {\lambda}\id_N| =
   0}
\end{equation*}
and 
$${X}(P) = {\pi}_1\left(\Xi(P)\right) \subseteq T^*\br^n$$ 
\end{defn}

\begin{rem}
Observe that $X(P)$ is nowhere dense in $T^*\br^n$ and
$P(w)$ has  constant characteristics
near~$w_0 \notin {X}(P)$. This means that $|P(w) -
{\lambda}\id_N| = 0$ if and only if ${\lambda} = {\lambda}_j(w)$ for
$j = 1, \dots k$, where the  eigenvalues ${\lambda}_j(w) \ne {\lambda}_k(w)$ for $j
\ne k$ when $|w -w_0| \ll 1$. 
\end{rem}

In fact, ${\pi}_1^{-1}(w)$ is a finite set for any $w \in T^*\br^n$ and
since the eigenvalues are continuous functions of the parameters, the relative topology
on ${\Omega}_1(P)$ is generated by ${\pi}_1^{-1}({\omega})\bigcap {\Omega}_1(P)$ for open
sets ${\omega} \subset T^*\br^n$.

\begin{defn}  For an $N \times N$ system $P \in C^\infty ( T^* \br^n)
  $ we define the {\em weakly singular eigenvalue set} 
\begin{equation}
{\Sigma}_{ws}(P)
= {\pi}_2\left({\Xi}(P)\right) \subseteq \bc
\end{equation}
and the  {\em strongly singular eigenvalue set}
\begin{equation}
{\Sigma}_{ss}(P) 
= \set{{\lambda}:\
 {\pi}_2^{-1}({\lambda})\bigcap {\Omega}_1(P) \subseteq {\Xi}(P)}.
\end{equation}
\end{defn}

\begin{rem}\label{closedrem}
It is clear from the definition that ${\Sigma}_{ss}(P) \subseteq {\Sigma}_{ws}(P)$. 
We have that ${\Sigma}_{ws}(P)\bigcup {\Sigma}_\infty(P)$ and
${\Sigma}_{ss}(P)\bigcup {\Sigma}_\infty(P)$ are closed, and
${\Sigma}_{ss}(P)$ is nowhere dense.
\end{rem}

In fact, if ${\lambda}_j \to {\lambda} \notin
{\Sigma}_\infty(P)$, then ${\pi}_2^{-1}({\lambda}_j)\bigcap {\Omega}_1(P)$ is contained in a
compact set for $j \gg 1$, which then either intersects ${\Xi}(P)$ or
is contained in ${\Xi}(P)$. Since ${\Xi}(P)$ is closed, these
properties are preserved in the limit.
 
Also, if  ${\lambda} \in \ol{{\Sigma}_{ss}(P)}$, then there exists
$(w_j,{\lambda}_j) \in {\Xi}(P)$ such that ${\lambda}_j \to
{\lambda}$ as $j \to \infty$. Since ${\Xi}(P)$ is nowhere dense in $ {\Omega}_1(P)$, there exists
$(w_{jk},{\lambda}_{jk}) \in {\Omega}_1(P)\setminus {\Xi}(P)$
converging to $(w_j,{\lambda}_j) $ as $k \to \infty$. Then
${\Sigma}(P) \setminus {\Sigma}_{ss}(P) \ni
{\lambda}_{jj} \to {\lambda}$, so ${\Sigma}_{ss}(P)$ is nowhere dense.
On the other hand, it is possible that ${\Sigma}_{ws}(P) =
{\Sigma}(P)$ by the following example.

\begin{exe}\label{singex}
Let $P(w)$ be the system in Example~\ref{simpleex}, then we have 
$${\Sigma}_{ws}(P)= {\Sigma}(P) = \br$$ 
and ${\Sigma}_{ss}(P) = \emptyset$.  In fact, the eigenvalues
coincide only when $w_2 = w_3= 0$ and the eigenvalue ${\lambda} = w_1$ is
also attained at some point where $w_2 \ne 0$.  
If we multiply $P(w)$ with $w_4 + i w_5$, we obtain that
${\Sigma}_{ws}(P) = {\Sigma}(P) = \bc$.
If we set $\wt P(w_1,w_2) = P(0, w_1,w_2)$ we find that
${\Sigma}_{ss}(\wt P)= {\Sigma}_{ws}(\wt P) = \set 0$.
\end{exe}

\begin{lem} \label{evlem}
Let $P \in C^\infty ( T^* \br^n)$ be an $N \times N$ system.
If $(w_0,{\lambda_0}) \in {\Omega}_1(P)\setminus {\Xi}(P)$ then there
exists a unique $C^\infty$ function ${\lambda}(w)$ so that
$(w,{\lambda}) \in {\Omega}_1(P)$ if and only if ${\lambda} = {\lambda}(w)$ in a
neighborhood of ~$(w_0,{\lambda}_0)$.
If ${\lambda}_0  \in {\Sigma}(P)\setminus({\Sigma}_{ws}(P)\bigcup
{\Sigma}_\infty(P))$ then $\exists\ {\lambda}(w) \in C^\infty$ such that
$(w,{\lambda}) \in {\Omega}_1(P)$ if and only if ${\lambda} = {\lambda}(w)$
in  a neighborhood of ~${\Sigma}_{\lambda_0}(P)$.
\end{lem}

We find from Lemma~\ref{evlem} that
${\Omega}_1(P)\setminus{\Xi}(P)$ is locally given as a $C^\infty$ manifold over~
$T^*\br^n$, and that the eigenvalues ${\lambda}_j(w)\in C^\infty$
outside $X(P)$. This is not true if we instead assume that ${\kappa}_P$
is constant on ${\Omega}_1(P)$, see Example~\ref{constgeocharex}.

\begin{proof} 
If $(w_0,{\lambda}_0) \in {\Omega}_1(P)\setminus \Xi(P)$, then 
\begin{equation*}
 {\lambda} \to |P(w)- {\lambda} \id_N|
\end{equation*}
vanishes of exactly order $k > 0$ on ${\Omega}_1(P)$
in a neighborhood of~$(w_0,{\lambda}_0)$, so
$$\partial_{\lambda}^{k}|P(w_0) -
{\lambda}\id_N| \ne 0 \qquad\text{for } {\lambda} = {\lambda}_0$$ 
Thus ${\lambda} = {\lambda}(w)$ is
the unique solution to $\partial_{\lambda}^{k-1}|P(w) -
{\lambda}\id_N| = 0$ near $w_0$ which is $C^\infty$ by the Implicit
Function Theorem.

If ${\lambda_0} \in {\Sigma}(P)\setminus({\Sigma}_{ws}(P)\bigcup {\Sigma}_\infty(P))$ then
we obtain this in a neighborhood of any ~$w_0 \in {\Sigma}_{\lambda_0}(P)
 \Subset T^*\br^n$.
By using a $C^\infty$ partition of unity 
we find by the uniqueness that  ${\lambda}(w) \in C^\infty$ in a
neighborhood of ${\Sigma}_{\lambda_0}(P)$.
\end{proof}

\begin{rem}\label{sardrem} 
Observe that if ${\lambda}_0 
\in {\Sigma}(P) \setminus ({\Sigma}_{ws}(P)\bigcup {\Sigma}_\infty(P))$
and ${\lambda}(w) \in C^\infty$ satisfies $|P(w) - {\lambda}(w)\id_N|
\equiv 0$ near ~${\Sigma}_{{\lambda}_0}(P)$ and
${\lambda}\restr{{\Sigma}_{{\lambda}_0}(P)} = {\lambda}_0$, then we find 
by the Sard Theorem that $d\re {\lambda}$
and $d\im {\lambda}$ are linearly independent on the codimension~ $2$
manifold ${\Sigma}_{\mu}(P)$ 
for almost all values ${\mu}$ close to ${\lambda}_0$.
Thus for $n=1$ we find that ${\Sigma}_{\mu}(P)$ is a discrete set 
for almost all values ${\mu}$ close to ${\lambda}_0$. 
\end{rem}

In fact, since ${\lambda}_0 \notin {\Sigma}_\infty(P)$ we find that
${\Sigma}_{\mu}(P) \to {\Sigma}_{\lambda_0}(P)$ when ${\mu} \to
{\lambda}_0$ so ${\Sigma}_{\mu}(P) = \set{w:\ {\lambda}(w) = {\mu}}$
for $|{\mu}-{\lambda}_0| \ll 1$.

\begin{defn}
A $C^\infty$ function ${\lambda}(w)$ is called a {\em germ of
  eigenvalues} at $w_0$ for the $N \times N$ system ~$P \in C^\infty ( T^* \br^n)$ if 
\begin{equation}
 |P(w) - {\lambda}(w)\id_N| \equiv 0 \qquad\text{in a neighborhood of $w_0$}
\end{equation}
If this holds in a neighborhood of every point in ${\omega} \Subset
T^*\br^n$ then we say that ${\lambda}(w)$ is a germ of eigenvalues for ~$P$ on ${\omega}$.
\end{defn}

\begin{rem}
If ${\lambda}_0 \in {\Sigma}(P)\setminus({\Sigma}_{ss}(P)\bigcup
{\Sigma}_\infty(P))$ then there exists $w_0 \in
{\Sigma}_{\lambda_0}(P)$ so that $(w_0,{\lambda}_0) \in {\Omega}_1(P)
\setminus {\Xi}(P)$. By Lemma~\ref{evlem} there exists a
$C^\infty$ germ ${\lambda}(w) $ of eigenvalues
at ~$w_0$ for ~$P$ such that ${\lambda}(w_0) =
{\lambda}_0$. If ${\lambda}_0 \in {\Sigma}(P)\setminus({\Sigma}_{ws}(P)\bigcup
{\Sigma}_\infty(P))$ then there exists a $C^\infty$ germ ${\lambda}(w) $ of eigenvalues 
on~ ${\Sigma}_{\lambda_0}(P)$.
\end{rem}

As in the scalar case we obtain that
the spectrum is essentially discrete outside ${\Sigma}_\infty(P)$.

\begin{prop}\label{discretespec}
\label{p:3}
Assume that the $N \times N$ system $P(h)$ is given by ~\eqref{Pdef} with principal symbol
$P \in C^\infty_\rb(T^*\br^n)$. Let $ \Omega $ be an open connected set, satisfying
\[ \ol \Omega \bigcap \Sigma_\infty ( P ) = \emptyset \quad
\text{and} \quad
 \Omega \bigcap  \complement \Sigma ( P ) \neq \emptyset \,\]
Then 
$ ( P(h) - z\id_N )^{-1} $, $ 0 < h \ll 1 $, $ z \in \Omega $,
is a meromorphic family of operators 
with poles of finite rank.
In particular, for $ h$ sufficiently small, the spectrum of 
$ P(h) $ is discrete in any such set. When ${\Omega}\bigcap
{\Sigma}(P) = \emptyset$ we find that ${\Omega}$ contains no spectrum 
of ~$P^w(x,hD)$.
\end{prop}

\begin{proof}
We shall follow the proof of Proposition~3.3 in \cite{dsz}.
If  $ \Omega $ satisfies the assumptions of the proposition then 
there exists $ C > 0$ such that
\begin{equation}\label{resbound}
  | (P (w ) -  z\id_N)^{-1}| \le C \qquad \text{if $ z \in \Omega $ and $| w | > C $ }
\end{equation}
In fact, otherwise there would exist $w_j \to \infty$ and $z_j \in
{\Omega}$ such that $|(P(w_j) - z_j\id_N)^{-1}| \to \infty$, $j \to
\infty$. Thus, $\exists \, u_j \in \bc^N$ such that $|u_j| = 1$ and $P(w_j) u_j
- z_j u_j \to 0$. Since ${\Sigma}(P) \Subset\bc$ we obtain after picking a subsequence that 
$z_j \to z \in \ol {\Omega}\bigcap {\Sigma}_\infty(P) = \emptyset$.
The assumption that $ \Omega \cap \complement \Sigma ( p ) \neq \emptyset$ 
implies that for some $ z_0 \in \Omega $ we have $ ( P ( w) - z_0\id_N )^{-1} \in
C^\infty_{\rb}$. Let $ \chi \in C^\infty_0 ( T^* \br^n) $, $0 \le {\chi}(w) \le 1$ and
${\chi}(w)= 1$ when $|w| \le C$, where $C$ is given by ~\eqref{resbound}. Let
\[  R ( w , z ) = \chi ( w ) (  P ( w) - z_0\id_N )^{-1} 
+ ( 1 - \chi ( w )) (P ( w ) - z\id_N)^{-1} \in  C_\rb^\infty\]
for $z \in {\Omega}$.
The symbolic calculus 
then gives
\[  \begin{split}
& R^w ( x , h D, z ) ( P(h) - z\id_N) = I + hB_1(h,z) + K_1 (h, z ) \\ 
& ( P(h) -  z\id_N ) R^w ( x , h D, z )  = I + hB_2(h,z) +
K_2 (h, z )
\end{split} \]
where $ K_j ( h, z ) $ are
compact operators on $ L^2 ( \br^n ) $ depending holomorphically on
$z$, vanishing for $z=z_0$, and  $B_j(h,z)$ are bounded  on $ L^2 ( \br^n ) $, $ j = 1, 2$.
By the analytic Fredholm theory we conclude that $ ( P(h)-  z\id_N  )^{-1} 
$  is meromorphic in $z \in \Omega $ for $ h $ sufficiently small.
When  ${\Omega}\bigcap
{\Sigma}(P) = \emptyset$ we can choose $R(w, z) = (P(w) -
z\id_N)^{-1}$, then $K_j \equiv 0$ for $j = 1$, 2, and $P(h) - z\id_N$ is 
invertible for small enough ~$h$.
\end{proof}

We shall show how the reduction to the case of bounded operator 
can be done in the systems case, following~\cite{dsz}.
Let $ m ( w  ) $ be
a positive function on $ T^* \br^n  $ satisfying
$$ 1 \leq  m ( w )  \leq C \langle w - w_0 \rangle^N
m (w_0 ) \,, \qquad \forall \; w,\ w_0 \in T^*\br^{n}$$
for some fixed $ C  $ and $ N $, where $\w{w} = 1 + |w|$. Then
$m$ is an admissible weight function and we can define the symbol
classes $P \in S(m)$ by
$$ \mn{\partial_w^\alpha P ( w )} \leq C_\alpha m (w)  \qquad \forall {\alpha}$$ 
Following \cite{DiSj} we can then
define the semiclassical operator $ P(h) = P^w ( x , h D ) $. In the analytic case we 
require that the symbol estimates hold in a tubular neighborhood of
$T^*\br^n$:
\begin{equation}\label{analell}
   \mn{\partial_w^\alpha  P ( w ) } \leq  C_\alpha m ( \Re w )
   \qquad\text{for }\quad
|\im w | \leq 1/C \qquad \forall {\alpha}
\end{equation}
One typical example of an admissible weight
function is $ m ( x , \xi ) = (\langle \xi\rangle^2 + \langle x \rangle^p) $.

Now we make the ellipticity assumption
\begin{equation}\label{hypoell} 
\mn{P^{-1}(w)} \le C_0m^{-1}(w) \qquad  |w| \gg 1
\end{equation} 
and in the analytic case we assume this in a tubular neighborhood of
$T^*\br^n$ as in ~~\eqref{analell}.
By Leibniz' rule we obtain that   $P^{-1}
\in S(m^{-1})$ at infinity, i.e.,
\begin{equation}\label{invest}
  \mn{\partial_w^\alpha P^{-1} ( w ) } \leq C'_\alpha m^{-1}(w) \qquad  |w| \gg 1
\end{equation}
When $ z \not \in {\Sigma}(P)\bigcup {\Sigma}_\infty(P) $ we find as
before that 
$$
\mn {(P(w) - z\id_N)^{-1}} \le C\qquad \forall\, w
$$
since the resolvent is uniformly bounded at infinity.
This implies that $P(w)(P(w) - z\id_N)^{-1}$ and $(P(w) -
z\id_N)^{-1}P(w)$ are bounded. Again by Leibniz' rule, \eqref{hypoell}
holds with $P$ replaced by $P - z\id_N$ thus $(P(w) - z\id_N)^{-1} \in S(m^{-1})$
and we may define the semiclassical operator $((P - z\id_N)^{-1})^w(x,hD)$.
Since $m \ge 1$ we find that $P(w) - z\id_N \in S(m)$, so 
by using the calculus we obtain that
\begin{align*}
&( P^w - z \id_N)  ((P -z\id_N)^{-1})^w = 1 + hR_1^w\\
& ((P -z\id_N)^{-1})^w  ( P^w - z \id_N) = 1 + hR_2^w
\end{align*}
where $R_j \in S(1)$, $j= 1$, 2. For small enough ~~$h$
we get invertibility and the following result.

\begin{prop}\label{reduxlem} 
Assume that $P \in S(m)$ is an $N \times N$ system satisfying ~\eqref{hypoell} and
that $ z \not \in {\Sigma}(P)\bigcup {\Sigma}_\infty(P) $. Then
we find that $ P^w - z \id_N$ is invertible for small enough
~~$h$.
\end{prop}

This makes it possible to reduce to the case of operators with bounded symbols.

\begin{rem}\label{reduxrem} 
If $z_1 \notin \spec(P)$ we may define the operator
\[  Q = ( P - z_1\id_N)^{-1} ( P - z_2\id_N ) \qquad z_2 \neq z_1 \]
then the  
resolvents of $ Q $ and $ P$ are related by 
\[  ( Q - \zeta\id_N)^{-1} =  ( 1 - \zeta)^{-1} ( P - z_1\id_N ) \left( P - 
\frac{ \zeta z_1 - z_2 }{ \zeta - 1 }\id_N \right)^{-1} \qquad {\zeta}
\ne 1\]
when $\frac{ \zeta z_1 - z_2 }{ \zeta - 1 } \notin \spec(P)$.
\end{rem}

\begin{exe}\label{sysex}
Let
\begin{equation*}
P(x,{\xi}) = |{\xi}|^2\id_N + i K(x)
\end{equation*}
where $0 \le K(x)\in C_\rb^\infty$, then we find that $P \in S(m)$
with $m(x,{\xi}) = |{\xi}|^2 + 1$. 
If $0 \notin {\Sigma}_\infty(K)$ then $K(x)$ is invertible for $|x|
\gg 1$, so $P^{-1} \in S(m^{-1})$ at infinity. Since $\re z \ge 0$ in
${\Sigma}(P)$ we find from
Proposition~\ref{reduxlem} that 
$P^w(x,hD) + \id_N$ is invertible for small
enough $h$ and  $P^w(x,hD)(P^w(x,hD) + \id_N)^{-1}$ is bounded in
$L^2$ with principal symbol $P(w)(P(w)+ \id_N)^{-1} \in C^\infty_\rb$.
\end{exe}

In order to measure the singularities of the solutions, we shall introduce
the semiclassical wave front sets.

\begin{defn}
For $u \in L^2(\br^n)$  we say that $w_0 \notin \wf_h(u)$ 
if there exists $a \in C^\infty_0 (T^*\br^n)$ such that
$a(w_0) \ne 0$ and the $L^2$ norm
\begin{equation}\label{hinf}
 \mn{a^w(x,hD)u} \le C_kh^k\qquad \forall\, k
\end{equation}
We call $ \wf_h(u)$ the semiclassical wave front set of $u$. 
\end{defn}

Observe that this definition is equivalent to the definition  
(2.5) in~\cite{dsz} which use the FBI transform ~$T$ given
by~\eqref{fbidef}: $w_0 \notin \wf_h(u)$ if $\mn{Tu(w)} = \Cal
O(h^\infty)$ when ~$|w-w_0| \ll 1$. 
We may also define the {\em analytic}
semiclassical wave front set by the condition that 
 $\mn{Tu(w)} = \Cal O(e^{-c/h})$ in a neighborhood of $w_0$ for some $c > 0$,
see (2.6) in~\cite{dsz}.

Observe that if $u = (u_1, \dots,
u_N) \in L^2(\br^n, \bc^N)$ we may define $\wf_h(u) = \bigcap_j
\wf_h(u_j)$ but this gives no information about which components
of $u$ that are singular. Therefore we shall define the corresponding
vector valued polarization sets.

\begin{defn}
For $u \in L^2(\br^n, \bc^N)$, we say that $(w_0,z_0) \notin
\wf_h^{pol}(u) \subseteq T^*\br^n \times \bc^N$ 
if there exists $A(h)$ given by ~\eqref{Pdef} with principal symbol
$A(w)$ such that
$A(w_0)z_0 \ne 0$ and $A(h)u$ satisfies ~\eqref{hinf}.
We call $ \wf_h^{pol}(u)$ the semiclassical polarization set of $u$. 
\end{defn}

We could similarly define the {\em analytic} semiclassical
polarization set by using the FBI transform and analytic
pseudodifferential operators.

\begin{rem}\label{polrem}
The semiclassical polarization sets are closed, linear in the fiber
and has the functorial properties of the 
 $C^\infty$ polarization sets in~ \cite{de:pol}. In particular,
we find that 
$${\pi}(\wf_h^{pol}(u)\setminus 0) = \wf_h(u) = \bigcup_j \wf_h(u_j)$$ 
if ${\pi}$ is the
projection along the fiber variables: ${\pi}: T^*\br^n \times \bc^N \mapsto
T^*\br^n$. We also find that $$
A(\wf_h^{pol}(u)) = \set{(w,A(w)z):\ (w,z) \in \wf_h^{pol}(u)} \subseteq \wf_h^{pol}(A(h)u)$$
if $A(w)$ is the principal symbol of $A(h)$, which implies that
$ \wf_h^{pol}(Au) = A(\wf_h^{pol}(u))$ when $A(h)$ is elliptic.
\end{rem}

This follows from the proofs of Propositions~~2.5 and~2.7 in \cite{de:pol}.

\begin{exe}\label{polexe}
Let $u = (u_1, \dots,u_N) \in L^2(T^*\br^n,\bc^N)$ where $\wf_h (u_1) =
\set{w_0 }$ and $\wf_h (u_j) = \emptyset$ for $j > 1$.
Then $$\wf_h^{pol} (u) = \set{(w_0, (z,0,\dots)):\ z \in \bc}$$
since $\mn{A^w(x,hD)u} = \Cal O(h^\infty)$ if  $A^wu = \sum_{j > 1}
A_j^wu_j$ and $w_0 \in \wf _h (u)$.
By taking a suitable invertible $E$
we obtain
$$\wf_h^{pol} (Eu) = \set{(w_0, zv):\ z \in \bc}$$
for any $v \in \bc^N$.
\end{exe}

We shall use the
following definitions from ~\cite{p-s},
here and in the following $\mn{P(h)}$ will denote the $L^2$ operator norm of $P(h)$.

\begin{defn}\label{psdef}
Let $P(h)$, $0 < h \le 1$, be a semiclassical family of operators on
$L^2(\br^n)$ with domain~$D$. For
${\mu} > 0$ we define the {\em pseudospectrum of 
  index ${\mu}$} as the set
\begin{equation*}
 {\Lambda}^{\rm sc}_{\mu}(P(h)) = \{z \in \bc:\ \forall\, C>0,\
   \forall\, h_0 > 0, \exists\, 0 < h < h_0, \mn{(P(h) - z\id_N)^{-1}}
   \ge Ch^{-{\mu}}\}
\end{equation*}
and the {\em injectivity pseudospectrum of
  index ${\mu}$} as
\begin{multline*}
 {\lambda}^{\rm sc}_{\mu}(P(h)) = \{z \in \bc:\ \forall\, C>0,\
   \forall\, h_0 > 0,\\ \exists\, 0 < h < h_0,\ \ \exists\,u \in D,\ \mn u = 1, \ \mn{(P(h) - z\id_N)u}
   \le Ch^{{\mu}}\}
\end{multline*}
We define the {\em pseudospectrum of infinite
  index} as $ {\Lambda}^{\rm sc}_\infty(P(h)) =  \bigcap_{\mu}
{\Lambda}^{\rm sc}_{\mu}(P(h))$ and correspondingly the
 {\em injectivity pseudospectrum of infinite
  index}.
\end{defn}

Here we use the convention that $\mn{(P(h) -{\lambda}\id_N)^{-1}} = \infty$ when
${\lambda}$ is in the spectrum~$ \spec(P(h))$.
Observe that we have the obvious inclusion $ {\lambda}^{\rm
  sc}_{\mu}(P(h))\subseteq  {\Lambda}^{\rm sc}_{\mu}(P(h)) $, $\forall\,
{\mu}$. We get equality if, for example, $P(h)$ is Fredholm of index
$\ge 0$.

\section{The Interior Case}

Recall that the scalar symbol $p(x,{\xi})\in C^\infty(T^*\br^n)$ is of {\em principal
  type} if $dp \ne 0$ when $p = 0$. In the
following we let $\partial_{\nu}P(w) = \w{{\nu}, dP(w)}$ for 
$P\in C^1( T^*\br^n)$ and ${\nu} \in T^*\br^n$.
We shall use the following definition  of systems of principal type,
in fact, most of the systems we consider will be of this type.
We shall denote $\ker P$ and $\ran P$ the kernel and range of ~$P$.

\begin{defn} \label{princtype}
The $N \times N$ system $P(w) \in C^\infty(T^*\br^n)$ is of {\em
  principal type} at $w_0$ if
\begin{equation}\label{pr_type}
\ker P(w_0) \ni u \mapsto \partial_{\nu}P(w_0)u \in \coker P(w_0)  = \bc^N/\ran P(w_0)
\end{equation} 
is bijective for some ${\nu} \in T_{w_0}(T^*\br^n)$.
The operator $P(h)$ given by ~\eqref{Pdef} is of principal type if the principal symbol
$P = {\sigma}(P(h))$ is of principal type. 
\end{defn}

\begin{rem}\label{princrem}
If $P(w) \in C^\infty$ is of principal type and $A(w)$, $B(w) \in
C^\infty$ are invertible then  $APB$ is of principal
type. We have that $P(w)$ is of principal type if and only if the adjoint $P^*$
is of principal type. 
\end{rem}

In fact, by Leibniz' rule we have
\begin{equation} \label{dapb}
 \partial (APB) = (\partial A)PB + A(\partial P)B+  AP \partial B
\end{equation}
and $\ran (APB) = A(\ran P)$ and $\ker (APB) =
B^{-1}(\ker P)$ when $A$ and $B$ are invertible, which gives the invariance 
under left and right multiplication. Since $ \ker P^*(w_0) = \ran
P(w_0)^\bot$ we find that $P$ satisfies ~\eqref{pr_type} if and 
only if  
\begin{equation}\label{bilform}
 \ker P(w_0) \times \ker P^*(w_0) \ni (u,v) \mapsto
\w{\partial_{\nu}P(w_0) u,v}
\end{equation}
is a non-degenerate bilinear form. 
Since $\w{\partial_{\nu}P^* v,u} = \ol{\w{\partial_{\nu}P u, v}}$ 
we find that $P^*$ is of principal type if and only if $P$ is.

Observe that if $P$ only has one vanishing eigenvalue ${\lambda}$ (with
multiplicity one) then the condition that $P$ ~is of
principal type reduces to the condition in the scalar case: $d{\lambda} \ne 0$.
In fact, by using the
spectral projection one can find invertible systems~ $A$ and~
$B$ so that
\begin{equation*}
 APB = 
\begin{pmatrix}
{\lambda} & 0\\ 0 & E 
\end{pmatrix}
\end{equation*}
with $E$ invertible $(N-1) \times (N-1)$ system, and this system is
obviously of principal type.

\begin{exe}
Consider the system in Example~\ref{varcharex}
\begin{equation*}
 P(w) = 
\begin{pmatrix}
{\lambda}_1(w) & 1 \\ 0 & {\lambda}_2(w) 
\end{pmatrix}
\end{equation*}
where ${\lambda}_j(w) \in C^\infty$, $j=1$, 2. We find that $
P(w) - {\lambda}\id_2$ is not of principal type 
when ${\lambda} = {\lambda}_1(w) = {\lambda}_2(w)$ since
$\ker  (P(w) - {\lambda}\id_2) =
\ran (P(w) - {\lambda}\id_2) = \bc \times\set{0}$ is preserved by
$\partial P$.
\end{exe}

Observe that the property of being of principal type is not stable
under $C^1$ perturbation,
not even when $P = P^*$ is symmetric, by the following example.

\begin{exe}\label{prtrem} 
The system
\begin{equation*}
 P(w) = 
\begin{pmatrix}
w_1 - w_2 & w_2 \\ w_2 & -w_1 -w_2  
\end{pmatrix}
= P^*(w) \qquad w = (w_1,w_2)
\end{equation*}
is of principal type when $w_1 = w_2 = 0$, but {\em not} of principal
type when $w_2 \ne 0$ and $w_1 = 0$. In fact, 
\begin{equation*}
 \partial_{w_1}P =
\begin{pmatrix} 1 & 0 \\ 0 & -1 
\end{pmatrix}
\end{equation*}
is invertible, and when $w_2 \ne 0$ we have that 
$$\ker P(0,w_2)= \ker  \partial_{w_2}P(0,w_2) =\set{z(1,1):
  z \in \bc}$$ 
which is mapped to\/ $\ran P(0,w_2)  =\set{z(1,-1):  z \in \bc}$ by $\partial_{w_1}P$.
\end{exe}

We shall obtain a simple characterization of systems of principal type.
Recall ${\kappa}_P$, $K_P$ and ${\Xi}(P)$ given by Definition~\ref{Omegadef}.

\begin{prop}\label{princprop}
Assume $P(w) \in C^\infty$ is an $N \times N$ system and 
that $(w_0,{\lambda}_0)\in {\Omega}_1(P)\setminus {\Xi}(P)$, then $P(w) -
{\lambda}_0\id_N$ is of principal type at $w_0$ if and only if
${\kappa}_P \equiv K_P $ at $(w_0,{\lambda}_0)$ and $d{\lambda}(w_0)
\ne 0$ for the $C^\infty$ germ of eigenvalues~${\lambda}(w)$ for $P$ at $w_0$ satisfying
${\lambda}(w_0) = {\lambda}_0$.
\end{prop}

Thus, in the case ${\lambda}_0 = 0 \notin {\Sigma}_{ws}(P)$ we find
that $P(w)$ is of principal type if and only if the germ of eigenvalues
${\lambda}(w)$ is of principal type and we have no non-trivial
Jordan boxes in the normal form. 
Observe that by the proof of
Lemma~\ref{evlem} the $C^\infty$ germ ${\lambda}(w)$ is the unique solution
to $\partial^{k}_{\lambda} 
p(w,{\lambda}) = 0$ for $k = K_P(w,{\lambda}) - 1$ where
$p(w,{\lambda}) = |P(w) - {\lambda}\id_N|$ is the characteristic
equation. Thus we find that 
$d{\lambda}(w) \ne 0$ if and only if $\partial_w\partial^{k}_{\lambda}
p(w,{\lambda}) \ne 0$. For symmetric operators we have ${\kappa}_P
\equiv K_P $ and we only need this condition when $(w_0,{\lambda}_0)
\notin {\Xi}(P)$.

\begin{exe}
The system $P(w)$ in Example~\ref{prtrem} has eigenvalues
$ -w_2 \pm \sqrt{w_1^2 + w_2^2}$ which are equal if and only if $w_1 =
w_2 = 0$, so $\set{0} = {\Sigma}_{ws}(P)$. 
When $w_2 \ne 0$ and $w_1 \approx 0$ the eigenvalue close to zero is $w_1^2/2w_2 + \Cal
O(w_1^4)$ which has vanishing differential at
$w_1 = 0$. The characteristic equation is $p(w,{\lambda}) =
{\lambda}^2 + 2{\lambda}w_2 - w_1^2$, so $d_w p = 0$ when $w_1 =
{\lambda} = 0$. 
\end{exe}

\begin{proof}[Proof of Proposition~\ref{princprop}]
Of course, it is no restriction to assume ${\lambda}_0 = 0$.
First we note that $P(w)$ is of principal type at $w_0$ if and only if
\begin{equation}\label{princtypecond}
 \partial_{\nu}^k|P(w_0)| \ne 0 \qquad k = {\kappa}_P(w_0,0)
\end{equation}
for some ${\nu} \in T(T^*\br^n)$. Observe that $ \partial^j|P(w_0)| =
0$ for $j < k$. In fact, by choosing bases for $\ker
P(w_0)$ and $\im P(w_0)$ respectively, and extending to bases of $\bc^N$, we obtain
matrices $A$ and $B$ so that 
\begin{equation*}
 AP(w)B = 
\begin{pmatrix}
P_{11}(w) & P_{12}(w) \\ P_{21}(w) & P_{22}(w)
\end{pmatrix}
\end{equation*}
where $|P_{22}(w_0)| \ne 0$ and $P_{11}$, $P_{12}$ and $P_{21}$ all vanish at
~$w_0$. By the invariance, $P$ is of principal type if and
only if $\partial_{\nu}P_{11}$ is invertible for some ~${\nu}$, so by
expanding the determinant we obtain~\eqref{princtypecond}.

Since  $(w_0, 0)\in {\Omega}_1(P) \setminus {\Xi}(P)$ we find
from Lemma~\ref{evlem} that we may choose a 
neighborhood ${\omega}$ of $(w_0, 0)$ such that $(w,{\lambda}) \in
{\Omega}_1(P)\bigcap {\omega}$ if and only if ${\lambda} =
{\lambda}(w) \in C^\infty$.  Then
\begin{equation*}
 |P(w) - {\lambda}\id_N| = ({\lambda}(w)-{\lambda})^m e(w,{\lambda})
\end{equation*}
near $w_0$, where  $e(w,{\lambda}) \ne 0$ and $m =
K_P(w_0,0) \ge {\kappa}_P(w_0,0)$. Letting ${\lambda}= 0$ we obtain that 
$ \partial_{\nu}^j|P(w_0)| = 0$ if $j < m$ and $
\partial_{\nu}^m|P(w_0)| = (\partial_{\nu}{\lambda}(w_0))^m e(w_0,0)$.
\end{proof}

\begin{rem}\label{princtyprem}
Proposition~\ref{princprop} shows that for a\/ {\em symmetric} system  the property
to be of principal type is stable outside ${\Xi}(P)$: if the symmetric
system $P(w) -{\lambda}\id_N$ is of principal type at a 
point $(w_0,{\lambda}_0) \notin {\Xi}(P)$ then it is in a
neighborhood. It follows from the Sard Theorem that symmetric systems
$P(w) - {\lambda}\id_N$ are of principal type
almost everywhere on~${\Omega}_1(P)$.
\end{rem}

In fact, for symmetric systems we have ${\kappa}_P \equiv K_P$ and the
differential $d{\lambda} \ne 0$ almost everywhere on ${\Omega}_1(P) \setminus {\Xi}(P)$.
For  $C^\infty$ germs of eigenvalues we can define
the following bracket condition.

\begin{defn}\label{brackdef}
Let $P \in C^\infty ( T^* \br^n) $ be an $N \times N$ system, then we define
$${\Lambda}(P) = \overline{{\Lambda}_-(P) \bigcup {\Lambda}_+(P)}$$
where 
$ {\Lambda}_\pm(P)$ is the set of ${\lambda}_0 \in {\Sigma}(P)$ 
such that there exists $w_0 \in
{\Sigma}_{\lambda_0}(P)$ so that  $(w_0, {\lambda}_0)  \notin
 {\Xi}(P)$ and
\begin{equation}\label{bracketcond}
 \pm \set{\re {\lambda}, \im {\lambda}}(w_0) > 0
\end{equation}
for the unique $C^\infty$ germ ${\lambda}(w)$ of eigenvalues
at~$w_0$ for ~$P$ such that ${\lambda}(w_0) = {\lambda}_0$.
\end{defn} 

Observe that  ${\Lambda}_\pm(P) \bigcap
{\Sigma}_{ss}(P) = \emptyset$, and it follows from  Proposition~\ref{princprop}
that $P(w) - {\lambda}_0 \id_N$ is of principal
type at ~$w_0 \in {\Lambda}_\pm(P)$ if and only if ${\kappa}_P = K_P$ at
~$(w_0,{\lambda}_0)$, since $d{\lambda}(w_0) 
\ne 0$ when ~\eqref{bracketcond} holds. Because of the bracket
condition ~\eqref{bracketcond} we find that ${\Lambda}_\pm(P)$ is
contained in the interior of the values ${\Sigma}(P)$.

\begin{exe} \label{princex}
Let 
\begin{equation*}
 P(x,{\xi}) = 
\begin{pmatrix}
q(x,{\xi}) & {\chi}(x) \\ 0 & q(x,{\xi}) 
\end{pmatrix}\qquad (x,{\xi}) \in T^*\br
\end{equation*}
where $q(x,{\xi})  = {\xi} + ix^2$ and $0 \le {\chi}(x) \in
C^\infty(\br)$ such that ${\chi}(x) = 0$ when $x \le 0$ and ${\chi}(x) >
0$ when $x > 0$. Then ${\Sigma}(P) = 
\set{\im z \ge 0}$, ${\Lambda}_\pm (P) = \set{\im z > 0}$ and
${\Xi}(P) = \emptyset$. For $\im
{\lambda} > 0$ we find ${\Sigma}_{\lambda}(P) = \set{(\pm \sqrt{\im {\lambda}}, \re
  {\lambda})}$ and $P - {\lambda}\id_2$ is of principal
type at ~${\Sigma}_{\lambda}(P)$ only when $x < 0$.
\end{exe}

\begin{thm}\label{bracketthm}
Let  $P \in C^\infty ( T^* \br^n) $ be an  $N
\times N$ system, then 
we have that
\begin{equation}\label{1}
 {\Lambda}(P)\setminus \left({\Sigma}_{ws}(P)\bigcup
{\Sigma}_\infty(P) \right) \subseteq \ol{{\Lambda}_-(P)}
\end{equation}
when $n\ge 2$. Assume that $P(h)$ is given by ~\eqref{Pdef} with
principal symbol ~$P \in C^\infty_\rb(T^*\br^n)$, and that
${\lambda}_0\in {\Lambda}_-(P)$,  $0 \ne u_0
\in \ker (P(w_0)- {\lambda}_0\id_N)$ and $P(w) - {\lambda}\id_N$ is of
principal type on~${\Sigma}_{\lambda}(P)$ near~$w_0$ for $|{\lambda} - {\lambda}_0| \ll 1$,
for the $w_0 \in
{\Sigma}_{\lambda_0}(P)$ in
Definition~\ref{brackdef}. Then there exists $h_0 >  
0$ and $u(h) \in L^2(\br^n)$, $0 < h \le h_0$, so that $\mn {u(h)} \le
1$
\begin{equation}\label{2}
 \mn{(P(h) -{\lambda}_0\id_N)u(h)} \le C_Nh^N\qquad \forall \, N
 \qquad 0 < h \le h_0
\end{equation}
and $\wf^{pol}_h(u(h)) = \set{(w_0,u_0)}$.
There also exists a dense subset of values ${\lambda}_0 \in {\Lambda}(P)$ so that
\begin{equation}\label{2'}
\mn{(P(h) -{\lambda}_0\id_N)^{-1}} \ge C'_Nh^{-N} \qquad \forall \, N
\end{equation}
If all the terms $P_j $ in the expansion~\eqref{Pdef} are analytic satisfying~\eqref{analsymb} 
then $ h^{\pm N}  $ may be replaced by $ \exp (\mp c / h ) $ in ~\eqref{2}--\eqref{2'}.
\end{thm}

Here we use the convention that $\mn{(P(h) -{\lambda}\id_N)^{-1}} = \infty$ when
${\lambda}$ is in the spectrum~$ \spec(P(h))$.
Condition~\eqref{2} means that ${\lambda}_0$ is in the {\em
  injectivity} pseudospectrum ${\lambda}^{\rm sc}_\infty(P)$, and
~\eqref{2'} means that  ${\lambda}_0$ is in the
pseudospectrum ${\Lambda}^{\rm sc}_\infty(P)$.

\begin{rem} \label{intrem}
If $P(h)$ is Fredholm of non-negative index then \eqref{2}  holds for ~${\lambda}_0$
in a dense subset of ${\Lambda}(P)$.
In the analytic case, it follows from the proof that it suffices that $P_j(w)$ is
analytic satisfying ~\eqref{analsymb} in a fixed complex neighborhood
of $w_0 \in {\Sigma}_{\lambda}(P)$, $\forall\, j$.
\end{rem}

In fact, if  $P(h)$ is Fredholm of non-negative index and ${\lambda}_0
\in \spec(P(h))$ then the dimension of $\ker (P(h)-{\lambda}_0\id_N)$
is positive and  \eqref{2} holds.

\begin{exe}\label{Qex1}
Let 
\begin{equation*}
 P(x,{\xi}) = |{\xi}|^2\id + i
 K(x) \qquad (x,{\xi}) \in T^*\br^n
\end{equation*}
where $K(x)\in C^\infty(\br^n)$ is symmetric for all $x$. Then
we find that 
$$\ol {{\Lambda}_-(P)} ={\Lambda}(P) = \set{\re z \ge 0
   \land \im z \in \ol{{\Sigma}(K)\setminus \left({\Sigma}_{ss}(K)\bigcup
   {\Sigma}_\infty(K)\right)}}$$ 
In fact, for any $\im z\in {\Sigma}(K)\setminus
({\Sigma}_{ss}(K)\bigcup {\Sigma}_\infty(K))$ there
 exists a germ of eigenvalues ${\lambda}(x)\in C^\infty({\omega})$ for
 $K(x)$ in an open set ${\omega} \subset \br^n$ so that ${\lambda}(x_0)
 = \im z$ for some $x_0 \in {\omega}$. By Sard's Theorem, we find that almost all
 values of~~ ${\lambda}(x)$ in ~~${\omega}$ are non-singular, and if 
 $d{\lambda} \ne 0$ and $\re z > 0$ we may
choose ${\xi}_0\in \br^n$ so that $|{\xi}_0|^2 = \re z$ and
$\w{{\xi}_0, \partial_x {\lambda}} < 
0$. Then the $C^\infty$ germ of eigenvalues $|{\xi}|^2  +
i{\lambda}(x)$ for $P$ satisfies ~\eqref{bracketcond} at
~$(x_0,{\xi}_0)$ with the minus sign. 
Since $K(x)$ is symmetric, we find that $P(w) - z \id_N$ is of
principal type.
\end{exe}

\begin{proof}[Proof of Theorem \ref{bracketthm}]
First we are going to prove ~\eqref{1} in the case $n \ge 2$.
Let 
$$W = {\Sigma}_{ws}(P)\bigcup {\Sigma}_\infty(P)$$ 
which is a closed set by Remark~\ref{closedrem}, then we find that
every point in $ {\Lambda}(P)\setminus W$ is a limit point of  
$$
\left({\Lambda}_-(P) \bigcup {\Lambda}_+(P)\right)\setminus W  =
({\Lambda}_-(P)\setminus W ) \bigcup ({\Lambda}_+(P)\setminus W )
$$ 
Thus, we only have to show that $ {\lambda}_0 \in    \ol{{\Lambda}_-(P)}$ if
\begin{equation}\label{lcond}
 {\lambda}_0 \in {\Lambda}_+(P)\setminus W = {\Lambda}_+(P)\setminus
 \left({\Sigma}_{ws}(P)\bigcup {\Sigma}_\infty(P) \right)
\end{equation} 
By Lemma ~\ref{evlem} and Remark~\ref{sardrem} we find from
\eqref{lcond} that there exists a $C^\infty$ germ of eigenvalues 
${\lambda}(w) \in C^\infty$ so that 
 ${\Sigma}_{\mu}(P)$ is equal to the level sets $\set{w:\
  {\lambda}(w)={\mu}}$ for $|{\mu} - {\lambda}_0| \ll 1$. By the
definition we find that the  Poisson bracket $\set{\re {\lambda}, \im
  {\lambda}}$ does not vanish identically on
~${\Sigma}_{\lambda_0}(P)$. 
Now by Remark~\ref{sardrem}, $d\re {\lambda}$ and $d\im {\lambda}$ are linearly
independent on ${\Sigma}_{\mu}(P)$ for almost all ${\mu}$ close to
${\lambda}_0$, and then ${\Sigma}_{\mu}(P)$ is
a $C^\infty$  manifold of codimension 2.
By using Lemma~3.1 of \cite{dsz} 
we obtain that $\set{\re {\lambda}, \im {\lambda}}$ changes sign
on ${\Sigma}_{\mu}(P)$ for almost all values ${\mu}$ near
${\lambda}_0$, so we find that those values also are in
${\Lambda}_-(P)$. By taking the closure we obtain ~\eqref{1}.

Next, assume that ${\lambda} \in {\Lambda}_-(P)$, it is no restriction
to assume ${\lambda} = 0$. By the assumptions
there exists $w_0 \in {\Sigma}_0(P)$ and ${\lambda}(w) \in C^\infty$
such that ${\lambda}(w_0) = 0$,
$\set{\re {\lambda}, \im {\lambda}} < 0$ at $w_0$, $(w_0, 0) \notin
{\Xi}(P)$, and  $P(w)- {\lambda}\id_N$ is of principal
type on ${\Sigma}_{\lambda}(P)$ near ~$w_0$ when $|{\lambda}| \ll 1$. Then
Proposition~\ref{princprop} gives that ${\kappa}_P \equiv K_P$ is constant on
${\Omega}_1(P)$ near $(w_0,{\lambda}_0)$, so
\begin{equation}\label{dimkerconst}
 \dim \ker(P(w) - {\lambda}(w)\id_N) \equiv K > 0
\end{equation}
in a neighborhood of $w_0$. 
Since the dimension is constant we can construct a base $\set{u_1(w),
\dots, u_K(w)} \in C^\infty$ for $\ker (P(w) -{\lambda}(w)\id_N)$ in a
neighborhood of $w_0$. By orthonormalizing it and extending to an
orthonormal base of $\bc^N$,
we obtain orthogonal $E(w) \in C^\infty$ so that 
\begin{equation}\label{analprep}
E^*(w)P(w)E(w) = 
\begin{pmatrix}
{\lambda}(w) \id_K & P_{12} \\ 0 & P_{22}
\end{pmatrix} = P_0(w)
\end{equation}
If $P(w)$ is analytic in a tubular neighborhood of $T^*\br^n$ then
$E(w)$ can be chosen analytic in that neighborhood.  
Since $P_0$ is of principal type at ~$w_0$ by Remark~\ref{princrem}
and $\partial P_0(w_0)$ maps $\ker P_0(w_0)$ into itself, we find that 
$\ran P_0(w_0) \bigcap \ker P_0(w_0) = \set{0}$ and thus $|P_{22}(w_0)| \ne
0$. In fact, if there exists $u'' \ne 0 $ 
such that $P_{22}(w_0)u'' = 0$, then by applying $P(w_0)$ on $u =
(0,u'') \notin \ker P_0(w_0)$ we obtain 
$$0 \ne P_0(w_0)u = (P_{12}(w_0)u'',0) \in
\ker P_0(w_0)\bigcap \ran P_0(w_0)$$ 
which gives a contradiction.
Clearly, the norm of the resolvent $P(h)^{-1}$ only changes with a
multiplicative constant under left and right multiplication of $P(h)$
by invertible systems. Now $E^w(x,hD)$ and $(E^*)^w(x,hD)$ are invertible in $L^2$ for
small enough ~$h$, and 
\begin{equation}\label{smoothprep}
 (E^*)^wP(h) E^w = 
\begin{pmatrix}
P_{11} & P_{12} \\ P_{21} & P_{22} 
\end{pmatrix}
\end{equation}
where ${\sigma}(P_{11}) = {\lambda}\id_N$, $P_{21} = \Cal O(h)$ and
$P_{22}(h)$ is invertible for small~$h$. By multiplying from right by 
$$
\begin{pmatrix}
\id_K & 0 \\ -P_{22}(h)^{-1}P_{21}(h) &  \id_{N-K}
\end{pmatrix}
$$
for small~$h$, we obtain that $P_{21}(h) \equiv 0$ and this only
changes lower order terms in $P_{11}(h)$. Then by multiplying from left by 
\begin{equation*}
 \begin{pmatrix}
 \id_K & -P_{12}(h)P_{22}(h)^{-1} \\ 0 &  \id_{N-K}
\end{pmatrix}
\end{equation*}
we obtain that $P_{12}(h) \equiv 0$ without changing $P_{11}(h)$ or $P_{22}(h)$.

Thus, in order to prove~\eqref{2} we may assume $N = K$ and $P(w) =
{\lambda}(w)\id_K$. By conjugating similarly as in the scalar case
(see the proof of Proposition~26.3.1 in  \cite{ho:yellow}), 
we can reduce to the case when $P(h) =
{\lambda}^w(x,hD)\id_K$.     
In fact, let
\begin{equation}\label{normformsys}
P(h) = {\lambda}^w(x,hD)\id_K + \sum_{j \ge 1}^{}
h^j P_j^w(x,hD)
\end{equation}
$A(h) = \sum_{j \ge 0} h^j A_j^w(x,hD)$ and $B(h) = \sum_{j \ge 0}
h^j B_j^w(x,hD)$ with $B_0(w) \equiv A_0(w)$. Then the calculus gives 
$$
P(h)A(h) - B(h){\lambda}^w(x,hD) = \sum_{j \ge 1}h^j E_j^w(x,hD)
$$ 
with 
\begin{equation*}
 E_k = \frac{1}{2i} H_{\lambda}(A_{k-1} + B_{k-1}) + P_1 A_{k-1} + {\lambda}(A_{k} -
 B_{k}) + R_k \qquad k \ge 1
\end{equation*}
Here $H_{\lambda}$ is the Hamilton vector field of ${\lambda}$, $R_k$
only depends on $A_j$ and $B_j$ for $j < k-1$ and $R_1 \equiv 0$. Now we
can choose $A_0$ so that $A_0 = \id_K$ on  $V_0 = \set{w:\ \im
{\lambda}(w)=0}$ and 
$\frac{1}{i}H_{\lambda}A_0 + P_1 A_0$
vanishes of infinite order on~$V_0$ near ~$w_0$. In fact, since
$\set{\re {\lambda}, \im {\lambda}} \ne 
0$ we find $d\im {\lambda}\ne 0$ on $V_0$, and $V_0$ is
non-characteristic for $H_{\re {\lambda}}$. Thus, the equation
determines all derivatives of $A_0$ on $V_0$, and we may use the Borel
Theorem to obtain a solution. Then, by taking 
$$B_1 - A_1 = \left(\frac{1}{i}H_{\lambda}A_0 + P_1 A_0\right){\lambda}^{-1} \in
C^\infty$$ 
we obtain
$E_0 \equiv 0$. Lower order terms are eliminated similarly, by making
$$ \frac{1}{2i} H_{\lambda}(A_{j-1} + B_{j-1}) + P_1 A_{j-1}+ R_j$$ 
vanish of infinite order on~~$V_0$.
Observe that only the difference $A_{j-1} - B_{j-1}$
is determined in the previous step. Thus we can reduce to the case $P =
{\lambda}^w(x,hD)\id$ and then the $C^\infty$ result follows from the scalar
case (see Theorem~1.2 in~\cite{dsz}) by using Remark~~\ref{polrem} and
Example~~\ref{polexe}. 

The analytic case follows as in the proof of
Theorem~1.2${}'$ in~\cite{dsz} by applying a holomorphic WKB construction
to $P = P_{11}$ on the form
\begin{equation*}
 u(z,h) \sim e^{i{\phi}(z)/h}\sum_{j=0}^{\infty} A_j(z)h^j\qquad z = x
 + iy \in \bc^n
\end{equation*}
which will be an approximate solution to $P(h)u(z,h)= 0$.
Here  $P(h)$ is given by~\eqref{Pdef} with $P_0(w) = {\lambda}(w)\id$,
$P_j$ satisfying ~\eqref{analsymb} and $P_j^w(z,hD_z)$  given by the
formula ~\eqref{weyl} where the
integration may be deformed to a suitable chosen contour instead of $T^*\br^n$
(see~\cite[Section 4]{sjo}).
The  holomorphic  phase function ${\phi}(z)$  satisfying
${\lambda}(z,d_z{\phi}) = 0$ is constructed as in
~\cite{dsz} so that $d_z{\phi}(x_0) = {\xi}_0$ and $\im {\phi}(x) \ge
c|x -x_0|^2$, $c > 0 $, and $w_0 = (x_0,{\xi}_0)$. The holomorphic
amplitude $A_0(z)$ satisfies the transport equation   
\begin{equation*}
 \sum_j\partial_{\zeta_j}{\lambda}(z,d_z{\phi}(z))D_{z_j}A_0(z)
 + P_{1}(z,d_z{\phi}(z))    
A_0(z) = 0
\end{equation*}
with $A_0(x_0) \ne 0$.
The lower order terms in the expansion solve
\begin{equation*}
 \sum_j\partial_{\zeta_j}{\lambda}(z,d_z{\phi}(z))D_{z_j}A_k(z) +
 P_{1}(z,d_z{\phi}(z))A_k(z) = S_k(z) 
\end{equation*}
where $S_k(z)$ only depends on $A_j$ and $P_{j+1}$ for $j < k$. As in the scalar
case, we find from ~\eqref{analsymb} that the solutions satisfy
$\mn{A_k(z)} \le C_0C^k k^k$ see Theorem~9.3 in ~\cite{sjo}. By solving up to $k
< c/h$, cutting of near ~$x_0$ and restricting to $\br^n$ we obtain
that $P(h)u = \Cal O(e^{-c/h})$. The details are
left to the reader, see the proof of
Theorem~1.2${}'$ in~\cite{dsz}.

For the last result, we observe that $  \set{\re \ol{\lambda}, \im
  \ol{\lambda}} = - \set{\re {\lambda}, \im 
  {\lambda}} $,
${\lambda} \in {\Sigma}(P) \iff
\ol {\lambda} \in {\Sigma}(P^*)$, $P^*$ is of principal type if and only if $P$ is, and
Remark~\ref{constcharlem} gives $(w,{\lambda}) \in {\Xi}(P) \iff (w,
\ol{\lambda}) \in {\Xi}(P^*)$.
 Thus,  ${\lambda} \in 
{\Lambda}_+(P)$ if and only if  $\ol {\lambda} \in
{\Lambda}_-(P^*)$ and
 $$\mn{(P(h) -{\lambda}\id_N)^{-1}} =
\mn{(P^*(h) - \ol{\lambda}\id_N)^{-1}}$$
From the definition, we find
that any ${\lambda}_0 \in {\Lambda}(P)$ is an accumulation point
of ${\Lambda}_\pm(P)$, so we obtain the result
from~\eqref{2}.
\end{proof}

\begin{rem}
In order to get the estimate~\eqref{2} it suffices that there exists a
semibicharacteristic\/~${\Gamma}$ of ~${\lambda}-{\lambda}_0$ through~$w_0$
such that\/  ${\Gamma}\times\set{{\lambda}_0} \bigcap {\Xi}(P) =
\emptyset$, $P(w) - {\lambda}\id_N$ is of principal type near
${\Gamma}$ for ${\lambda}$ near ${\lambda}_0$ and that 
condition\/~{\rm($\ol{\Psi}$)} is not satisfied on\/~ ${\Gamma}$, see
~\cite[Definition~26.4.6]{ho:yellow}. 
This means that there exists $0 \ne q \in C^\infty$ such that\/ ${\Gamma}$
is a bicharacteristic of\/ ~$\re q({\lambda}-{\lambda}_0)$ through
~$w_0$ and\/~ $\im q({\lambda}-{\lambda}_0)$  changes sign from $+$
to $-$ when going in the positive direction on\/ ~${\Gamma}$.
\end{rem}

In fact, once we have reduced to the normal
form ~\eqref{normformsys}, the construction of approximate local
solutions in the proof of\/~\cite[Theorem~26.4.7]{ho:yellow} can be
adapted to this case, since the principal part is scalar. See also
Theorem~1.3 in \cite[Section~3.2]{ps:thesis} for a similar scalar
semiclassical estimate.

When $P(w)$ is not of principal type, the reduction in the proof of
Theorem~\ref{bracketthm} may not be possible since $P_{22}$ in
~\eqref{analprep} needs not be invertible by the following example.

\begin{exe}
Let 
\begin{equation*}
 P(h) = 
\begin{pmatrix}
{\lambda}^w(x,hD) & 1 \\ h & {\lambda}^w(x,hD)
\end{pmatrix}
\end{equation*}
where ${\lambda}\in C^\infty$ satisfies the bracket
condition~\eqref{bracketcond}.
The principal symbol is
\begin{equation*}
 P(w) = 
\begin{pmatrix}
{\lambda}(w) & 1 \\ 0 & {\lambda}(w)
\end{pmatrix}
\end{equation*}
with eigenvalue ${\lambda}(w)$ and we have
$$\ker (P(w) - {\lambda}(w) \id_2) = \ran (P(w) - {\lambda}(w)
\id_2) = \set{(z,0): \ z \in \bc}\qquad \forall\, w$$ 
We find that $P$ is {\em not} of principal type since $dP =
d{\lambda}\id_2$. Observe that ${\Xi}(P) = \emptyset$ since 
$K_P$ is constant on ${\Omega}_1(P)$.
\end{exe}

When the dimension is equal to one, we have to
add some conditions in order to get the inclusion~\eqref{1}.

\begin{lem}
Let $P(w) \in C^\infty(T^*\br)$ be an $N \times N$ system, then for
every component\/~ ${\Omega}$ of\/ $\bc \setminus ({\Sigma}_{ws}(P)\bigcup
{\Sigma}_\infty(P))$ which has
non-empty intersection with\/ $\complement {\Sigma}(P)$ we find that
\begin{equation}\label{lll}
{\Omega} \subseteq  \ol{{\Lambda}_-(P)}
\end{equation}
\end{lem}

The condition of having non-empty intersection with the complement  
is necessary even in the scalar case, see the remark and Lemma~3.2' on page 394 in~\cite{dsz}.

\begin{proof} 
If ${\mu} \notin {\Sigma}_\infty(P)$ we find that the index
\begin{equation}\label{index}
 i = \var \arg_{{\gamma}}|P(w) - {\mu}\id_N| 
\end{equation}
is well-defined and continuous when ${\gamma}$ is a positively oriented
circle $\set{w:\ |w| = R}$ for $ R \gg 1$. 
If ${\lambda}_0 \notin {\Sigma}_{ws}(P)\bigcup {\Sigma}_\infty(P)$ then we
find from Lemma~\ref{evlem} that the characteristic
polynomial is equal to
\begin{equation*}
 |P(w) - {\mu}\id_N| = ({\lambda}(w) -{\mu})^ke(w,{\mu})
\end{equation*}
near $w_0 \in {\Sigma}_{{\lambda}_0}(P)$, here ${\lambda}$, $e \in
C^\infty$, $e \ne 0$ and $k = K_P(w_0)$. By
Remark~\ref{sardrem} we find   
for almost all~${\mu}$ close to ${\lambda}_0$ that $d\re
{\lambda}\wedge d\im {\lambda}\ne 0$  on 
${\lambda}^{-1}({\mu}) = {\Sigma}_{\mu}(P)$, which is then a finite
set of points on which the Poisson bracket is non-vanishing. If
${\mu} \notin {\Sigma}(P)$ we find that the index~\eqref{index} vanishes, since one
can then let $R \to 0$. Thus, if a component ${\Omega}$ 
of $\bc \setminus ({\Sigma}_{ws}(P)\bigcup {\Sigma}_\infty(P))$ has
non-empty intersection with $\complement {\Sigma}(P)$, we obtain that
$i = 0$ in ~${\Omega}$. When ${\mu}_0\in {\Omega}\bigcap
{\Lambda}(P)$ we find  from the definition that the Poisson bracket
$\set{\re {\lambda}, \im {\lambda}}$  
cannot vanish identically on ${\Sigma}_{\mu}(P)$ for all ~~${\mu}$
close to ${\mu}_0$. Since the index  is
equal to the sum of positive multiples
of the values of the Poisson brackets at ${\Sigma}_{\mu}(P)$, we find that
the bracket must be negative at some point $w_0 \in
{\Sigma}_{\mu}(P)$, for almost all ~${\mu}$ near ~${\lambda}_0$,
which gives~\eqref{lll}.
\end{proof}

\section{The Quasi-Symmetrizable Case}

First we note that if  the system $P(w) - z\id_N$ is of
principal type near ${\Sigma}_z(P)$ for $z$ close to
${\lambda} \in \partial {\Sigma}(P)\setminus 
\left({\Sigma}_{ws}(P)\bigcup{\Sigma}_\infty(P)\right)$ and ${\Sigma}_{\lambda}(P)$
has no closed semibicharacteristics, then one can generalize Theorem~1.3
in~\cite{dsz} to obtain  
\begin{equation}\label{boundaryest}
 \mn{(P(h) - {\lambda}\id_N)^{-1}} \le C/h \qquad h \to 0
\end{equation}
In fact, by using the reduction in the proof of
Theorem~\ref{bracketthm} this follows from the scalar case, see
Example~\ref{evexe1}. But then the eigenvalues close to ${\lambda}$ have
constant multiplicity.

Generically, we have that the eigenvalues of the principal symbol $P$
have constant multiplicity almost everywhere since ${\Xi}(P)$ is nowhere
dense. But at the boundary $\partial
{\Sigma}(P)$ this needs not be the case. For example, if 
\begin{equation*}
 P(t,{\tau}) = {\tau}\id + i K(t)
\end{equation*}
where $C^\infty \ni K \ge 0$ is unbounded and $0 \in {\Sigma}_{ss}(K)$, then $\br =
\partial{\Sigma}(P) \subseteq {\Sigma}_{ss}(P)$.

When the multiplicity of the eigenvalues of the principal symbol is
not constant the situation is more complicated.
Then the following example shows that it is not
sufficient to have conditions only on the eigenvalues
in order to obtain the estimate~\eqref{boundaryest}, not even 
in the principal type case.

\begin{exe} \label{saex}
Let $a_1(t)$, $a_2(t) \in C^\infty(\br)$ be real valued,  $a_2(0) = 0$,
$a_2'(0) > 0$
and let
$$
P^w(t, hD_t) =
\begin{pmatrix}
hD_t + a_1(t)  & a_2(t) -ia_1(t) \\  a_2(t) +ia_1(t) & -hD_t +a_1(t) 
\end{pmatrix} = P^w(t,hD_t)^* 
$$ 
Then the
eigenvalues of $P(t,{\tau})$ are 
$$ {\lambda} = a_1(t) \pm \sqrt{{\tau}^2 +
a_1^2(t) + a_2^2(t)} \in \br$$ 
which coincide if and only if ${\tau} = a_1(t)= a_2(t) = 0$.
We have that
$$
\frac{1}{2}
\begin{pmatrix}
1 & i \\ 1 & -i 
\end{pmatrix}P
\begin{pmatrix}
1 & 1 \\ i & -i 
\end{pmatrix} =
\begin{pmatrix}
hD_t + ia_2(t) & 0 \\ 2a_1(t) & hD_t -i a_2(t) 
\end{pmatrix} = \wt P(h)
$$
Thus we can construct $u_h(t) = {}^t(0, u_2(t))$ so that $\mn {u_h} = 1$ and 
$\wt P(h) u_h = \Cal O(h^N)$ for $ h \to 0$, see Theorem~1.2 in~\cite{dsz}.
When $a_2$ is analytic we may obtain that $\wt P(h) u_h = \Cal
O(\exp(-c/h))$ by Theorem~1.2${}'$ in \cite{dsz}.
By the invariance, we see that $P$ is of principal type at $t = {\tau}
= 0$ if and only if $a_1(0) = 0$.
If $a_1(0) = 0 $ then ${\Sigma}_{ws}(P) = \set 0$ and
when $a_1 \ne 0$ we have that $P^w$ is a selfadjoint diagonalizable
system. 
In the case $a_1(t) \equiv 0$ and $a_2(t) \equiv t$
the eigenvalues of $P(t,hD_t)$ are $\pm \sqrt{2n}h$, $n \in \bn$,
see the proof of Proposition~3.6.1 in~\cite{HeSj}.
\end{exe}

Of course, the problem is that the eigenvalues are not invariant under
multiplication with elliptic systems.
To obtain the estimate~\eqref{boundaryest} for operators that are {\em
not} of principal type, it is not even sufficient that the eigenvalues are real
having constant multiplicity.

\begin{exe}Let $a(t) \in C^\infty(\br)$ be real valued, $a(0) = 0$, $a'(0) > 0$ and
$$
P^w(t,hD_t) =
\begin{pmatrix}
hD_t  & a(t) \\ -ha(t) & hD_t  
\end{pmatrix}
$$
then the
principal symbol is
$P(t,{\tau}) = 
\begin{pmatrix}
{\tau} & a(t) \\ 0 & {\tau} 
\end{pmatrix}$
so the only eigenvalue is ${\tau}$. Thus  ${\Xi}(P) = \emptyset$
but  the principal symbol is not diagonalizable, and  
when $a(t)\ne 0$ the system is not of principal type. We have
$$
\begin{pmatrix}
h^{1/2} & 0 \\ 0 & -1 
\end{pmatrix}P
\begin{pmatrix}
h^{-1/2} & 0 \\ 0 & 1
\end{pmatrix}  = \sqrt{h}
\begin{pmatrix}
\sqrt{h}D_t & a(t)   \\ a(t) & -\sqrt{h}D_t 
\end{pmatrix}
$$
thus we obtain that
$\mn{P^w(t,hD_t)^{-1}} \ge C_N{h}^{-N}$, $\forall \,N$, when $
h \to 0$ by using Example~\ref{saex} with $a_1 \equiv 0$ and
$a_2 \equiv a$. When $a$ is analytic we obtain $\mn{P(t,hD_t)^{-1}} \ge
\exp(c/\sqrt{h})$.
\end{exe}

For non-principal type operators, to obtain the
estimate~\eqref{boundaryest} it is not even sufficient that the
principal symbol has real eigenvalues of multiplicity one.

\begin{exe}
Let $a(t) \in C^\infty(\br)$, $a(0) = 0$, $a'(0) > 0$ and
\begin{equation*}
 P(h) = 
\begin{pmatrix}
1  & hD_t \\ h &  ih a(t)
\end{pmatrix}
\end{equation*}
with principal symbol 
$ \begin{pmatrix}
1 & {\tau} \\ 0 & 0 
\end{pmatrix}$
thus the eigenvalues are $0$ and $1$, so  ${\Xi}(P) = \emptyset$. Since 
\begin{equation*}
\begin{pmatrix}
 1 & 0 \\ -h & 1
\end{pmatrix}
P(h)
\begin{pmatrix}
 1 & -hD_t \\ 0  & 1
\end{pmatrix}
=
\begin{pmatrix}
 1 & 0 \\ 0 & -h
\end{pmatrix}
\begin{pmatrix}
1 & 0 \\ 0 &  hD_t -ia(t)
\end{pmatrix}
\end{equation*}
we obtain as in Example~\ref{saex} that $\mn {P(h)^{-1}} \ge C_N
h^{-N}$ when $h \to 0$, $ \forall\, N$, and for analytic $a(t)$ we obtain
$\mn {P(h)^{-1}} \ge C e^{c/h}$, $h \to 0$ .
Now $\partial_{\tau}P$ maps $\ker P(0)$ into $\ran P(0)$ so the system
is not of principal type. Observe that this property is not preserved
under the multiplications above, since the systems are not elliptic.
\end{exe}

Instead of using properties of the eigenvalues of the principal
symbol, we shall use properties that are invariant under
multiplication with invertible systems. 
First we consider the scalar case, recall that a scalar $p \in C^\infty$ is of {\em principal
  type} if $dp \ne 0$ when $p = 0$. We have the following normal form 
for scalar principal type operators near the boundary~ $ \partial \Sigma (P)$.
Recall that a {\em semibicharacteristic} of~ $p$ is a non-trivial bicharacteristic
of $\re q p$, for $q \ne 0$.

\begin{exe}\label{scalarex0}
Assume that $p(x,{\xi})\in C^\infty(T^*\br^n)$ is of principal type and
$0 \in \partial {\Sigma}(p)\setminus {\Sigma}_\infty(p)$. Then we find
from the proof of Lemma~4.1 in~\cite{dsz}
that there exists $0 \ne q \in C^\infty$ so that
\begin{equation*}
\im q p \ge 0 \qquad d \re qp \ne 0
\end{equation*}
in a neighborhood of $w_0 \in {\Sigma}_0(p)$. In fact,
condition~(1.7) in that lemma is not needed to obtain a local preparation.
By making a symplectic change of
variables and using the Malgrange preparation theorem we then find that
\begin{equation}\label{normform}
  p(x,{\xi}) = e(x,{\xi})({\xi}_1 + if(x,{\xi}'))\qquad {\xi}= ({\xi}_1,{\xi}')
\end{equation}
in a neighborhood of $w_0 \in {\Sigma}_0(p)$, where $e \ne 0$ and $f \ge 0$.
If there are no closed semibicharacteristics of $p$ then we obtain
this in a neighborhood of ${\Sigma}_0(p)$ by a partition of unity.
\end{exe}

This normal form in the scalar case motivates the following definition.

\begin{defn} \label{QS} We say that the $N \times N$ system $P(w) \in
C^\infty(T^*\br^n)$ is {\em quasi-symmetrizable} with respect to the
real $C^\infty$ vector field $V$ in ${\Omega} \subseteq T^*\br^n$ if
$\exists$ $N\times N$ system $ 
M(w) \in C^\infty(T^*\br^n)$ 
so that in ${\Omega}$ we have
\begin{align}
&\re \w{M(VP)u,u} \ge c\mn u^2 - C\mn{Pu}^2 \qquad c >
0  \label{qs1}\\
&\im \w{MPu,u} \ge  - C\mn{Pu}^2\label{qs2}
\end{align}
for any $u  \in \bc^N$,  the system 
$M$ is called a {\em symmetrizer} for $P$. 
\end{defn}

The definition is clearly independent
of the choice of coordinates in $T^*\br^n$ and choice of base in
$\bc^N$.
When $P$ is elliptic, we may take $M = iP^*$ as multiplier, then $P$
is quasi-symmetrizable with respect to any vector field because $\mn{Pu}
\cong \mn{u}$.
Observe that for a {\em fixed} vector field $V$ the set of
multipliers ~~$M$ satisfying~~\eqref{qs1}--\eqref{qs2} is a
convex cone, a sum of two 
multipliers is also a multiplier. Thus, given a vector field ~~$V$
it suffices to make a local choice of multiplier~$M$ and then use a 
partition of unity to get a global one. 

Taylor has studied  {\em symmetrizable} systems
of the type $D_t\id + i K$, for which there exists $R > 0$
making $RK$ symmetric (see Definition~4.3.2 in \cite{ta:pseu}). 
These systems are quasi-symmetrizable with respect to $\partial_{\tau}$ with
symmetrizer~$R$.
We see from Example~\ref{scalarex0} that the scalar symbol $p$ of
principal type is quasi-symmetrizable in neighborhood of any point at
$\partial {\Sigma}(p)\setminus {\Sigma}_\infty(p)$.

The invariance properties of quasi-symmetrizable systems are partly due
to the following simple and probably well-known result on
semibounded matrices. In the following, we shall denote $\re A =
\frac{1}{2}(A + A^*)$ and $i\im A = \frac{1}{2}(A - A^*)$ the symmetric
and antisymmetric parts of the matrix~$A$. Also, if $U$ and $V$ are linear subspaces of
$\bc^N$, then we let $U + V = \set{u + v:\ u \in U\  \land\  v \in V}$.

\begin{lem}\label{semiprop}
Assume that $Q$ is an $N \times N$ matrix such that $\im zQ \ge 0$
for some $0 \ne z \in \bc$. Then we find 
\begin{equation}\label{kereq}
 \ker Q = \ker Q^* = \ker (\re Q) \bigcap \ker (\im Q)
\end{equation}
and $\ran Q = \ran (\re Q) + \ran (\im Q) $ is orthogonal to $\ker Q$.
\end{lem}

\begin{proof}
By multiplying with $z$ we may assume that $\im Q \ge 0$, clearly the
conclusions are invariant under multiplication with complex numbers. If
$u \in \ker Q$, then we have  $ \w{\im Qu,u} = \im\w{Qu,u}= 0$.  By
using the Cauchy-Schwarz inequality on $\im Q \ge 0$ we find that 
$\w{\im Qu,v} = 0$ for any~$v$. Thus $u \in \ker (\im Q)$ so
$\ker Q \subseteq \ker Q^*$. We get equality and ~\eqref{kereq} by the
rank theorem, since $\ker Q^* = \ran Q^\bot$. 

For the last statement we observe that $\ran Q \subseteq \ran (\re Q)
+ \ran (\im Q) = (\ker Q)^{\bot}$ by~\eqref{kereq} where we also
get equality by the rank theorem.
\end{proof}

\begin{prop}\label{princlem}
Assume that $P(w)  \in C^\infty( T^* \br^n)
$ is a quasi-symmetrizable system, then we find that $P$ is of principal type.
Also, the symmetrizer $M$ is invertible if $\im MP \ge cP^*P$ for some $c >0$.
\end{prop}

Observe that by adding $i{\varrho}P^*$ to $M$ we may assume that $Q = MP$ satisfies
\begin{equation} \label{immp}
\im Q \ge
({\varrho}-C)P^*P \ge P^*P \ge c Q^*Q \qquad c > 0
\end{equation} 
for ${\varrho} \ge C+1$, and then the symmetrizer is invertible by
Proposition~\ref{princlem}.

\begin{proof}
Assume that ~\eqref{qs1}--\eqref{qs2} hold at
~$w_0$, $\ker P(w_0) \ne \set{0}$ but~\eqref{pr_type} is not a
bijection. Then there exists $0 \ne u \in \ker P(w_0)$ and $v \in \bc^N$ such that
$V P(w_0) u =
 P(w_0)v$, so  ~\eqref{qs1} gives 
\begin{equation}\label{posconst}
 \re\w{MP(w_0)v, u} = \re\w{M V P(w_0) u, u} \ge c \mn u^2
 > 0.
\end{equation} 
This means that
\begin{equation}\label{contra}
\ran MP(w_0) \not \subseteq \ker P(w_0)^\bot
\end{equation}  
Let $M_{\varrho} = M + i{\varrho}P^*$ then we have that
\begin{equation} \label{pfref}
\im\w{M_{\varrho}Pu,u} \ge ({\varrho} - C)\mn{Pu}^2 
\end{equation}
so for large enough ${\varrho}$ we have $\im M_{\varrho}P \ge 0$. By
Lemma~\ref{semiprop} we find 
\begin{equation} \label{orto}
\ran M_{\varrho}P \bot \ker
M_{\varrho}P
\end{equation}
Since $\ker P \subseteq  \ker M_{\varrho}P$ and
$\ran P^*P \subseteq \ran P^* \bot \ker P$ we find that $\ran
M P \bot \ker P$ for any ~${\varrho}$. This gives a
contradiction to~\eqref{contra}, thus $P$ is of principal type. 

Next, we shall show that $M$ is invertible at ~$w_0$ if
$\im M P \ge cP^*P$ at ~$w_0$ for some $c > 0$. Then we find as before
from Lemma~\ref{semiprop} that $\ran MP(w_0) \bot \ker
MP(w_0)$. By choosing a
base for $\ker P(w_0)$ and completing it to a base of $\bc^N$
we may assume that 
\begin{equation*}
 P(w_0) = 
\begin{pmatrix}
0 & P_{12}(w_0) \\ 0 & P_{22}(w_0) 
\end{pmatrix}
\end{equation*}
where $P_{22}$ is $(N-K) \times (N-K)$ system, $K = \dim \ker P(w_0)$.
Now, by multiplying $P$ from left with an orthogonal matrix $E$ 
we may assume that $P_{12}(w_0) = 0$.
In fact, this only amounts to choosing an orthonormal base for $\ran
P(w_0)^\bot$ and completing to an orthonormal base for $\bc^N$. Observe that
$MP$ is unchanged if we replace $M$ with
$ME^{-1}$, which is invertible if and only if $M$
is. Since $\dim \ker P(w_0) = K$ we obtain $|P_{22}(w_0)| \ne 0$.
Let 
\begin{equation*}
 M = 
\begin{pmatrix}
M_{11} & M_{12} \\ M_{21} & M_{22} 
\end{pmatrix}
\end{equation*}
then we find 
\begin{equation*}
 MP = 
\begin{pmatrix}
0 & 0 \\ 0 & M_{22}P_{22}
\end{pmatrix}\qquad \text{at $w_0$}.
\end{equation*}
In fact, $(MP)_{12}(w_0)= M_{12}(w_0)P_{22}(w_0) = 0$ since
 $\ran MP(w_0) =\ker M P(w_0)^\bot$. We obtain that
 $M_{12}(w_0)= 0$, and by assumption
$$\im M_{22}P_{22}\ge cP^*_{22}P_{22} \qquad\text{at~ $w_0$,}$$ 
which gives $|M_{22}(w_0)| \ne 0$. Since $P_{11}$, $P_{21}$ and
$M_{12}$ vanish at ~$w_0$ we 
find
\begin{equation*}
 \re V(M P)_{11}(w_0) = \re M_{11}(w_0) V P_{11}(w_0) > c
\end{equation*}
which gives $|M_{11}(w_0)| \ne 0$. Since $M_{12}(w_0) = 0$ and
$|M_{22}(w_0)| \ne 0$ we obtain that $M(w_0)$ is invertible.
\end{proof}

\begin{rem}\label{eqrem}
The $N \times N$ system $P\in C^\infty ( T^* \br^n)$ is
quasi-symmetrizable with respect to $V$ 
if and only if there exists an invertible symmetrizer~$M$ such that $Q = MP$ satisfies 
\begin{align}
&\re \w{ (V Q) u,u} \ge c\mn u^2 - C\mn{Qu}^2   \qquad c > 0 \label{qs1b}\\
&\im \w{Q u,u} \ge 0 \label{qs2b}
\end{align}
for any $u \in \bc^N$.
\end{rem}

In fact, by the Cauchy-Schwarz inequality we find
\begin{equation*}
|\w{(VM)Pu,u}| \le {\varepsilon}\mn u^2 +
C_{\varepsilon}\mn{Pu}^2
 \qquad \forall\, {\varepsilon} >0 \quad \forall\, u \in \bc^N
\end{equation*}
Since $M$ is invertible, we also have that $\mn {Pu} \cong \mn{Qu}$.

\begin{defn}
If  the $N \times N$ system $Q\in C^\infty ( T^* \br^n)$ satisfies
~\eqref{qs1b}--\eqref{qs2b} then $Q$ 
is {\em  quasi-symmetric} with respect to the real $C^\infty$ vector field~$V$.
\end{defn}

\begin{prop}\label{invprop}
Let  $P(w) \in C^\infty( T^* \br^n)$ be an
$N \times N$ quasi-symmetrizable system, then $P^*$ is
quasi-symmetrizable. If  $A(w)$ and $B(w) \in 
C^\infty ( T^* \br^n)$ are  invertible $N \times N$ systems then $BPA$ is 
quasi-symmetrizable.
\end{prop}

\begin{proof}
Clearly ~\eqref{qs1b}--\eqref{qs2b} are invariant under
{\em left} multiplication of $P$ with invertible systems~$E$,
just replace~$M$ with $ME^{-1}$.
Since we may write $BPA = B(A^*)^{-1}A^*PA$ it suffices to show that
$E^*PE$ is quasi-symmetrizable if ~$E$ is invertible.
By Remark~\ref{eqrem} there exists a symmetrizer~$M$ so that 
$Q = MP$ is quasi-symmetric, i.e., satisfies ~\eqref{qs1b}--\eqref{qs2b}.
It then follows from Proposition~\ref{invrem0} that
$$Q_E = E^*QE = E^*M(E^*)^{-1}E^*PE$$
also satisfies ~\eqref{qs1b} and~\eqref{qs2b}, so $E^*PE$ is
quasi-symmetrizable. 

Finally, we shall prove that $P^*$ is quasi-symmetrizable if $P$ is.
Since $Q = MP$ is quasi-symmetric, we find from
Proposition~\ref{invrem0} that $Q^* = P^*M^*$ is quasi-symmetric. By
multiplying with $(M^*)^{-1}$ from right, we find from the first part of
the proof that $P^*$ is quasi-symmetrizable.
\end{proof}

\begin{prop}\label{invrem0}
If $Q  \in C^\infty ( T^* \br^n)$ is
quasi-symmetric, then $Q^*$ is quasi-symmetric. If ~$E \in
C^\infty ( T^* \br^n)$ is  
invertible, then~$E^*QE$ are quasi-symmetric. 
\end{prop}

\begin{proof}
First we note that~\eqref{qs1b} holds if and only if 
\begin{equation}\label{qs1a}
 \re \w{(V Q)u,u} \ge c\mn u^2\qquad \forall\,u \in \ker Q
\end{equation}
for some $c >0 $. In fact, $Q^*Q$ has a positive lower bound on  
the orthogonal complement $\ker Q^{\bot}$ so that 
\begin{equation*}
 \mn {u} \le C\mn{Qu} \qquad \text{for $u \in \ker Q^{\bot}$}
\end{equation*}
Thus, if $u = u' + u''$ with $u' \in \ker
Q$ and $u'' \in \ker Q^{\bot}$ we find that $Qu = Qu''$,
\begin{equation*}
 \re \w{(V Q) u',u''} \ge -{\varepsilon}\mn {u'}^2 - C_{\varepsilon}\mn
 {u''}^2 \ge  -{\varepsilon}\mn {u'}^2 - C'_{\varepsilon}\mn {Qu}^2
 \qquad \forall\, {\varepsilon} > 0 
\end{equation*}
and $\re \w{(V  Q)u'',u''} \ge -C \mn{u''}^2 \ge -C'\mn{Qu}^2$. By
choosing ${\varepsilon}$ small enough we obtain ~\eqref{qs1b}
by using ~\eqref{qs1a} on $u'$.

Next, we note that $\im Q^* = -\im Q$ and $\re Q^* = \re Q$, so~$-Q^*$ satifies
~\eqref{qs2b} and~\eqref{qs1a} with $V$ replaced by $-V$, and thus it is quasi-symmetric.
Finally, we shall show that $Q_E = E^*QE$ is quasi-symmetric when $E$
is invertible. We obtain from ~\eqref{qs2b} that 
$$\im \w{Q_E u,u} = \im \w{Q Eu,Eu} \ge 0 
\qquad \forall\ u  \in \bc^N 
$$
Next, we shall show that $Q_E$ satisfies~\eqref{qs1a} on $\ker Q_E = E^{-1}\ker
Q$, which will give~\eqref{qs1b}. We find from
Leibniz' rule that $ V  Q_E = (V  E^*)Q E + E^* (V  Q) E
+ E^*Q (V  E)$ where ~\eqref{qs1a} gives
\begin{equation*}
 \re \w{E^*( V  Q)Eu,u} \ge c\mn{Eu}^2 \ge c'\mn u^2 \qquad u \in
 \ker Q_E
\end{equation*}
since then $Eu \in \ker Q$. Similarly we obtain that $\w{( V  E^*)Q Eu,u} =
0$ when $u \in \ker Q_E$. Now since $\im Q_E \ge 0$ we find from 
Lemma~\ref{semiprop} that
\begin{equation}\label{qstarprop}  
\ker Q^*_E = \ker Q_E
\end{equation}
which gives
$ 
\w{E^*Q ( V  E)u,u} = \w{E^{-1}( V  E)u, Q^*_Eu} =0
$
when $u \in \ker Q_E = \ker Q^*_E$. Thus $Q_E$ satisfies ~\eqref{qs1a}
so it is quasi-symmetric, which finishes the proof.
\end{proof}

\begin{exe}\label{evexe1}
Assume that $P(w)\in C^\infty$ is an $N \times N$ system such that $z \in
{\Sigma}(P)\setminus({\Sigma}_{ws}(P) \bigcap 
{\Sigma}_\infty(P))$ and that $P(w) - {\lambda}\id_N$ is of principal
type when $|{\lambda} - z| \ll 1$. By Lemma~\ref{evlem} and
Proposition~\ref{princprop} there exists a $C^\infty$ germ of eigenvalues 
${\lambda}(w) \in C^\infty$ for ~$P$ so that $\dim \ker (P(w) -
{\lambda}(w)\id_N)$ is constant near ${\Sigma}_z(P)$. By using the
spectral projection as in the proof of Proposition~\ref{princprop} and
making a base change $B(w) \in C^\infty$ we obtain
\begin{equation}\label{4.8a}
 P(w) = B^{-1}(w)
\begin{pmatrix}
{\lambda}(w)\id_K & 0 \\ 0 &   P_{22}(w)
\end{pmatrix}
B(w)
\end{equation}
in a neighborhood of ${\Sigma}_z(P)$, here $|P_{22}-{\lambda}(w)\id| \ne 0$.
We find from Proposition~\ref{princprop} that $d{\lambda} \ne 0$ when
${\lambda} = z$, so ${\lambda}-z$  is of principal type.
Proposition~\ref{invprop} gives that $P - z\id_N$ is
quasi-symmetrizable near any $w_0 \in {\Sigma}_z(P)$ if
$z \in \partial {\Sigma}({\lambda})$. In fact, by Example~\ref{scalarex0}
there exists $q(w) \in
C^\infty$  so that 
\begin{align}\label{4.8}
&|d\re q({\lambda} - z)|\ne  0 \\
& \im q({\lambda}- z) \ge 0 \label{4.9}
\end{align}
and we get the normal form~\eqref{normform} for ${\lambda}$ near
${\Sigma}_z(P) = \set{{\lambda}(w) = z}$.  One can then take  $V$ 
normal to ${\Sigma} = \set{\re q({\lambda}- z) = 0}$ at
${\Sigma}_z(P)$ and use 
$$M = B^*
\begin{pmatrix}
q\id_K & 0 \\ 0 & M_{22} 
\end{pmatrix}B
$$
with $M_{22}(w) = (P_{22}(w)- z\id)^{-1}$ for example, then 
\begin{equation}\label{Qnorm}
 Q = M(P - z\id_N) = B^*
\begin{pmatrix}
q({\lambda}-z)\id_K & 0 \\ 0 & \id_{N-K} 
\end{pmatrix}
B
\end{equation}
If there are no closed semibicharacteristics of ${\lambda}-z$ 
then we also find from Example~\ref{scalarex0} that $P
- z\id_N$ is quasi-symmetriz\-able in a 
neighborhood of ${\Sigma}_z(P)$, see the proof of Lemma~4.1 in ~\cite{dsz}.
\end{exe}

\begin{exe}\label{Qex2}
Let 
\begin{equation*}
P (x,{\xi})= |{\xi}|^2\id_N + i K(x)
\end{equation*}
where $0 \le K(x)\in C^\infty$. 
When  $z > 0$ we find that $P - z\id_N$ is quasi-symmetric in a neighborhood
of ${\Sigma}_{z}(P)$ with respect to the exterior normal $\w{{\xi},
  \partial_{\xi}}$ to ${\Sigma}_z(P) = \set{|{\xi}|^2 = z}$. 
\end{exe}

For scalar symbols, we find that $0 \in
\partial {\Sigma}(p)$ if and only if $p$ is quasi-symmetriz\-able,  see Example~\ref{scalarex0}.
But in the system case, this needs not be the case according to the
following example. 

\begin{exe}
Let
\begin{equation*}
 P(w) = 
\begin{pmatrix}
w_2+ i w_3 & w_1 \\ w_1 & w_2-i w_3  
\end{pmatrix} \qquad w = (w_1,w_2,w_3)
\end{equation*}
which is quasi-symmetrizable with respect to $\partial_{w_1}$
with symmetrizer $M = 
\begin{pmatrix}
0 & 1 \\ 1 & 0 
\end{pmatrix}$.
In fact, $\partial_{w_1}MP = \id_2$ and
\begin{equation*}
 MP(w) = 
\begin{pmatrix}
w_1 & w_2 - i w_3 \\ w_2 + i w_3 & w_1 
\end{pmatrix} =  (MP(w))^*
\end{equation*}
so $\im MP \equiv 0$.
Since eigenvalues of $P(w)$ are  $w_2 \pm \sqrt{w_1^2 - w_3^2}$ we
find that ${\Sigma}(P) = \bc$ so
$0 \in \interior {\Sigma}(P)$ is not a boundary point of the eigenvalues. 
\end{exe}

For quasi-symmetrizable systems we have the following result.

\begin{thm}\label{QSthm}
Let the $N \times N$ system $P(h)$ be given by ~\eqref{Pdef} with principal symbol $P \in
C^\infty_\rb ( T^* \br^n)$.
Assume  that  $ z \notin 
\Sigma_\infty (P ) $ and there exists a real valued\/ {function} 
$T(w) \in C^\infty$ such that $P(w)- {z}\id_N $ is quasi-symmetrizable 
with respect to the Hamilton vector field $H_T(w)$ in a neighborhood of~${\Sigma}_{z}(P)$.
Then  for any $ K > 0$ we have
\begin{equation}\label{thmest0}
  \big\{ {\zeta} \in \bc : \; | {\zeta} - z | < K h \log (1/h) \big\} \bigcap
\spec ( P ( h ) ) = \emptyset 
\end{equation}
 for $0 < h \ll 1$, and 
\begin{equation}\label{thmest}
  \left\| ( P ( h ) - z )^{-1} \right\| \leq C /h\quad 0 < h \ll 1
\end{equation}
If $P$ is analytic  in a tubular neighborhood of $T^*\br^n$ then $\exists\, c_0 > 0 $ 
such that
\begin{equation} 
\big\{ {\zeta} \in \bc  : \; | {\zeta} - z | < c_0 \big\} \bigcap \spec ( P ( h ) ) 
= \emptyset 
\end{equation}
\end{thm}

Condition ~\eqref{thmest} means that ${\lambda}\notin {\Lambda}^{\rm
  sc}_1(P)$, which is the pseudospectrum  of index ~1 by
Definition~\ref{psdef}. The reason for the difference between~\eqref{thmest0} and
~\eqref{thmest} is that we make a change of norm in
the proof that is not uniform in ~$h$. The conditions in
Theorem~\ref{QSthm} give  some geometrical information 
on the bicharacteristic flow of the eigenvalues according to the
following result.

\begin{rem} \label{qshamrem}
The conditions in Theorem~\ref{QSthm} imply that the limit set at
${\Sigma}_{z}(P)$ of the non-trivial semibicharacteristics of the
eigenvalues close to zero of $Q =
M(P - z \id_N)$ is a union of compact curves on which $T$ is
strictly monotone, thus they cannot form closed orbits.
\end{rem}

In fact, locally $(w,{\lambda}) \in {\Omega}_1(P)\setminus {\Xi}(P)$ if and
only if ${\lambda} = {\lambda}(w) \in C^\infty$  by
Lemma~\ref{evlem}. Since $P(w) -{\lambda}\id_N$ is of principal type
by Proposition~\ref{princlem}, we find that $\dim \ker(P(w) -
{\lambda}(w) \id_N)$ is constant by Proposition~\ref{princprop}. Thus
we obtain the normal form  ~\eqref{Qnorm} as in
Example~\ref{evexe1}. This shows that the Hamilton
vector field $H_{\lambda}$ of a germ of an eigenvalue is determined by
$\w{dQ u,u}$ with $0 \ne u \in 
\ker(P - {\lambda}\id_N)$ by the invariance property given by~\eqref{dapb}. Now
we have that $\w{(H_T\re Q)u,u} > 0$ and
$d\w{\im Qu,u} = 0$ for $u \in \ker M(P - z\id_N)$ by ~\eqref{qs2b}.
Thus by picking subsequences we find that the limits of non-trivial
semibicharacteristics for eigenvalues of $Q$ close to $0$ give curves on  
which $T$ is strictly monotone. Since $z \notin {\Sigma}_\infty(P)$
these limit bicharacteristics are compact and cannot form closed orbits.

\begin{exe}\label{Qex2a}
Consider the system in Example~\ref{Qex2}
\begin{equation*}
P(x,{\xi}) = |{\xi}|^2\id_N + i K(x)
\end{equation*}
where $0 \le K(x)\in C^\infty$, then  for $z > 0$ we find that $P - z\id_N$ is
quasi-symmetric in a neighborhood of ${\Sigma}_{z}(P)$ with
respect to $V = H_T$, for $T(x,{\xi}) = -\w{{\xi},x}$. If
$ K(x)\in C^\infty_\rb$ and $0
\notin {\Sigma}_\infty(K)$ then we obtain from Proposition~\ref{reduxlem},
Remark~\ref{reduxrem},
Example~\ref{sysex} and Theorem~\ref{QSthm} that
\begin{equation*}
 \mn{(P^{w}(x,hD) - z)^{-1}} \le C/h \qquad 0 < h \ll 1
\end{equation*}
since $z \notin {\Sigma}_\infty(P)$.
\end{exe}

\begin{proof}[Proof of Theorem~\ref{QSthm}]
We shall first consider the $C^\infty_\rb$ case. We
may assume without loss of generality that  $ z =
0 $, and we shall follow the proof of Theorem ~1.3 in ~\cite{dsz}.
By the conditions, we find from Definition~\ref{QS},
Remark~\ref{eqrem} and~\eqref{immp}
that there exists a function $T(w) \in C_0^\infty$ and a multiplier $M (w)
\in C_\rb^\infty(T^*\br^n)$ so that
$ Q = MP$ satisfies
\begin{align}
&\re H_T Q \ge c -C\im Q\label{9}\\
&\im Q \ge c\, Q^*Q \label{10}
\end{align}
for some $c > 0$ and then $M$ is invertible by Proposition~\ref{princlem}.
In fact, outside a neighborhood of~${\Sigma}_0(P)$ we
have $P^*P \ge c_0$, then we may choose $M = iP^*$ so that $Q
= iP^*P$ and use a partition of unity to get a global multiplier. 
Let
\begin{equation}
\label{4}
{C_1h\le \varepsilon \le C_2h\log {\frac 1 h}}
\end{equation}
where $C_1\gg 1$ will be chosen large. Let $ T = T^w ( x , h D)$
\begin{equation} 
Q(h) = M^w(x,hD)P(h) = Q^w(x, hD) + \Cal O(h)
\end{equation}
and
\begin{equation*}
Q_\varepsilon ( h ) = 
e^{\varepsilon T/h}Q ( h ) e^{-\varepsilon T/h}=e^{{\frac \varepsilon h}{\rm
ad}_T}Q( h )  \sim \sum_{k = 0}^\infty  \frac{\varepsilon ^k} {h^k k!}({\rm
ad}_T)^k(Q ( h ) )
\end{equation*}
where ${\rm ad}_TQ(h) = [ T(h), Q(h)] = \Cal O(h)$. 
By the assumption on $ \varepsilon $ and 
the boundedness of  $ {\rm ad}_T/h $ we find that
the asymptotic expansion makes sense. 
Since ${\varepsilon}^2 = \Cal O(h)$ we see
that the symbol of $Q_\varepsilon (h)$ is equal to 
\begin{equation*}
Q_\varepsilon =Q+i\varepsilon \{T, Q\} +{\mathcal O}(h)
\end{equation*}
Since $T$ is a scalar function, we obtain
\begin{equation}\label{11}
 {\im Q_\varepsilon =\im Q+\varepsilon \Re  H_T Q +{\mathcal O}(h)}
\end{equation}

Now to simplify notation, we drop the parameter $ h$ in the operators
$Q(h)$ and $ P ( h ) $, and we shall use the same letters for
operators and the corresponding symbols.
Using  \eqref{9} and ~\eqref{10} in \eqref{11}, we obtain for small
enough ${\varepsilon}$ that
\begin{equation}
\label{12}
\im Q_\varepsilon  \ge (1 - C{\varepsilon}) \im Q + c{\varepsilon} -
Ch \ge c{\varepsilon} - Ch
\end{equation}
Since the symbol of
$\frac 1{2i}(Q_\varepsilon-(Q_\varepsilon )^*)$
is equal to the expression \eqref{12} modulo $\Cal O(h)$, the
sharp G\aa{}rding inequality for systems in
Proposition~\ref{garding}  gives
\begin{equation*}
{\im\w{Q_{\varepsilon}u,u} \ge  (c{\varepsilon - C_0 h)}\Vert u\Vert ^2 \ge
  \frac {\varepsilon c} 2\Vert u\Vert ^2}
\end{equation*}
for $h \ll {\varepsilon} \ll 1$.
By using the Cauchy-Schwarz inequality, we obtain
\begin{equation}\label{Qest} 
 \frac {\varepsilon c} 2 \Vert u\Vert \le \Vert Q_\varepsilon u\Vert 
\end{equation} 
Since $Q = MP$ the calculus gives
\begin{equation}\label{Qest0} 
Q_{\varepsilon} = M_{{\varepsilon}}P_{\varepsilon} +
\Cal O(h)
\end{equation}  
where $P_{\varepsilon} = 
e^{-\varepsilon T/h}P e^{\varepsilon T/h}$ and $M_{\varepsilon} = e^{-\varepsilon T/h}M
e^{\varepsilon T/h}= M + \Cal O({\varepsilon})$ is bounded 
and invertible for small enough ~${\varepsilon}$.
For $h \ll {\varepsilon}$ we obtain from ~\eqref{Qest}--\eqref{Qest0} that
\begin{equation}\label{Pepsest}
 \mn {u}\le \frac{C}{{\varepsilon}} \Vert P_{\varepsilon}u \Vert
\end{equation}
so $P_{\varepsilon}$ is injective with closed range.
Now $-Q^*$ satisfies the conditions~\eqref{qs1}--\eqref{qs2}, with ~$T$
replaced by~ $-T$. Thus we also obtain the estimate~\eqref{Qest} 
for $Q^*_{\varepsilon} = P_{\varepsilon}^*M_{{\varepsilon}}^* + \Cal
O(h) $. Since $M_{\varepsilon}^*$ is invertible for small enough
~~$h$ we obtain the estimate ~\eqref{Pepsest} for
~$P^*_{\varepsilon}$, thus $P_{\varepsilon}$ is surjective. Because the
conjugation by $e^{\varepsilon T/h}$ is uniformly bounded on $L^2$ when ${\varepsilon} \le Ch$
we obtain the estimate ~\eqref{thmest} from  ~\eqref{Pepsest}.

Now conjugation with $e^{\varepsilon T/h}$ is bounded in $L^2$ (but not
uniformly) also when ~\eqref{4} holds. By taking ~$C_2$ arbitrarily large 
in  ~\eqref{4} we find from the estimate ~\eqref{Pepsest} for $P_{\varepsilon}$ and
$P_{\varepsilon}^*$ that
$$
{
D\left(0,Kh\log {\frac 1 h}\right)\bigcap\spec (P)=\emptyset 
}$$
for any $K > 0$ when $h > 0$ is sufficiently small.

{\bf The analytic case.} We assume as before that $z=0$ and
\begin{gather*}
 P ( h ) \sim \sum_{ j \ge 0} h^j P_j ^w ( x , h D ) \qquad P_0 = P
\end{gather*}
where $P_j$ are bounded and holomorphic in a tubular neighborhood of $T^*\br^n$,
satisfying ~\eqref{analsymb},
and $P_j^w(z,hD_z)$ is defined by the formula ~\eqref{weyl}, where we
may change the
integration to a suitable chosen contour instead of $T^*\br^n$
(see~\cite[Section 4]{sjo}).
As before, we shall follow the proof of Theorem~1.3 in ~\cite{dsz}
and use the theory of the weighted spaces $ H( \Lambda_{{\varrho}T} ) $
developed in \cite{HeSj} (see also \cite{mart}). 

The complexification  $T^* \bc^n $ of the symplectic manifold $ T^* \br^n $
is equipped with a complex symplectic form $ \omega_\bc $ giving two 
natural real symplectic forms $ \im \omega_\bc $ and $ \Re \omega_\bc $.
We find that $ T^* \br^n $ is Lagrangian with respect to the first form
and symplectic with respect to the second. In general, a submanifold
satisfying these two conditions is called an {\em IR-manifold}.  

Assume that $ T \in C^\infty_0 ( T^* \br^n ) $, then we may associate to it a
natural family of IR-manifolds:
\begin{equation}
\label{eq:2.ir}
\Lambda_{ {\varrho} T } = \{ w + i {\varrho} H_T ( w ): \; w \in T^* \br^n \}
\subset T^* \bc^n \,  \ \ \text{with $ {\varrho} \in \br $ and $ |{\varrho} | $ small}
\end{equation}
where as before we identify $T(T^*\br^n)$ with $T^*\br^n$,
see~\cite[p.\ 391]{dsz}.
Since $ \im ( \zeta dz ) $ is closed on $ \Lambda_{{\varrho} T } $,we find that there exists
a function $ G_{\varrho} $ on $ \Lambda_{{\varrho} T } $ such that
\[ d G_{\varrho} = - \im  ( \zeta dz )|_{ \Lambda_{{\varrho} T } }\]
In fact, we can write it down explicitely by parametrizing $ \Lambda_{{\varrho}T} $ 
by $ T^* \br^n $:
\[ G_{\varrho} ( z, \zeta ) = - \langle \xi, {\varrho} \nabla_\xi T ( x , \xi ) \rangle
+ {\varrho} T ( x , \xi ) \quad\text{for}\quad  ( z, \zeta ) = ( x
, \xi ) + i {\varrho} H_T ( x , \xi )   \]
The associated spaces $ H ( \Lambda_{ {\varrho} T } ) $ are going to be
defined by using the FBI transform:
\[  T : L^2 ( \br^n ) \rightarrow L^2 ( T^* \br^n )  \]
given by
\begin{equation}\label{fbidef}
  T u ( x , \xi ) = c_n h^{ - \frac{3n}{4}} \int_{\br^n } 
e^{ {i} ( \langle x - y , \xi \rangle + i  | x - y |^2)/2h  }
u ( y ) d y  
\end{equation}
The FBI transform may be
continued analytically to $ \Lambda_{{\varrho}T} $ so that $ T_{ \Lambda_{{\varrho}T} }
u \in C^\infty ( \Lambda_{{\varrho} T} )  $.  Since $ \Lambda_{{\varrho} T } $ differs from 
$ T^* \br^n $ on a compact set only, we find that $  T_{ \Lambda_{{\varrho}T} } u $ is 
square integrable on ~$ \Lambda_{{\varrho} T} $. The FBI transform can of
course also be defined on $u\in L^2(\br^n)$
having values in~ $\bc^N$, and
the spaces $ H ( \Lambda_{{\varrho} T} ) $ are defined by putting $ h$ dependent
norms on $ L^2 (\br^n ) $:
\[ \| u \|_{ H( \Lambda_{ \varrho T} ) } ^2 = \int_{ \Lambda_{{\varrho} T} } 
| T_{\Lambda_{\varrho T}} u ( z , \zeta )|^2 e^{ - 2 G_{\varrho} ( z , \zeta ) / h } 
( \omega |_{\Lambda_{ {\varrho} T} } )^n / n! = \mn{
  T_{\Lambda_{\varrho T}} u }^2_{L^2({\varrho},h)}\, \]

Suppose that $ P_1 $ and $ P_2$ 
are bounded and holomorphic $N \times N$ systems in 
a neighbourhood of $ T^* \br^n $ in $T^* \bc^{n} $ 
and that $u\in L^2(\br^n, \bc^N)$. Then we find for $ {\varrho}> 0 $ small enough
\begin{multline} \label{eq:cofe}
\langle P_1^w ( x , h D) u , P_2^w ( x, h D) v \rangle_{ H ( \Lambda_{
    {\varrho} T} ) } \\ =  
\langle (P_1|_{ \Lambda_{ {\varrho} T } } ) T_{\Lambda_{\varrho T}} u , 
 (P_2|_{ \Lambda_{ {\varrho} T } } ) T_{\Lambda_{ \varrho T}} v 
\rangle_{L^2({\varrho},h)} 
+  {\mathcal O} ( h ) \| u \|_{ H ( \Lambda_{ {\varrho} T} ) } 
\| v \|_{ H ( \Lambda_{ {\varrho} T} ) } 
\end{multline}
By taking $ P_1 = P_2 = P $ and $ u = v $
we obtain
\begin{equation}
\label{eq:toep}
\| P^w ( x , h  D) u \|^2 _{ H ( \Lambda_{ {\varrho} T}  ) }  
= \| ( P|_{ \Lambda_{  {\varrho} T}  })  T_{\Lambda_{ {\varrho} T} } u
\|^2_{L^2({\varrho},h)}
 + {\mathcal O} ( h ) \| u \|^2 _{  H ( \Lambda_{ {\varrho} T})}   
\end{equation}
as in the scalar case, see \cite{HeSj} or~ \cite{mart}.

By Remark~\ref{eqrem} we may assume that $MP = Q$ satisfies \eqref{qs1b}--\eqref{qs2b}, with
invertible ~$M$. The analyticity of ~$P$ gives
\begin{equation} 
P(w +i{\varrho}H_T)= P(w )+ i{\varrho}H_{T} P(w )+{\mathcal
  O}({\varrho}^2) \qquad \qquad |{\varrho}| \ll 1
\end{equation}
by Taylor's formula, thus
\[ 
\im M(w)P(w + i{\varrho}H_T(w ))
= \im Q(w )+ {\varrho}\, {\Re M(w)H_{T}P}(w )+{\mathcal 
O}({\varrho}^2) 
\]
Since we have ${\Re M H_{T} P} >c - C\im Q$, $c > 0$,  by~\eqref{qs1b} and $\im Q
\ge 0$  by~\eqref{qs2b}, we obtain for  sufficiently small $ {\varrho} >0 $ that
\begin{equation}
\label{eq:6}
\im M(w)P (w + i{\varrho} H_T(w)) \ge (1- C{\varrho})\im
Q(w) + c{\varrho} + \Cal O({\varrho}^2) \ge c{\varrho}/2  
\end{equation}
which gives by the Cauchy-Schwarz inequality that 
$ 
 \mn{ P\restriction_{ \Lambda_{ {\varrho} T} } u }  \ge c'{\varrho}\mn u 
$ 
and thus
\begin{equation} \label{eq:l.4} 
 \mn{ P^{-1}\restriction_{ \Lambda_{ {\varrho} T} } } \le C/{\varrho}
\end{equation}
Now recall that  $H(\Lambda _{{\varrho}T})$ is equal to $L^2$ as a space and that
the norms are equivalent for every fixed $h$ (but not uniformly). Thus
the spectrum of $P ( h ) $ does not depend on whether  
the operator is realized on $L^2$ or on $H(\Lambda _{{\varrho}T})$. 
We conclude from ~\eqref{eq:toep} and
~\eqref{eq:l.4} that $0$ has an $h$-independent
 neighbourhood which is disjoint from the spectrum of $ P ( h ) $, when
$h$ is small enough. 
\end{proof}

Summing up, we have proved the following result.

\begin{prop}\label{analthm} 
Assume that $P(h)$ is an $N \times N$ system on the form given
by ~\eqref{Pdef} with  analytic principal symbol $P(w)$,
and that there exists a real valued {function} 
$T(w) \in C^\infty(T^*\br^n)$ such that $P(w)- {z}\id_N $ is quasi-symmetrizable 
with respect to $H_T$ in a neighborhood of ${\Sigma}_{z}(P)$.
Define the IR-manifold  
$$\Lambda _{\varrho T}=\{ w +i {\varrho} H_T(w );\, w 
\in T^*{\br}^{n}\}$$ 
for ${\varrho}>0$ small enough. Then 
\begin{gather*}
P( h ) - z : \; H(\Lambda _{{\varrho}T}) \ \longrightarrow \  H(\Lambda _{{\varrho}T}) 
\end{gather*}  
has a bounded inverse for $ h $ small enough, which gives

\[  \spec ( P ( h ) ) \bigcap D ( z , \delta ) = \emptyset  \qquad 0 < h < h_0 \]
for $ \delta $ small enough. 
\end{prop}

\begin{rem}\label{QSrem}
It is clear from the proof of  Theorem~\ref{QSthm} that in the
analytic case it suffices that $P_j$ is analytic in a fixed complex neighborhood of 
${\Sigma}_{z}(P) \Subset T^*\br^n$, $j \ge 0$.
\end{rem}

\section{The Subelliptic Case}
\label{subell}

We shall investigate when we have an estimate of the resolvent which
is better than the quasi-symmetric estimate, for example the subelliptic type of estimate 
\begin{equation*}
\mn{(P(h) - {\lambda}\id_N)^{-1}} \le  Ch^{-{\mu}}\qquad h \to 0
\end{equation*}
with ${\mu} < 1$, which we obtain in the scalar case under the
bracket condition, see Theorem~1.4 in~\cite{dsz}.

\begin{exe} \label{scalarex}
Consider the scalar operator $p^w =  hD_t + i f^w(t, x,hD_x)$ where
$0 \le f(t,x,{\xi}) \in C_\rb^\infty$, $(t,x) \in \br\times \br^n$, 
and $0 \in \partial
{\Sigma}(f)$. Then we obtain from Theorem~1.4 in ~\cite{dsz} the
estimate
\begin{equation}\label{subk}
 h^{k/k+1}\mn u \le C\mn{p^w u} \qquad h \ll 1 \qquad\forall\, u \in C^\infty_0
\end{equation}
if $0 \notin {\Sigma}_\infty(f)$ and
\begin{equation} \label{scalarcond}
 \sum_{j\le k} |\partial_{t}^jf| \ne 0
\end{equation} 
These conditions are also necessary. For example, if $ |f(t)|
\le C |t|^{k}$ then an easy computation gives 
$\mn{hD_t u + ifu}/\mn u \le ch^{k/k+1} $ if $u(t) =
{\phi}(th^{-1/k+1})$ with $0 \ne {\phi}(t) \in C_0^\infty(\br)$.
\end{exe}

The following example
shows that condition~\eqref{scalarcond} is not sufficient for systems.

\begin{exe}\label{ex1}
Let $P = h D_t\id_2 + i F(t)$ where
\begin{equation*}
 F(t) =
\begin{pmatrix}
t^2 & t^3 \\ t^3 & t^4 
\end{pmatrix}
\end{equation*}
Then we have
$F^{(3)}(0)  = 
\begin{pmatrix}
0 & 6 \\ 6 & 0 
\end{pmatrix}
$ which gives that
\begin{equation*}
 \bigcap_{j \le 3} \ker F^{(j)}(0) = \set{0}
\end{equation*}
But by taking $u(t) = {\chi}(t)(t, -1)^t$ with $0 \ne {\chi}(t) \in
C^\infty_0(\br)$, we obtain $F(t)u(t) \equiv 0$ so
we find $\mn {Pu}/\mn u  \le ch$.
Observe that 
$$F(t) = 
\begin{pmatrix}
 1 &  -t \\ t & 1
\end{pmatrix}
\begin{pmatrix}
t^2 & 0 \\ 0 & 0 
\end{pmatrix}
\begin{pmatrix}
 1 &  t \\ -t & 1
\end{pmatrix} 
$$ 
thus  $F(t) = t^2B^*(t){\Pi}(t)B(t)$
where $B(t)$ is invertible and ${\Pi}(t)$ is a projection of rank one.
\end{exe}

\begin{exe}\label{ex2}
Let $P =  h D_t\id_2 + i F(t)$ where
\begin{equation*}
 F(t) =
\begin{pmatrix}
t^2 + t^8 & t^3 - t^7 \\ t^3 - t^7 & t^4 +t^6 
\end{pmatrix} 
 = 
\begin{pmatrix}
 1 &  -t \\ t & 1
\end{pmatrix}
\begin{pmatrix}
t^2 & 0 \\ 0 & t^6
\end{pmatrix}
\begin{pmatrix}
 1 &  t \\ -t & 1
\end{pmatrix}
\end{equation*}
Then we have that
\begin{equation*}
 P = (1 + t^2)^{-1}
\begin{pmatrix}
 1 &  -t \\ t & 1
\end{pmatrix}
\begin{pmatrix}
hD_t + i(t^2+ t^4) & 0 \\ 0 & hD_t + i(t^6 + t^8)
\end{pmatrix}
\begin{pmatrix}
 1 &  t \\ -t & 1
\end{pmatrix}
+ \Cal O(h)
\end{equation*}
Thus we find from the scalar case that
$h^{6/7}\mn u \le C\mn{Pu}$ for $h \ll 1$, see ~\cite[Theorem~1.4]{dsz}.
Observe that this operator is, element for element, a
higher order perturbation of the operator of Example~\ref{ex1}.
\end{exe}

\begin{defn} 
Let  $0 \le F(t)\in L^\infty_{loc}(\br)$ be an $N \times N$ system, then we define
\begin{equation} \label{omegadef}
{\Omega}_{\delta}(F) = \set{t: \min_{\mn u = 1} \w{F(t)u,u} \le
  {\delta}} \qquad 0 < {\delta} \le 1
\end{equation}
which is well-defined almost everywhere and contains ${\Sigma}_0(F) = |F|^{-1}(0)$.
\end{defn}

Observe that one can also use this definition in the scalar
case, then ${\Omega}_{\delta}(f) = f^{-1}([0,{\delta}])$ for
non-negative functions ~$f$.

\begin{rem}\label{subinvrem}
Observe that if $F \ge 0$ and $E$ is invertible then
we find that 
\begin{equation}\label{subinv}
{\Omega}_{\delta}(E^*FE)\subseteq {\Omega}_{C\delta}(F)  
\end{equation}
where $C = \mn {E^{-1}}^2$. 
\end{rem}

\begin{exe} 
For the scalar symbols $p(x,{\xi}) = {\tau} + i f(t, x,{\xi})$ in
Example~\ref{scalarex} we find from  Proposition~\ref{scalarsubcondlem} that \eqref{scalarcond} 
is equivalent to
\begin{equation*}
 |\set{t: f(t,x,{\xi}) \le {\delta}}| = |{\Omega}_{\delta}(f_{x,{\xi}})|\le C{\delta}^{1/k}\qquad
 0 < {\delta} \ll 1  \quad \forall\,x, \xi 
\end{equation*}
where $f_{x,{\xi}}(t) = f(t,x,{\xi})$.
\end{exe}

\begin{exe}
For the matrix $F(t)$ in Example~\ref{ex2} we find from Remark~\ref{subinvrem} that
$|{\Omega}_{\delta}(F)| \le C {\delta}^{1/6}$, and
for the  matrix in Example~\ref{ex1} we find that
$|{\Omega}_{\delta}(F)| = \infty$.
  
\end{exe}

We also have examples when the semidefinite imaginary part vanishes of
infinite order. 

\begin{exe}
Let 
$
p(x,{\xi}) = {\tau} + i f(t, x,{\xi}) 
$
where 
$
 0 \le f(t,x,{\xi}) \le C e^{-1/|t|^{{\sigma}}}
$, ${\sigma} > 0$,
then we obtain that 
\begin{equation*}
 |{\Omega}_{\delta}(f_{x,{\xi}})|\le C_0 |\log {\delta}|^{-1/{\sigma}} \qquad
 0 <{\delta} \ll 1 \quad \forall\,x, \xi
\end{equation*}
(We owe this example to Y. Morimoto.)
\end{exe}

The following example shows that for subelliptic type of estimates it is not sufficient to have
conditions only on the vanishing of the symbol, we also need conditions 
on the semibicharacteristics of the eigenvalues.

\begin{exe}\label{subex}
Let 
\begin{equation*}
 P =  hD_t\id_2 +
{\alpha} h \begin{pmatrix}
  D_x & 0 \\ 0 & -D_x
\end{pmatrix}
+ i(t - {\beta} x)^2\id_2 \qquad (t,x)\in \br^2
\end{equation*}
with ${\alpha}$, ${\beta} \in \br$, then we see from the scalar case that $P$ satisfies the
estimate~\eqref{subk} with ${\mu} = 2/3$ if and only either ${\alpha} = 0$
or ${\alpha} \ne 0$ and ${\beta} \ne \pm 1/{\alpha}$.
\end{exe}

\begin{defn}\label{apprdef}
Let  $Q(w) \in C^\infty ( T^* \br^n)$ be an $N \times N$ system and let $w_0
\in {\Sigma} \subset T^*\br^n$. We say that $Q$
satisfies the {\em approximation property} on~${\Sigma}$ near~$w_0$ if 
there exists a $Q$ invariant $C^\infty$ subbundle $\Cal V$ of ~$\bc^N$
over $T^*\br^n$ such that $\Cal V(w_0) = \ker Q^N(w_0)$ and
\begin{equation}\label{approxcond}
 \re\w{Q(w)v,v} = 0\qquad v \in \Cal V(w)  \qquad w \in {\Sigma} 
\end{equation}
near $w_0$. That  $\Cal V$ is  $Q$ invariant means that $Q(w)v \in \Cal V(w)$ for
$ v \in \Cal V(w)$.  
\end{defn}

Here $\ker Q^N(w_0)$ is  the
space of the generalized eigenvectors corresponding to the zero eigenvalue.
The symbol of the system in Example~\ref{subex} satisfies the approximation
property on ${\Sigma} = \set{{\tau} = 0}$ if and only if ${\alpha}
= 0$. 

Let $\wt Q = Q\restr{\Cal V}$ then since $\im i \wt Q = \re \wt
Q$ we obtain from  
Lemma~\ref{semiprop} that $\ran \wt Q \bot \ker \wt Q$ on ~${\Sigma}$. Thus $\ker \wt Q^N =
\ker \wt Q$ on ~${\Sigma}$, and since $\ker \wt Q^N(w_0) = \Cal V(w_0)$ we find that
$ \ker Q^N(w_0) = \Cal V(w_0)= \ker Q(w_0)$. It follows from
Example~\ref{normformex} that $\ker Q \subseteq \Cal V$ near~$w_0$.

\begin{rem} \label{condrem}
Assume that $Q$ satisfies the approximation property on~ the
$C^\infty$ hypersurface ${\Sigma}$
and is quasi-symmetric with 
respect to  $V  \notin T{\Sigma}$.
Then the limits of the non-trivial semibicharacteristics
of the eigenvalues of~$Q$ close to zero coincide with the bicharacteristics
of ~${\Sigma}$. 
\end{rem}

In fact, the approximation
property in Definition~\ref{apprdef} and Example~\ref{normformex}
give that $ \w{\re Qu,u
} = 0$ for $u\in \ker Q \subseteq \Cal V$ on ${\Sigma}$. Since $\im
Q \ge 0$ we find that
\begin{equation}\label{tangvan}
 \w{dQu,u } = 0 \qquad \forall\, u
 \in \ker Q \qquad\text{on $T{\Sigma}$}
\end{equation}
By Remark~\ref{qshamrem} the limits of the  non-trivial
semibicharacteristics of the eigenvalues close to zero  of 
$Q$ are curves with tangents determined by
$\w{dQ u,u}$ for $u \in \ker Q$.
Since $V \re Q \ne 0$ on $\ker Q$ we find from~\eqref{tangvan}
that the limit curves coincide with the bicharacteristics
of ~${\Sigma}$, which are the flow-outs of the Hamilton vector field.

\begin{exe} 
Observe that Definition~\ref{apprdef} is empty if $\dim \ker Q^N(w_0) = 0$.
If $\dim \ker Q^N(w_0) > 0$, then there exists ${\varepsilon} > 0$  and a neigborhood
${\omega}$ to $w_0$ so that
\begin{equation}\label{spprojdef}
 {\Pi}(w) = \frac{1}{2{\pi}i}\int_{|z|= {\varepsilon}} (z\id_N - Q(w))^{-1}\,dz
 \in C^\infty({\omega})
\end{equation}
is the spectral projection on the (generalized) eigenvectors with
eigenvalues having absolute value less than~${\varepsilon}$. Then $\ran {\Pi}$ is a $Q$
invariant bundle over~${\omega}$ so that $\ran {\Pi}(w_0)  = \ker
Q^N(w_0)$. Condition~\eqref{approxcond} with $\Cal V = 
\ran {\Pi}$ means that ${\Pi}^*\re Q{\Pi} \equiv 0$ in
~${\omega}$. When $\im Q(w_0) 
\ge 0 $ we find that ${\Pi}^*Q{\Pi}(w_0) = 0$, then ~$Q$ satisfies the
approximation property on ~${\Sigma}$ near ~$w_0$
with $\Cal V = \ran {\Pi}$ if and only if 
$$
d({\Pi}^*(\re Q){\Pi})\restr{T{\Sigma}} \equiv 0 \qquad \text{ near ~$w_0$}
$$
\end{exe}

\begin{exe}\label{normformex}
If $Q$ satisfies the approximation property on ${\Sigma}$, then by choosing an
orthonormal base for $\Cal V$ and extending it to an orthonormal base for
~$\bc^N$ we obtain the system on the form
\begin{equation}\label{normformex0}
 Q = 
\begin{pmatrix}
Q_{11} & Q_{12} \\ 0 & Q_{22} 
\end{pmatrix}
\end{equation}
where $Q_{11}$ is $K \times K$ system such that $Q^{N}_{11}(w_0) = 0$,
$\re Q_{11} = 0$ on ${\Sigma}$ and $|Q_{22}| \ne 0$.
By multiplying from left with
\begin{equation*}
 \begin{pmatrix}
\id_K & -Q_{12}Q_{22}^{-1} \\ 0 & \id_{N-K} 
\end{pmatrix}
\end{equation*}
we obtain that $Q_{12} \equiv 0$ without changing $Q_{11}$ or $Q_{22}$. 
\end{exe}

In fact, the eigenvalues of ~$Q$ are then eigenvalues of either
~$Q_{11}$ or ~$Q_{22}$. Since ~$\Cal V(w_0)$ are the (generalized)
eigenvectors corresponding to the zero eigenvalue of~$Q(w_0)$ we find
that all eigenvalues of ~ $Q_{22}(w_0)$ are non-vanishing, thus
$Q_{22}$ is invertible near~$w_0$,

\begin{rem}
If $Q$ satisfies the approximation property on ${\Sigma}$ near ~~$w_0$,
then it satisfies the approximation property on ~ ${\Sigma}$ near ~$w_1$,
for $w_1$ sufficiently close to ~$w_0$.  
\end{rem}

In fact, let $Q_{11}$ be the restriction of $Q$ to $\Cal V$ as in Example~~\ref{normformex},
then since $\re Q_{11} = \im i Q_{11} = 0$ on ${\Sigma}$ we find from Lemma~\ref{semiprop} that
 $\ran Q_{11} \bot \ker Q_{11}$ and
$\ker Q_{11} = \ker Q_{11}^N$ on
${\Sigma}$. Since $Q_{22}$ is invertible in ~\eqref{normformex0}, we
find that $\ker Q \subseteq \Cal V$. Thus, by using the spectral
projection ~\eqref{spprojdef} of $Q_{11}$ near ~$w_1 \in {\Sigma}$
for small enough~${\varepsilon}$ we obtain an invariant subbundle
$\wt {\Cal V} \subseteq \Cal V$ so that $\wt {\Cal V}(w_1) = \ker Q_{11}(w_1)
= \ker Q^N(w_1)$.

If  $Q \in C^\infty$ satisfies the approximation property and $Q_E =
E^*QE$ with invertible $E\in C^\infty$, then it follows from the proof of 
Proposition~\ref{fininv} below that there exist invertible $A$ and $B \in C^\infty$,
so that $AQ_E$ and  $Q^*B$ satisfy the approximation property.

\begin{defn}\label{subdef}
Let $P \in C^\infty (T^*\br^n) $ be an $N \times N$ system  and let ${\phi}(r)$
be a positive non-decreasing function on $\br_+ $. We say that ~$P$ is
of {\em  subelliptic type~ ${\phi}$} if for any $w_0 \in {\Sigma}_0(P)$
there exists a neighborhood 
${\omega}$ of $w_0$, a $C^\infty$ hypersurface~${\Sigma}\ni w_0$, a real
$C^\infty$ vector field $V \notin T{\Sigma}$
and an invertible symmetrizer ~$M \in C^\infty$ so that 
$Q = MP$ is quasi-symmetric with respect to~$V$ in ${\omega}$ and
satisfies the approximation property on ${\Sigma}\bigcap {\omega}$. Also,  
for every bicharacteristic ~${\gamma}$ of ~${\Sigma}$ 
the arc length
\begin{equation}\label{subcond}
\big| {\gamma}\cap {\Omega}_{\delta}(\im Q) \cap
{\omega}\big| \le C{\phi}({\delta}) \qquad 0 < {\delta} \ll 1
\end{equation}
We say that $z$  is of 
{subelliptic type~ ${\phi}$}  for $P \in C^\infty$ if  $P-
z\id_N$ is of subelliptic type ${\phi}$.
If ${\phi}({\delta}) = {\delta}^{{\mu}}$ then we say that the system is of
finite type of order~${\mu}\ge 0$, which generalizes the definition of finite type for
scalar operators in~\cite{dsz}. 
 \end{defn}

Recall that the bicharacteristics of a hypersurface in $T^*X$ are the
flow-outs of the Hamilton vector field of~${\Sigma}$.
Of course, if $P$ is elliptic then by choosing $M = iP^{-1}$ we obtain
$Q = i\id_N$, so ~$P$ is trivially of subelliptic type. 
If $P$ is of subelliptic type, then it is quasi-symmetrizable by
definition and thus of principal type.

\begin{rem}\label{nfrem}
Observe that we may assume that 
\begin{equation}\label{XQC}
 \im\w{Qu,u} \ge c\mn{Qu}^2 \qquad \forall \, u \in \bc^N
\end{equation}
in Definition~\ref{subdef}.
\end{rem}

In fact, by adding $i{\varrho}P^*$ to $M$ we obtain~\eqref{XQC} for
large enough~${\varrho}$ by~\eqref{immp}, and this only
increases $\im Q$.
 
Since $Q$ is in $C^\infty$ the estimate~\eqref{subcond} cannot be satisfied
for any ${\phi}({\delta}) \ll {\delta}$ (unless $Q$ is
elliptic) and it is trivially satisfied with
${\phi} \equiv 1$, thus we shall only consider
$c{\delta} \le {\phi}({\delta}) \ll 1$ (or finite type of order
$0 < {\mu} \le 1$). Actually, for $C^\infty$ symbols of
finite type, the 
only relevant values in ~\eqref{subcond} are ${\mu} = 1/k$ for even $k
> 0$, see Proposition~\ref{subcondlem} in the Appendix.

Actually, the condition that ${\phi}$ is non-decreasing is
unnecessary, since the left-hand side in ~\eqref{subcond} is
non-decreasing (and upper semicontinuous) in~~${\delta}$, we can
replace ${\phi}({\delta})$ by 
$\inf_{{\varepsilon} > {\delta}} {\phi}({\varepsilon})$ to make it
non-decreasing (and upper semicontinuous).

\begin{exe}\label{Qex}
Assume that $Q$ is quasi-symmetric with respect to the real vector
field $V$, satisfying~\eqref{subcond} and
the approximation property on ${\Sigma}$. Then by choosing an
orthonormal base as in Example~\ref{normformex} we obtain the system on the form 
\begin{equation*}
 Q = 
\begin{pmatrix}
Q_{11} & Q_{12} \\ 0 & Q_{22} 
\end{pmatrix}
\end{equation*}
where $Q_{11}$ is $K \times K$ system such that $Q^{N}_{11}(w_0) = 0$,
$\re Q_{11} = 0$ on ${\Sigma}$ and $|Q_{22}| \ne 0$. Since $Q$ is
quasi-symmetric with respect to ~$V$ we also obtain that  $Q_{11}(w_0) = 0$, $\re VQ_{11}
> 0$, $\im Q_{jj} \ge 0$ for $j = 1$, $2$. In fact, then 
Lemma~\ref{semiprop} gives that $\im Q \bot \ker Q$ so $\ker Q^N =
\ker Q$. Since $Q$ satisfies~\eqref{subcond} and ${\Omega}_{\delta}(\im Q_{11}) \subseteq
{\Omega}_{\delta}(\im Q)$ we find that
$Q_{11}$ satisfies~\eqref{subcond}. 
By multiplying from left as in Example~\ref{normformex} we 
obtain that $Q_{12} \equiv 0$ without changing $Q_{11}$ or $Q_{22}$. 
\end{exe}

\begin{prop}\label{fininv}
If the $N \times N$ system $P(w)\in C^\infty( T^* \br^n)$ is of
subelliptic  type ${\phi}$ then $P^*$ is of  subelliptic  
type ${\phi}$. If $A(w)$ and $B(w)\in C^\infty ( T^* \br^n)$ are
invertible  $N \times N$ systems, then $APB$ is of  subelliptic type ${\phi}$.
\end{prop}

\begin{proof}
Let $M$ be the symmetrizer in Definition~\ref{subdef} so that $Q = MP$
is quasi-symmetric with respect to $V$. By choosing a suitable base 
and changing the symmetrizer as
in Example~\ref{Qex}, we may write
\begin{equation}\label{specformorig}
 Q = 
\begin{pmatrix}
Q_{11} & 0 \\ 0 & Q_{22} 
\end{pmatrix}
\end{equation}
where $Q_{11}$ is $K \times K$ system such that $Q_{11}(w_0) = 0$,  $V
\re Q_{11} > 0$, $\re Q_{11} = 0$ on
${\Sigma}$ and that $Q_{22}$ is invertible. We also have $\im Q \ge
0$ and that $Q$ satisfies~\eqref{subcond}. 
Let $\Cal V_1 = \set{u\in \bc^N:\ u_j = 0 \text{ for $j > K$}}$ and $\Cal V_2 =
\set{u\in \bc^N:\ u_j = 0 \text{ for $j \le K$}}$, these are $Q$
invariant bundles such that $ \Cal V_1 \oplus \Cal V_2 = \bc^N$.

First we are going to show that $\wt P =APB$ is of subelliptic  type. 
By taking $\wt M = B^{-1}MA^{-1}$ we find that 
\begin{equation}\label{specform}
\wt M \wt P = \wt Q  = B^{-1}QB
\end{equation}
and it is clear that $B^{-1}\Cal V_j$ are $\wt Q$ invariant
bundles, $j = 1$, $2$. By choosing bases in  $B^{-1}\Cal V_j$ for
$j = 1$, $2$, we obtain a base for $\bc^N$ in which $\wt Q$ has a block
form:
\begin{equation}\label{qblock}
 \wt Q = 
\begin{pmatrix}
\wt Q_{11} & 0 \\ 0 & \wt Q_{22}  
\end{pmatrix}
\end{equation}
Here $\wt Q_{jj}: B^{-1}\Cal V_j \mapsto B^{-1}\Cal V_j$, is given by
$\wt Q_{jj} = B_{j}^{-1} Q_{jj}B_{j}$ with 
$$B_{j}: B^{-1} \Cal V_j\ni u \mapsto Bu \in \Cal V_j \qquad j= 1,\ 2$$ 
By multiplying $\wt Q$ from the left with 
\begin{equation*}
\Cal B =
\begin{pmatrix}
 B_{1}^*B_{1} & 0 \\ 0 &  B_{2}^*B_{2}
\end{pmatrix}
\end{equation*}
we obtain that 
\begin{equation*}
\ol Q = \Cal B \wt Q = \Cal B \wt M \wt P =
\begin{pmatrix}
 B_{1}^*Q_{11}B_{1} & 0 \\ 0 &  B_{2}^*Q_{22}B_{2}
\end{pmatrix}
=
\begin{pmatrix}
\ol Q_{11} & 0 \\ 0 & \ol Q_{22}  
\end{pmatrix}
\end{equation*}
It is clear that $\im \ol Q \ge 0$, $Q_{11}(w_0) = 0$, $\re \ol Q_{11}
= 0$ on ${\Sigma}$, $|\ol Q_{22}| \ne 0$
and $V \re \ol Q_{11} > 0$ by Proposition~\ref{invrem0}. Finally, we
obtain from Remark~\ref{subinvrem} that 
\begin{equation}\label{subinv0}
{\Omega}_{\delta}(\im \ol Q) \subseteq {\Omega}_{C\delta}(\im Q)
\end{equation}
for some $C > 0$, which proves that $\wt P = APB$ is of  subelliptic  type.
Observe that $\ol Q = AQ_B$, where $Q_B = B^*QB$ and $A = \Cal B B^{-1}(B^*)^{-1}$.

To show that $P^*$ also is of subelliptic type, we may assume as before that
$Q= MP$ is on the form~\eqref{specformorig} with $Q_{11}(w_0) = 0$,  $
V \re Q_{11} > 0$, $\re Q_{11} = 0$ on
${\Sigma}$, $Q_{22}$ is invertible, $\im Q \ge 0$ and  $Q$ satisfies~\eqref{subcond}.
Then we find that
\begin{equation*}
 -P^*M^* = -Q^* = 
\begin{pmatrix}
 - Q_{11}^* & 0 \\ 0 & - Q_{22}^*   
\end{pmatrix}
\end{equation*}
satisfies the same conditions with respect to $-V$, so it is of subelliptic
type with multiplier~ $\id_N$. By the first part of the proof we 
obtain that $P^*$ is of subelliptic type, which finishes the proof.
\end{proof}

\begin{exe} \label{scalsubex} 
In the scalar case, $p\in C^\infty(T^*\br^n)$ is
quasi-symmetrizable with respect to $H_t = \partial_{\tau}$ near~$w_0$ if and only if  
\begin{equation}\label{pex}
 p(t,x;{\tau},{\xi}) = q(t,x;{\tau},{\xi})({\tau} + i f(t,x,{\xi}))\qquad\text{near~$w_0$}
\end{equation}
with $f \ge 0$ and $q \ne 0$, see Example~\ref{scalarex0}. If $0
\notin {\Sigma}_\infty(p)$ we find by taking $q^{-1}$ as symmetrizer
that $p$ in~\eqref{pex} is of finite type of order ${\mu}$ if and only
if ${\mu} = 1/k$ for an even ~$k$ such that
\begin{equation*}
 \sum_{j \le k}|\partial_{t}^kf | > 0
\end{equation*}
by Proposition~\ref{scalarsubcondlem}. In fact, the approximation property
is trivial since $f$ is real. Thus we obtain
the case in~\cite[Theorem~1.4]{dsz}, see Example~\ref{scalarex}. 
\end{exe}

\begin{thm}\label{subthm}
Assume that the $N \times N$ system $P(h)$ is
given by the expansion ~\eqref{Pdef} with principal symbol  $P \in
C^\infty_\rb(T^*\br^n)$. Assume that $z \in {\Sigma}(P) 
\setminus {\Sigma}_\infty(P)$ is of subelliptic type ${\phi}$ for
~$P$, where ${\phi}> 0$ is non-decreasing on $\br_+$. 
Then there exists  $h_0>0$ so that
\begin{equation}
\label{eq:sube}
 \| ( P ( h ) - z\id_N )^{-1} \| \leq C/{\psi}(h)
 \qquad 0 <  h \le h_0 
\end{equation}
where $ {\psi}(h) = {\delta}$ is the inverse to $h = {\delta}{\phi}({\delta})$.
It follows that there exists $c_0 > 0$ such that
\begin{equation}
\label{eq:sube'}
\set{w:\ |w - z|\le c_0 {\psi}(h)} \cap {\sigma}(P(h)) = \emptyset
\qquad\text{$0 < h \le h_0$}.
\end{equation} 
\end{thm}

Theorem~\ref{subthm} will be proved in section~\ref{thmpf}.
Observe that if ${\phi}({\delta}) \to c > 0$ as ${\delta} \to 0$
then ${\psi}(h) = \Cal O(h)$ and Theorem~\ref{subthm} follows from Theorem~\ref{QSthm}.
Thus we shall assume that ${\phi}({\delta}) \to 0$ as ${\delta} \to 0$,
then we find that $h = {\delta}{\phi}({\delta}) = o({\delta})$ so ${\psi}(h)
\gg h$ when $h \to 0$.
In the finite type case: ${\phi}({\delta}) = {\delta}^{\mu}$ we find that
${\delta}{\phi}({\delta}) =
{\delta}^{1 +{\mu}}$ and ${\psi}(h) = h^{1/{\mu} + 1}$. 
When ${\mu} = 1/k$ we find that
$1 +{\mu} = (k+1)/k$ and ${\psi}(h) = h^{k/k+1}$. 
Thus Theorem~\ref{subthm} generalizes
Theorem~1.4 in ~\cite{dsz} by Example~\ref{scalsubex}. 
Condition ~\eqref{eq:sube}  with ${\psi}(h) = h^{1/{\mu}+1}$ means that
${\lambda} \notin {\Lambda}^{\rm  sc}_{1/{\mu}+1}(P)$, which is the
pseudospectrum of index $1/{\mu}+1$.

\begin{exe}\label{evexe2}
Assume that $P(w)\in C^\infty$ is $N \times N$ and $z \in
{\Sigma}(P)\setminus({\Sigma}_{ws}(P) \bigcup 
{\Sigma}_\infty(P))$. Then
${\Sigma}_{\mu}(P) = \set{{\lambda}(w) = {\mu}}$ for ${\mu}$ close to $z$, where
${\lambda} \in C^\infty$ is a germ of eigenvalues for ~$P$ at
${\Sigma}_{z}(P)$, see Lemma~\ref{evlem}.
If $z \in \partial {\Sigma}({\lambda})$ we find from Example~\ref{evexe1} that $P(w) - z\id_N$ is
quasi-symmetrizable near $w_0 \in {\Sigma}_z(P)$ if 
$P(w) - {\lambda}\id_N$ is of principal type when $|{\lambda} - z| \ll
1$. Then  $P$ is on the form~\eqref{4.8a} and there exists $q(w) \in 
C^\infty$ so that~\eqref{4.8}--\eqref{4.9} hold
near ${\Sigma}_z(P)$. We can then choose the multiplier ~$M$ so that
~$Q$ is on the form ~\eqref{Qnorm}.
By taking ${\Sigma} = \set{\re q({\lambda} -z) = 0}$ we obtain that $P
- z\id_N$ is of subelliptic type ${\phi}$ if ~\eqref{subcond} is satified for
$\im q({\lambda}- z)$. In fact, by the invariance we find that the
approximation property is trivially satisfied since $\re q{\lambda}
\equiv 0$ on ${\Sigma}$. 
\end{exe}

\begin{exe}\label{Qex3}
Let 
\begin{equation*}
P(x,{\xi}) = |{\xi}|^2\id_N + i K(x)\qquad (x,{\xi}) \in T^*\br^n
\end{equation*}
where $K(x)\in C^\infty(\br^n)$ is symmetric as in
Example~~\ref{Qex1}. We find that $P - z\id_N$ is 
of finite type of order $1/2$ when $z = i{\lambda}$ for almost all
${\lambda} \in {\Sigma}(K)\setminus 
({\Sigma}_{ws}(K)\bigcup {\Sigma}_\infty(K))$ by Example~\ref{evexe2}.
In fact, then $z \in {\Sigma}(P)\setminus({\Sigma}_{ws}(P) \bigcap
{\Sigma}_\infty(P))$ and the $C^\infty$ germ of eigenvalues for $P$
near ${\Sigma}_{z}(P)$ is
${\lambda}(x,{\xi}) =  |{\xi}|^2 + i {\kappa}(x)$, where ${\kappa}(x)$ is a $C^\infty$
germ of eigenvalues for $K(x)$  
near ${\Sigma}_{\lambda}(K)  = \set{{\kappa}(x) =
  {\lambda}}$. 
For almost all values 
${\lambda}$ we have $d{\kappa}(x) \ne 0$ on ${\Sigma}_{{\lambda}}(K)$. 
By taking $q = i$ we obtain for such values that~\eqref{subcond} is satified 
for $\im i({\lambda}(w)- i{\lambda}) = |{\xi}|^2$ with
${\phi}({\delta}) = {\delta}^{1/2}$, since $\re i({\lambda}(w)- i{\lambda})
= {\lambda} - {\kappa}(x) = 0$ on ${\Sigma} = {\Sigma}_{\lambda}(K)$. If $K(x)\in
C_\rb^\infty$ and $0 \notin {\Sigma}_\infty(K)$ then we may 
use Theorem~\ref{subthm}, Proposition~\ref{reduxlem}, Remark~\ref{reduxrem} and
Example~\ref{sysex} to obtain the estimate 
\begin{equation}
\mn{(P^w(x,hD) - z \id_N)^{-1}} \le C h^{-2/3}\qquad 0 < h \ll 1 
\end{equation}
on the resolvent.
\end{exe}

\begin{exe}\label{simplex}
Let  
\begin{equation*}
 P(t,x;{\tau},{\xi}) = {\tau}M(t,x,{\xi})+ iF(t,x,{\xi}) \in
 C^\infty 
\end{equation*}
where  $M \ge c_0> 0$ and $F \ge 0$ satisfies 
\begin{equation}\label{subcond0}
\left|\set{t:\ \inf_{|u| = 1}\w{F(t,x,{\xi})u,u} \le {\delta}} \right| \le C
{\phi}({\delta}) \qquad \forall\, x,{\xi}
\end{equation}
Then $P$ is quasi-symmetrizable with respect to $\partial
_{{\tau}}$ with symmetrizer $\id_N$. When ${\tau} =
  0$ we obtain that $\re P = 0$, so by taking $\Cal V = \ran {\Pi}$
  for the spectral projection
  ${\Pi}$ given by ~\eqref{spprojdef} for $F$, we find that
$P$ satisfies the approximation
  property with respect to ${\Sigma}= \set{{\tau} = 0}$. 
Since ${\Omega}_{\delta}(\im P) =
{\Omega}_{\delta}(F)$ we find from ~\eqref{subcond0}
that $P$ is of subelliptic type ~${\phi}$.
Observe that if $0 \notin {\Sigma}_\infty( F)$ we obtain from
Proposition~\ref{subcondlem} that ~\eqref{subcond0} is satisfied for
${\phi}({\delta}) = {\delta}^{\mu}$ if
and only if ${\mu} \le 1/k$ for an even $k \ge 0$ so that 
\begin{equation*}
 \sum_{j \le k}|\partial_t^j\w{F(t,x,{\xi})u(t), u(t)}| > 0 \qquad
 \forall\,t, x, {\xi}
\end{equation*}
for any $0 \ne u(t) \in C^\infty(\br)$. 
\end{exe}

\section{Proof of Theorem \ref{subthm}}\label{thmpf}

By subtracting $z\id_N$ we may assume $z = 0$. Let $\wt w_0 \in
{\Sigma}_0(P)$, then by Definition~\ref{subdef} and Remark~\ref{nfrem}  there exist
a $C^\infty$ hypersurface ${\Sigma}$ and a real $C^\infty$  vector field
$V \notin T{\Sigma}$, an
invertible symmetrizer $M \in C^\infty$ so that
$Q = MP$ satisfies~\eqref{subcond}, the approximation property
on~${\Sigma}$, and
\begin{align}
&V\re Q  \ge c -  C \im Q \label{sub1}\\
&\im Q \ge c\, Q^*Q \label{sub2}
\end{align}
in a neighborhood ${\omega}$ of $\wt w_0$, here $c > 0$.

Since ~\eqref{sub1} is stable under small perturbations in $V$ we can
replace $V$ with $H_t$ for some real $t \in C^\infty$ after shrinking ~${\omega}$.
By solving the initial value problem $H_t {\tau} \equiv -1$,
${\tau}\restr{\Sigma} = 0$, and completing to a symplectic  $C^\infty$
coordinate system $(t,{\tau},x,{\xi})$
we obtain that $\Sigma = \set{\tau = 0}$ in a
neighborhood of ~$\wt w_0 = (0,0,w_0)$.
We obtain from Definition~\ref{subdef} that
\begin{equation} \label{sub3}
\re \w{Qu,u}  = 0
\qquad \text{when ${\tau}=0$ and $u \in \Cal V$}
\end{equation}
near $\wt w_0$.  Here $\Cal V$ is a $Q$ invariant $C^\infty$ subbundle of ~$\bc^N$
such that $\Cal V(\wt w_0) = \ker Q^N(\wt w_0)  = \ker Q(\wt w_0)$ by
Lemma~\ref{semiprop}.
By condition~\eqref{subcond} we have that
\begin{equation}\label{sub4}
\big| {\Omega}_{\delta}(\im Q_w) \cap
\set{|t| < c}\big| \le C {\phi}({\delta}) \qquad 0 < {\delta} \ll 1
\end{equation}
when $ |w -w_0| < c$, here $Q_w(t) = Q(t,0,w)$.
Since these are are all local conditions, we may assume that $M$ and
$Q \in C^\infty_\rb$.
We shall obtain Theorem~\ref{subthm} from the following estimate.

\begin{prop}\label{qestprop}
Assume that $Q \in C_\rb^\infty(T^*\br^n)$ is an $N \times N$ system
satisfying ~\eqref{sub1}--\eqref{sub4} in a neighborhood of\/ ~$\wt
w_0 = (0,0,w_0)$ with $V = \partial_{\tau}$ and non-decreasing
${\phi}({\delta}) \to 0$ as ${\delta} \to 0$. Then there exist~$h_0 >
0$ and $R \in C_\rb^\infty(T^*\br^n)$ so that $\wt w_0 \notin \supp R$
and
\begin{equation}\label{locsubest}
{\psi}(h)\mn {u} \le C (\mn{Q^w(t,x,hD_{t,x}) u} +
\mn{R^w(t,x,hD_{t,x})u} + h\mn u) \qquad 0 < h \le h_0  
\end{equation}
for any $u \in C^\infty_0(\br^n, \bc^N)$. Here ${\psi}(h) = {\delta}
\gg h$ is the inverse to $h = {\delta}{\phi}({\delta})$. 
\end{prop}

Let ${\omega}$ be a neighborhood of $\wt w_0$ such that $\supp R \bigcap
{\omega} = \emptyset$, where $R$ is given by Proposition~\ref{qestprop}.
Take ${\varphi}\in C^\infty_0({\omega})$ such that $0 \le {\varphi} \le
1$ and  ${\varphi} = 1$ in a neighborhood of ~$\wt w_0$. By substituting ${\varphi}^w(t,x,hD_{t,x})u$
in~\eqref{locsubest} we obtain from the calculus that for any $N$ we have
\begin{equation}\label{locsubest0}
   {\psi}(h)\mn {{\varphi}^w(t,x,hD_{t,x})u} \le C_N
 (\mn{Q^w(t,x,hD_{t,x}){\varphi}^w(t,x,hD_{t,x})u} + h^N\mn u) \qquad \forall\, u \in C^\infty_0
\end{equation}
for small enough ~$h$ since $R{\varphi} \equiv 0$. Now the commutator 
\begin{equation}\label{commest}
 \mn{[Q^w(t,x,hD_{t,x}),{\varphi}^w(t,x,hD_{t,x})]u} \le C h \mn u \qquad u \in C^\infty_0
\end{equation}
and since $Q= MP$ the calculus gives
\begin{multline}\label{simpest}
 \mn {Q^w(t,x,hD_{t,x})u}  \le \mn{M^w(t,x,hD_{t,x})P(h)u} + C h \mn u
 \\ \le
 C'(\mn{P(h)u} +  h \mn u) \qquad u \in C^\infty_0 
\end{multline}
The estimates~\eqref{locsubest0}--\eqref{simpest} give
\begin{equation}\label{locsubest1}
 {\psi}(h) \mn {{\varphi}^w(t,x,hD_{t,x})u} \le  C(\mn{P(h)u} +  h \mn u)
\end{equation}

Since $0 \notin {\Sigma}_\infty(P)$ we obtain by using the Borel Theorem finitely many
functions ${\phi}_j\in C^\infty_0$, $j = 1,\dots,N$, such that
$0 \le {\phi}_j \le 1$, $\sum_j {\phi}_j = 1 $ on $ {\Sigma}_0(P)$
and  the estimate
~\eqref{locsubest1} holds with ${\phi} = {\phi}_j$. Let
${\phi}_0 = 1 - \sum_{j\ge 1} {\phi}_j$, then since $0 \notin
{\Sigma}_\infty(P)$ we find that $\mn{P^{-1}} \le C$ on $\supp 
{\phi}_0$. Thus ${\phi}_0 = {\phi}_0P^{-1}P$ and the calculus gives
\begin{equation}
 \mn{{\phi}_0^w(t,x,hD_{t,x})u } \le C(\mn{P(h)u} + h\mn u)\qquad u \in C^\infty_0
\end{equation}
By summing up, we obtain 
\begin{equation}\label{locsubest2}
 {\psi}(h) \mn{u } \le C( \mn{P(h)u} + h\mn u)\qquad u \in C^\infty_0
\end{equation}
Since $h = {\delta}{\phi}({\delta}) \ll {\delta}$ we find $ {\psi}(h) ={\delta}
\gg h$ when $h \to 0$. Thus, we find 
for small enough ~$h$ that the last term in the right hand side of~\eqref{locsubest2} can be
cancelled by changing the constant, then $P(h)$ is injective with closed
range. Since $P^*(h)$ also is of subelliptic type ~${\phi}$ by
Proposition~\ref{fininv} we obtain the estimate
~\eqref{locsubest2} for $P^*(h)$. Thus $P^*(h)$ is injective making
$P(h)$ is surjective, which together with ~\eqref{locsubest2} gives Theorem~\ref{subthm}.

\begin{proof} [Proof of Proposition~\ref{qestprop}]
First we shall prepare the symbol $Q$ locally near ~$\wt w_0 =
(0,0,w_0)$. Since $\im Q \ge 0$ we obtain from Lemma~\ref{semiprop} that 
$\ran Q(\wt w_0) \bot \ker Q(\wt w_0)$ which gives $\ker Q^N(\wt w_0) = \ker Q(\wt w_0)$.
Let $\dim \ker Q(\wt w_0) = K$ then  by choosing an
orthonormal base and multiplying from the left as in
Example~\ref{Qex}, we may assume that 
\begin{equation}
 Q = 
\begin{pmatrix}
Q_{11} & 0 \\ 0 & Q_{22}
\end{pmatrix} 
\end{equation}
where $Q_{11}$ is $K \times K$ matrix, $Q_{11}(\wt w_0) = 0$ and $|Q_{22}(\wt w_0)| \ne 0$. 
Also, we find that $Q_{11}$ satisfies
the conditions~\eqref{sub1}--\eqref{sub4} with $\Cal V = \bc^K$ near $\wt w_0$.

Now it suffices to prove the estimate with $Q$ replaced by $Q_{11}$.
In fact, by using the ellipticity of $Q_{22}$ at~$\wt w_0$ we find 
\begin{equation}\label{q22est}
 \mn {u''} \le C(\mn{Q_{22}^w u''} + \mn{R_1^wu''} +h\mn {u''} \qquad
 u'' \in C^\infty_0(\br^n, \bc^{N-K}) 
\end{equation}
where $u = (u',u'')$ and $\wt w_0 \notin \supp R_1$. Thus, if we have
the estimate~\eqref{locsubest} for $Q_{11}^w$ with $R = R_2$, then 
since ${\psi}(h)$ is bounded we obtain the estimate for $Q^w$:
\begin{equation*}
{\psi}(h) \mn u \le C_0(\mn{Q_{11}^wu'} + \mn{Q_{22}^wu''} +\mn{R^wu}
+ h \mn u) \le C_1(\mn{Q^wu} +\mn{R^wu} +h\mn u)
\end{equation*}
where $\wt w_0 \notin \supp R$, $R = (R_1,R_2)$.

Thus, in the following we may assume that $Q= Q_{11}$ is $K \times K$
system satisfying the conditions
~\eqref{sub1}--\eqref{sub4} with $\Cal V = \bc^K$ near $\wt w_0$.
Since $\partial_{\tau}\re Q > 0$ at~$\wt w_0$ by~\eqref{sub1}, we find
from the matrix version of the Malgrange Preparation Theorem
in ~\cite[Theorem~4.3]{de:prep} that
\begin{equation}
Q(t,{\tau},w) = E(t,{\tau},w)({\tau}\id + K_0(t,w)) \qquad \text{near ~$\wt w_0$}
\end{equation}
where $E$, $K_0 \in C^\infty$, and $\re E > 0$ at $\wt w_0$.
By taking $M(t,w) = E(t,0,w)$ we find $\re M > 0$ and
$$
 Q(t,{\tau},w) = E_0(t,{\tau},w)({\tau}M(t,w) + iK(t,w))=
 E_0(t,{\tau},w) Q_0(t,{\tau},w)
$$
where $E_0(t,0,w) \equiv \id$. Thus we find that $Q_0$ satisfies
~\eqref{sub2}, \eqref{sub3} and~\eqref{sub4} when ${\tau} =
0$  near $\wt w_0$. Since $K(0, w_0) = 0$ we obtain that  $\im K
\equiv 0$ and $K \ge c K^2 \ge 0$ near ~$(0,w_0)$. We have $\re M > 0$ and 
\begin{equation}\label{realpartest}
 |\w{\im Mu,u}| \le C\w{Ku,u}^{1/2}\mn u \qquad  \text{near ~$(0,w_0)$}
\end{equation}
In fact, we have 
\begin{equation*}
 0 \le \im Q \le K + {\tau}(\im M + \re (E_1K)) + C{\tau}^2
\end{equation*}
where $E_1(t,w) = \partial_{\tau}E(t,0,w)$.
Lemma~\ref{semidefest} gives 
\begin{equation*}
  |\w{\im Mu,u} + \re \w{E_1Ku,u} | \le C\w{Ku,u}^{1/2}\mn u
\end{equation*}
and since $K^2 \le CK$ we obtain
\begin{equation*}
 | \re \w{E_1Ku,u}| \le C\mn{Ku}\mn u  \le  C_0\w{Ku,u}^{1/2}\mn u
\end{equation*}
which gives ~\eqref{realpartest}. Now by cutting off when $|{\tau}|
\ge c > 0$ we obtain that
\begin{equation*}
Q^w = E_0^wQ_0^w + R_0^w + hR_1^w
\end{equation*}
where $R_j \in C_\rb^\infty$ and $\wt w_0 \notin
\supp R_0$. Thus, it suffices to prove the estimate~\eqref{locsubest} for
$Q_0^w$. 
We may now reduce to the case when $\re M \equiv \id$. In fact, 
\begin{equation*}
 Q_0^w \cong M_0^w((\id + iM_1^w)hD_t +i K_1^w)M_0^w \qquad\text{modulo
   $\Cal O(h)$}
\end{equation*}
where $M_0 = (\re M)^{1/2}$ is invertible, $M_1^* = M_1$ and $K_1 =
M_0^{-1} K M_0^{-1} \ge 0$. By changing $M_1$ and $K_1$ and making $K_1
> 0$ outside a neighborhood of 
$(0, w_0)$ we may assume that $M_1$, $K_1 \in C^\infty_\rb$ and ~$iK_1$ satisfies
~\eqref{sub4} for all $c > 0$ and any ~$w$, by the invariance given by
Remark~\ref{subinvrem}. Observe that
condition~\eqref{realpartest} also is 
invariant under the mapping $Q_0 \mapsto E^*Q_0 E$.

We shall use the symbol classes $f \in S(m, g)$ defined by 
\begin{equation*}
 |\partial_{{\nu}_1}\dots \partial_{{\nu}_k}f | \le C_k m\, \prod_{j=1}^k
 g({\nu}_j)^{1/2} \qquad \forall\ {\nu}_1, \dots ,{\nu}_k \quad \forall\, k
\end{equation*}
for constant weight $m$ and metric $g$, and $\op S(m,g)$ the
corresponding Weyl operators ~$f^w$.
We shall need the following estimate for the model operator~$Q^w_0$.

\begin{prop}\label{locestprop}
Assume that 
$$
Q = M^w(t,x,hD_x)hD_t + iK^w(t, x, hD_x)
$$ 
where $M(t,w)$ and $0 \le K(t,w) \in  L^\infty(\br,
C^\infty_\rb(T^*\br^n))$ are $N \times N$ system such that  $\re M \equiv \id$,
$\im M$ satisfies~\eqref{realpartest} and  $iK$ satisfies~\eqref{sub4} for
all ~$w$ and $c > 0$ with non-decreasing $0 <{\phi}({\delta}) \to 0$ as ${\delta}
\to 0$. Then there exists a real valued $B(t,w) \in L^\infty(\br, S(1,
H|dw|^2/h))$ such that $hB(t,w)/{\psi}(h) \in \lip(\br, S(1,
H|dw|^2/h))$, and
\begin{equation}\label{locest}
{\psi}(h)\mn u^2  \le \im \w{Qu,B^w(t,x,hD_x)u} +C h^2\mn{D_tu}^2  \qquad 0 < h
\ll 1
\end{equation}
for any $u \in C^\infty_0(\br^{n+1}, \bc^N)$. Here the bounds on
$B(t,w)$ are uniform, ${\psi}(h) = {\delta}
\gg h$ is the inverse to $h = {\delta}{\phi}({\delta})$ so  $0 < H =
\sqrt{h/{\psi}(h)} \ll 1$ as $h \to 0$. 
\end{prop}

Observe that $H^2 = h/{\psi}(h)= {\phi}({\psi}(h))  \to 0$ and $h/H =
\sqrt{{\psi}(h)h} \ll {\psi}(h) \to 0$ as $h \to 0$, since 
$0 < {\phi}({\delta}) $ is non-decreasing.

To prove Proposition~\ref{qestprop} we shall cut off where $|{\tau}|
\gtrless {\varepsilon}\sqrt{{\psi}}/h$. Take ${\chi}_0(r) \in C^\infty_0(\br)$ such
that $0 \le {\chi}_0 \le 1$, ${\chi}_0(r) = 1$ when $|r| \le 1$ and
$|r| \le 2$ in $\supp {\chi}_0$. Then $1-{\chi}_0 = {\chi}_1$ where $0
\le {\chi}_1 \le 1$ is supported where 
$|r| \ge 1$. Let ${\phi}_{j,{\varepsilon}}(r) =
{\chi}_j(hr/{\varepsilon}\sqrt{{\psi}})$, $j = 0$, $1$,
for ${\varepsilon} > 0$, then ${\phi}_{0,{\varepsilon}}$ is supported
where $|r| \le 2{\varepsilon}\sqrt{{\psi}}/h$ and  ${\phi}_{1,{\varepsilon}}$ is supported
where $|r| \ge {\varepsilon}\sqrt{{\psi}}/h$. We have that
${\phi}_{j,{\varepsilon}}({\tau}) \in S(1, h^2d{\tau}^2/{\psi})$, $j = 0$, $1$,
and $u = 
{\phi}_{0,{\varepsilon}}(D_t)u + {\phi}_{1,{\varepsilon}}(D_t)u$,
where we shall estimate each term separately. Observe that we shall
use the ordinary quantization and not the semiclassical
for these operators. 

To estimate the first term, we
substitute ${\phi}_{0,{\varepsilon}}(D_t)u$ in ~\eqref{locest}. We find that 
\begin{multline}\label{est1}
 {\psi}(h)\mn{{\phi}_{0,{\varepsilon}}(D_t)u}^2 \le \im \w{Qu,
   {\phi}_{0,{\varepsilon}}(D_t) B^w(t,x,hD_x){\phi}_{0,{\varepsilon}}(D_t)u} \\
 +
 \im\w{[Q, {\phi}_{0,{\varepsilon}}(D_t)]u,
   B^w(t,x,hD_x){\phi}_{0,{\varepsilon}}(D_t)u} + 4C{\varepsilon}^2{\psi}\mn u^2
\end{multline}
In fact, $h\mn{D_t{\phi}_{0,{\varepsilon}}(D_t)u}
\le 2{\varepsilon}\sqrt{{\psi}}\mn u$
since it is a Fourier multiplier 
and $|h{\tau}{\phi}_{0,{\varepsilon}}({\tau})| \le
2{\varepsilon}\sqrt{{\psi}}$. Next we shall estimate the commutator term.
Since $\re Q = hD_t\id - h \partial_t \im M^w/2$ and $\im Q = h \im M^w
D_t + K^w +  h \partial_t \im M^w/2i$ we find that
$[\re Q, {\phi}_{0,{\varepsilon}}(D_t)] \in 
  \op S(h, {\Cal G})$
and
\begin{equation*}
 [Q, {\phi}_{0,{\varepsilon}}(D_t)] = i[\im Q,
 {\phi}_{0,{\varepsilon}}(D_t)] \\
 = i [K^w,  {\phi}_{0,{\varepsilon}}(D_t)] =
 -h\partial_tK^w
 {\phi}_{2,{\varepsilon}}(D_t) /{\varepsilon}\sqrt{{\psi}}
\end{equation*}
is a symmetric operator modulo $\op S(h, \Cal G)$, where
$\Cal G =  dt^2 + h^2d{\tau}^2/{\psi} + |dx|^2 + h^2|d{\xi}|^2$ and $
{\phi}_{2,{\varepsilon}}({\tau}) = {\chi}_0'(h{\tau}/{\varepsilon}\sqrt{{\psi}})$.
In fact, we have that $h^2/{\psi}(h) \le C h$, $h[\partial_t \im M^w,
{\phi}_{0,{\varepsilon}}(D_t)]$ and $ [\im M^w,
{\phi}_{0,{\varepsilon}}(D_t)]hD_t \in
\op S(h, \Cal G)$, since $|{\tau}| \le {\varepsilon}\sqrt{{\psi}}/h$ in
$\supp {\phi}_{0,{\varepsilon}}({\tau})$. Thus, we find that 
\begin{multline}\label{comterm}
 -2i \im  \big({\phi}_{0,{\varepsilon}}(D_t) B^w [Q,
 {\phi}_{0,{\varepsilon}}(D_t)]\big) = 2i  h{\varepsilon}^{-1}{\psi}^{-1/2}
 \im  \big({\phi}_{0,{\varepsilon}}(D_t) B^w\partial_tK^w
 {\phi}_{2,{\varepsilon}}(D_t)\big) \\
=  h{\varepsilon}^{-1}{\psi}^{-1/2} \Big
({\phi}_{0,{\varepsilon}}(D_t) B^w [\partial_tK^w,
 {\phi}_{2,{\varepsilon}}(D_t)] +  {\phi}_{0,{\varepsilon}}(D_t)[ B^w
, {\phi}_{2,{\varepsilon}}(D_t)]\partial_tK^w \\
+  {\phi}_{2,{\varepsilon}}(D_t)[{\phi}_{0,{\varepsilon}}(D_t), B^w]\partial_tK^w +
{\phi}_{2,{\varepsilon}}(D_t) B^w [ {\phi}_{0,{\varepsilon}}(D_t), \partial_tK^w ] \Big)
\end{multline}
modulo $\op S(h, \Cal G)$. As before, the calculus gives that $[ {\phi}_{j,{\varepsilon}}(D_t),
\partial_tK^w ] \in \op S(h{\psi}^{-1/2}, \Cal G)$ for any $j$. 
Since $t \to h B^w/{\psi} \in \lip(\br, \op S(1, \Cal G))$ uniformly
and ${\phi}_{j,{\varepsilon}}({\tau}) =
{\chi}_j(h{\tau}/{\varepsilon}\sqrt{{\psi}})$ with ${\chi}'_j \in
C^\infty_0(\br)$, Lemma~\ref{liplem} with ${\kappa}= {\varepsilon}\sqrt{{\psi}}/h$ gives that 
\begin{equation*}
 \left\| [{\phi}_{j,{\varepsilon}}(D_t), B^w] \right\|_{\Cal
   L(L^2(\br^{n+1}))} \le C\sqrt{{\psi}}/{\varepsilon}
\end{equation*}
uniformly. If we combine the estimates above we can estimate the commutator term: 
\begin{equation}\label{est2}
 | \im\w{[Q, {\phi}_{0,{\varepsilon}}(D_t)]u,
   B^w(t,x,hD_x){\phi}_{0,{\varepsilon}}(D_t)u}| \le Ch \mn u^2 \ll
 {\psi}(h) \mn u^2 \qquad h \ll 1
\end{equation}
which together with ~\eqref{est1} will give the estimate for the first term for small enough
~${\varepsilon}$ and~$h$.

We also have to estimate ${\phi}_{1,{\varepsilon}}(D_t)u$, then we shall use that $Q$
is elliptic when $|{\tau}| \ne 0$. We have
$$
\mn{{\phi}_{1,{\varepsilon}}(D_t)u}^2 = \w{{\chi}^w(D_t)u,u}
$$ 
where ${\chi}({\tau}) = {\phi}_{1,{\varepsilon}}^2({\tau})\in S(1,
h^2d{\tau}^2/{\psi})$ is real with support where 
$|{\tau}| \ge {\varepsilon}\sqrt{{\psi}}/h$. Thus, we may write
$
 {\chi}(D_t) = {\varrho}(D_t)hD_t$
where ${\varrho}({\tau}) = {\chi}({\tau})/h{\tau} \in
S({\psi}^{-1/2},  h^2d{\tau}^2/{\psi})$ by Leibniz' rule since
$|{\tau}|^{-1} \le h/{\varepsilon}\sqrt{{\psi}}$ in ~$\supp
{\varrho}$. Now $hD_t\id = \re Q + h\partial_t \im M^w/2$ so we find
\begin{equation*}
 \w{{\chi}(D_t)u,u} = \re\w{{\varrho}(D_t)Qu, u} +
 \frac{h}{2}\re\w{{\varrho}(D_t)(\partial_t  \im M^w)u,u}+ \im 
 \w{{\varrho}(D_t)\im Qu,u}
\end{equation*}
where $|h\re\w{{\varrho}(D_t)(\partial_t \im M^w)u,u}| \le Ch\mn
u^2/{\varepsilon}\sqrt{{\psi}}$ and 
$$|\re\w{{\varrho}(D_t)Qu, u}| \le \mn{Qu}\mn{{\varrho}(D_t)u} \le 
\mn{Qu}\mn u/{\varepsilon}\sqrt{{\psi}}$$ 
since ${\varrho}(D_t)$ is a Fourier multiplier and $|{\varrho}({\tau})| \le 1
/{\varepsilon}\sqrt{{\psi}}$. We have that
$$
\im Q = K^w(t,x,hD_x) +hD_t \im M^w(t,x,hD_x)
- \frac{h}{2i}\partial_t \im M^w(t,x,hD_x) 
$$
where $\im M^w(t,x,hD_x)$ and $K^w(t,x,hD_x) \in \op S(1, \Cal G)$ are symmetric.
Since ${\varrho} = {\chi}/h{\tau} \in S({\psi}^{-1/2}, \Cal G)$ is
real we find that 
\begin{multline*}
 \im ({\varrho}(D_t)\im Q) = \im {\varrho}(D_t) K^w + \im
 {\chi}(D_t)\im M^w \\= \frac{1}{2i}([{\varrho}(D_t),K^w(t,x,hD_x)] +
 [{\chi}(D_t),\im M^w(t,x,hD_x)])
\end{multline*}
modulo terms in $\op S(h/\sqrt{{\psi}}, \Cal G) \subseteq \op S(h/{{\psi}}, \Cal G)$.
Here the calculus gives
$$
[{\varrho}(D_t),K^w(t,x,hD_x)] \in \op  S(h/{\psi},  \Cal G)
$$ 
and similarly we have that 
$$ [{\chi}(D_t),\im M^w(t,x,hD_x)]
\in \op S(h/\sqrt{{\psi}},  \Cal G) \subseteq \op S(h/{\psi},  \Cal G)$$
which gives that $|\im \w{{\varrho}(D_t)\im Qu,u}| \le
Ch\mn u^2/{\psi}$. In fact, since the metric $\Cal G$ is constant,
it is uniformly ${\sigma}$ temperate for all $h>0$. We obtain that 
\begin{equation*}
 {\psi}(h)\mn{{\phi}_{1,{\varepsilon}}(D_t)u}^2 \le
 C_{\varepsilon}(\sqrt{\psi}\mn{Qu}\mn u + h\mn u^2)
\end{equation*}
which together with \eqref{est1} and~ \eqref{est2} gives the
estimate~\eqref{locsubest} for small enough~${\varepsilon}$ and~$h$,
since $h/{\psi}(h) \to 0$ as~$h \to 0$. 
\end{proof}

\begin{proof} [Proof of Proposition~\ref{locestprop}]
We shall do a second
microlocalization in $w = (x,{\xi})$.
By making a linear symplectic change of coordinates:
$(x,{\xi}) \mapsto (h^{1/2}x, h^{-1/2}{\xi})$ 
we obtain that $Q(t,{\tau}, x, h{\xi})$ is changed into
 $$
Q(t,{\tau},h^{1/2}w) \in S(1, dt^2 + d{\tau}^2 + h |dw|^2)
\qquad\text{when $|{\tau}| \le c$}
 $$
In these coordinates we find $B(h^{1/2}w) \in S(1,G)$,  $G = H|dw|^2$,
if $B(w) \in S(1,H|dw|^2/h)$. In the following, we shall use
ordinary Weyl quantization in the $w$ ~variables.

We shall follow an approach similar to the one of 
\cite[Section ~5]{dsz}. 
To localize the estimate we take
$\set{{\phi}_j(w)}_j$, $\set{{\psi}_j(w)}_j \in S(1,G)$
with values in $\ell^2$,
such that $0 \le {\phi}_j$,  $0 \le {\psi}_j$, $\sum_{j} {\phi}^2_j
\equiv 1$ and ${\phi}_j {\psi}_j = {\phi}_j$, $\forall\,j$. We may also
assume that ${\psi}_j$ is supported in a $G$
neighborhood of $w_j$. This can be done uniformly in ~$H$, by taking 
${\phi}_j(w) = {\Phi}_j(H^{1/2} w)$ and ${\psi}_j(w) =
{\Psi}_j(H^{1/2} w)$, with $\set{{\Phi}_j(w)}_j$ and
$\set{{\Psi}_j(w)}_j \in S(1, |dw|^2)$. 
Since $\sum {\phi}_j^2 = 1$ and $G = H|dw|^2$ the calculus gives
\begin{equation}
\sum_j \mn{{\phi}_j^w(x,D_x)u}^2 -C H^2 \mn{u}^2 \le
\mn{u}^2 \le \sum_j \mn{{\phi}_j^w(x,D_x)u}^2 + C
H^2 \mn{u}^2
\end{equation}
for $u \in C^\infty_0(\br^n)$, thus for small enough $H$ we find 
\begin{equation}\label{pouest}
\sum_j \mn{{\phi}_j^w(x,D_x)u}^2  \le 2\mn{u}^2 \le
4 \sum_j \mn{{\phi}_j^w(x,D_x)u}^2\qquad\text{for $u \in C^\infty_0(\br^n)$}
\end{equation}
Observe that since ${{\phi}_j}$ has values in $\ell^2$
we find that $\set{{\phi}_j^wR_j^w}_j\in \op S(H^{{\nu}}, G)$
also has values in $\ell^2$ if $R_j \in S(H^{{\nu}}, G)$
uniformly. Such terms will be summable:
\begin{equation}\label{errest}
 \sum_{j}^{}\mn{r_j^wu }^2 \le CH^{2{\nu}}\mn u^2
\end{equation}
for $\set{r_j}_j \in S(H^{\nu}, G)$ with values in $\ell^2$,
see ~\cite[p.\ 169]{ho:yellow}.
Now we fix~$j$ and let  
$$Q_j(t,{\tau}) =Q(t,{\tau},h^{1/2}w_j) = M_j(t) {\tau} + iK_j(t)
$$
where $M_j(t) = M(t,h^{1/2}w_j)$ and $K_j(t) = K(t,h^{1/2}w_j) \in
L^\infty(\br)$. Since   $K(t,w) \ge 0$ we find from
Lemma~\ref{semidefest} and~\eqref{realpartest} that
\begin{equation}\label{sub3a}
|\w{\im M_j(t)u,u}| + |\w{d_wK(t,h^{1/2}w_j)u,u}| \le C\w{K_j(t)u,u}^{1/2}\mn u
\qquad \forall\, u \in \bc^N\qquad \forall\,t
\end{equation}
and condition ~\eqref{sub4} means that 
\begin{equation}\label{sub4a}
\left|\set{t: \inf_{|u| = 1}\w{K_j(t)u,u} \le {\delta}}\right|
\le C {\phi}({\delta})
\end{equation}
We shall prove an estimate for the corresponding one-dimensional
operator 
$$Q_j(t,hD_t) = M_j(t) hD_t  + iK_j(t)$$
by using the following result.

\begin{lem}\label{locestproppen}
Assume that 
$$
Q(t,hD_t) = M(t)hD_t + iK(t)
$$ 
where $M(t)$ and $0 \le K(t)$ are $N \times N$
systems, which are uniformly bounded in $ L^\infty(\br)$, such that $\re M \equiv \id$,
$\im M$ satisfies~\eqref{realpartest} for almost all~$t$ and $iK$ 
satisfies~\eqref{sub4} for any $c > 0$ with non-decreasing ${\phi}({\delta}) \to 0$ as ${\delta}
\to 0$. Then there exists a uniformly bounded real $B(t) \in
L^\infty(\br)$ so that $hB(t)/{\psi}(h) \in \lip(\br)$ uniformly and
\begin{equation}\label{locesten}
{\psi}(h)\mn u^2 + \w{Ku,u} \le \im \w{Q u,Bu} + C h^2\mn{D_tu}^2 \qquad 0 < h
\ll 1
\end{equation}
for any $u \in C^\infty_0(\br, \bc^N)$.  Here ${\psi}(h) = {\delta}
\gg h$ is the inverse to $h = {\delta}{\phi}({\delta})$.
\end{lem}

\begin{proof} 
Let $0 \le {\Phi}_{h}(t) \le 1$ be the characteristic function of the set
${\Omega}_{{\delta}}(K)$  with ${\delta} = {\psi}(h)$.
Since ${\delta} ={\psi}(h)$ is the inverse of $h ={\delta}{\phi}({\delta})$
we find that ${\phi}({\psi}(h)) = h/{\delta} = h/{\psi}(h)$.
Thus, we obtain from~\eqref{sub4a} that
\begin{equation*}
 \int {\Phi}_{h}(t)\,dt  = |{\Omega}_{{\delta}}(K)| \le C
 h/{\psi}(h) 
\end{equation*}
Let 
\begin{equation}
E(t)= \exp \left( \frac{{\psi}(h)}{h}\int^{t}_0 {\Phi}_h(s)\, ds \right)
\end{equation}
then we find that $E$ and $E^{-1} \in L^{\infty}(\br)$ uniformly and
$hE'/ {\psi}(h) = {\Phi}_hE$ in~ $\Cal D'(\br)$.
We have
\begin{multline}\label{easyest}
E(t)Q(t,hD_t)E^{-1}(t) = Q(t,hD_t)+ E(t)h [M(t) D_t,
E^{-1}(t)]\id_N \\=  Q(t,hD_t) + i  {\psi}(h){\Phi}_h(t)\id_N -
{\psi}(h) {\Phi}_h(t)\im M(t) 
\end{multline} 
since $(E^{-1})' = - E'E^{-2}$. In the following, we let
\begin{equation}\label{Fjdef}
F(t)= K(t) + {\psi}(h)\id_N \ge  {\psi}(h)\id_N
\end{equation}
By the definition we have ${\Phi}_h(t) < 1 \implies K(t) \ge {\psi}(h)\id_N$, so 
$$
K(t) + {\psi}(h){\Phi}_h(t)\id_N \ge \frac{1}{2}F(t)
$$ 
Thus by taking the inner product in $L^2(\br)$ we find from ~\eqref{easyest} that 
\begin{multline*}
  \im \w{E(t) Q(t,hD_t)E^{-1}(t)u,u}  \\ \ge \frac{1}{2}\w{
   F(t)u,u} +  \w{\im M(t)hD_tu,u} -ch\mn u^2   \qquad u \in C^\infty_0(\br, \bc^N)
\end{multline*}
since  $\im Q(t,hD_t)= K(t)+ \im M(t)hD_t +  \frac{h}{2i}\partial_t
\im M(t)$. Now we may use~\eqref{realpartest} to estimate for any
${\varepsilon} > 0$
\begin{equation}\label{mest}
 |\w{\im MhD_tu,u}| \le {\varepsilon}\w{Ku,u} + C_{\varepsilon}(h^2\mn{D_t
 u}^2 + h\mn u^2)\qquad \forall\, u \in C^\infty_0(\br, \bc^N) 
\end{equation}
In fact, $u = {\chi}_0(hD_t)u +  {\chi}_1(hD_t)u$ where ${\chi}_0(r) \in
C_0^\infty(\br)$ and $|r| \ge 1$ in $\supp {\chi}_1$. We obtain from
~\eqref{realpartest} for any ${\varepsilon}>0$ that
\begin{equation*}
 |\w{\im M(t) {\chi}_0(h{\tau})h{\tau}u,u}| \le
 C\w{K(t)u,u}^{1/2}|{\chi}_0(h{\tau})h{\tau}|\mn u \le {\varepsilon}\w{K(t)u,u} +
 C_{\varepsilon}\mn{{\chi}_0(h{\tau}) h{\tau}u}^2
\end{equation*}
so by using G\aa rdings inequality in Proposition~\ref{garding} on
$${\varepsilon}K(t) +  C_{\varepsilon}{\chi}_0^2(h D_t)
h^2 D_t^2\pm \im M(t) {\chi}_0(hD_t)hD_t$$ we obtain 
\begin{equation*}
  |\w{\im M(t) {\chi}_0(hD_t)hD_t u,u}| \le {\varepsilon}\w{K(t)u,u} + C_{\varepsilon}h^2\mn{D_t
 u}^2 + C_0h\mn u^2\quad \forall\, u \in C^\infty_0(\br, \bc^N) 
\end{equation*}
since $\mn{ {\chi}_0(hD_t)hD_t u} \le C\mn{hD_tu}$. The other term is easier to estimate:
\begin{equation*}
  |\w{\im M(t) {\chi}_1(hD_t) hD_t u,u}| \le C\mn{hD_t u}\mn{ {\chi}_1(hD_t)
    u} \le C_1 h^2\mn{D_tu}^2
\end{equation*}
since $|{\chi}_1(h{\tau})| \le C|h{\tau}|$. By taking ${\varepsilon} =
1/6$ in ~\eqref{mest} we obtain
\begin{equation*}
  \w{F(t)u,u} \le 3 \im \w{E(t) Q(t,hD_t)E^{-1}(t)u,u} +  C(h^2\mn{D_t
 u}^2 + h\mn u^2)
\end{equation*}
Now $hD_tEu = EhD_tu -i {\psi}(h){\Phi}_hEu$ so we find by substituting $E(t)u$ that
\begin{multline}\label{conjest}
 {\psi}(h)\mn {E(t) u}^2 + \w{KE(t)u,E(t)u} \\ \le 3\im
\w{Q(t,hD_t)u,E^2(t)u}  +  C(h^2\mn{D_t
 u}^2 + h\mn u^2 + {\psi}^2(h)\mn{E(t)u}^2)
\end{multline} 
for $u \in C_0^\infty(\br, \bc^N)$. Since $E \ge c$, $K \ge 0$ and $h \ll
{\psi}(h) \ll 1$ when $h
\to 0$ we obtain~\eqref{locesten} with scalar $B =  {\varrho}E^2$ for
${\varrho} \gg 1$ and $h \ll 1$.
\end{proof}

To finish the proof of Proposition~\ref{locestprop}, we substitute ${\phi}_j^wu$ in
the estimate~\eqref{locesten} with $Q = Q_j$ to obtain that
\begin{equation}\label{newest}
  {\psi}(h) \mn{{\phi}_j^wu}^2 + \w{K_j{\phi}_j^wu,{\phi}_j^wu}
 \le \im \w{{\phi}_j^wQ_j(t,hD_t)u, B_j(t) {\phi}_j^wu} + Ch^2\mn{{\phi}_j^wD_tu}^2
\end{equation}
for $u \in C^\infty_0(\br^{n+1}, \bc^N)$, since ${\phi}_j^w(x,D_x)$ and $Q_j(t,hD_t)$ commute.
Next, we shall replace the approximation $Q_j$ by the original operator
$Q$. In a $G$ neighborhood of $\supp {\phi}_j$ we may
use the Taylor expansion in $w$ to write for almost all $t$
\begin{equation}\label{taylorappr}
 Q(t,{\tau},h^{1/2}w) - Q_j(t,{\tau})= i(K(t,h^{1/2}w) - K_j(t)) + (M(t,h^{1/2}w) - M_j(t)){\tau}
\end{equation}
We shall start by estimating the last term
in~\eqref{taylorappr}. Since $M(t,w) \in C^\infty_\rb$ we have 
\begin{equation}\label{Mest0}
 |M(t,h^{1/2}w) - M_j(t)| \le C h^{1/2}H^{-1/2}\qquad \text{in $\supp {\phi}_j$}
\end{equation}
because then $|w-w_j| \le c H^{-1/2}$. Since $M(t,h^{1/2}w) \in S(1,
h|dw|^2)$ and $h \ll H$ we find from~\eqref{Mest0} that $M(t,h^{1/2}w) - M_j(t) \in S(h^{1/2}H^{-1/2},
G)$ in ~$\supp {\phi}_j$ uniformly in ~~$t$.
By the Cauchy-Schwarz inequality we find
\begin{equation}\label{Mest}
 |\w{{\phi}_j^w(M^w- M_j)hD_tu, B_j(t) {\phi}_j^wu}| \le C(\mn{{\chi}_j^whD_tu}^2
 + hH^{-1}\mn{{\phi}_j^wu}^2)
\end{equation}
for $ u \in C^\infty_0(\br^{n+1}, \bc^N)$ where ${\chi}_j^w =
h^{-1/2}H^{1/2}{\phi}_j^w(M^w- M_j) \in \op S(1, 
 G)$ uniformly in $t$ with values in $\ell^2$. Thus we find from~\eqref{errest} that
\begin{equation*}
 \sum_{j} \mn{{\chi}_j^whD_tu}^2 \le C \mn {hD_tu}^2  \qquad u \in
C^\infty_0(\br^{n+1}) 
\end{equation*}
and for the last terms in ~\eqref{Mest} we have 
\begin{equation*}
 hH^{-1}\sum_{j} \mn{{\phi}_j^w u}^2 \le 2hH^{-1} \mn u^{2} \ll {\psi}(h) \mn u^2 \qquad h \to 0 \qquad u \in
C^\infty_0(\br^{n+1}) 
\end{equation*}
by ~\eqref{pouest}.
For the first term in the right hand side of~\eqref{taylorappr} we find from Taylor's
formula
\begin{equation*}
  K(t,h^{1/2}w) - K_j(t) =
 h^{1/2}\w{S_j(t),W_j(w)} + R_j(t,{\tau},w) \qquad \text{in $\supp {\phi}_j$}
\end{equation*}
where $S_j(t) =  \partial_w K(t,h^{1/2}w_j) \in L^\infty(\br)$,  $R_j \in
S(hH^{-1}, G)$ uniformly for almost all ~$t$ and $W_j \in S(h^{-1/2}, h|dw|^2) $ 
such that ${\phi}_j(w)W_j(w) = {\phi}_j(w)(w-w_j) = \Cal
O(H^{-1/2})$. Here we could take $W_j(w) =
{\chi}(h^{1/2}(w-w_j))(w-w_j)$ for a suitable cut-off function ${\chi}
\in C^\infty_0$. We obtain from the calculus that
\begin{equation}\label{xref}
{\phi}_j^wK_j(t) = {\phi}_j^wK^w(t,h^{1/2}x,h^{1/2}D_x) -
 h^{1/2} {\phi}_j^w\w{S_j(t), W_j^w} + \wt R^w_j 
\end{equation}
where $\set{\wt R_j}_j \in S(hH^{-1}, G)$
with values in $\ell^2$ for almost all~$t$. Thus we may estimate the sum
of these error terms by~\eqref{errest} to obtain
\begin{equation}\label{yref}
 \sum_j|\w{\wt R_j^wu,B_j(t){\phi}_j^wu }| \le C hH^{-1}\mn u^2 \ll
 {\psi}(h)\mn u^2\qquad h \to 0
\end{equation}
for $ u \in
C^\infty_0(\br^{n+1}, \bc^N)$. Observe that it follows from
~\eqref{sub3a} for any ${\kappa} > 0$ and almost all~~$t$ that
\begin{equation*}
|\w{S_j(t)u,u}| \le C
 \w{K_j(t)u,u}^{1/2}\mn u \le {\kappa}\w{K_j(t)u,u} + 
 C\mn u^2/{\kappa} \qquad \forall\, u \in \bc^N
\end{equation*}
Let 
$F_j(t) = F(t, h^{1/2}w_j) =
K_j(t) +  {\psi}(h)\id_N$, then 
by taking ${\kappa} = {\varrho}H^{1/2}h^{- 1/2}$ we find that for
any~${\varrho} > 0$ there exists ~~$h_{\varrho} >  0$ so that   
\begin{multline}\label{apprest}
h^{1/2}H^{-1/2}|\w{S_ju,u}|  \le {\varrho} \w{K_ju,u} + 
 Ch H^{-1}\mn u^2/{\varrho} \\\le {\varrho}\w{F_ju,u} \qquad \forall\, u \in
 \bc^N\qquad 0 < h \le h_{\varrho}
\end{multline}
since $hH^{-1} \ll {\psi}(h)$ when $h \ll 1$.
Now $F_j$ and $S_j$ only depend on ~$t$, so by~\eqref{apprest} we may
use Remark~\ref{grem} in the Appendix for fixed~ $t$ 
with $A = h^{1/2}H^{-1/2} S_j$, $B = {\varrho}F_j$, $u$ replaced with ${\phi}_j^wu$
and $v$ with $B_j H^{1/2}{\phi}_j^wW_j^wu$. Integration
then gives
\begin{equation}\label{apprest0}
 h^{1/2}|\w{B_j{\phi}_j^w \w{
     S_j(t),  W_j^w}u, {\phi}_j^w u}| \le \frac {3\varrho}2(\w{F_j(t){\phi}_j^wu,{\phi}_j^wu} 
+ \w{F_j(t){\psi}_j^wu,{\psi}_j^wu})
\end{equation}
for $ u \in C^\infty_0(\br^{n+1}, \bc^N) $, $0 < h \le h_{\varrho}$, where
$$ 
{\psi}_j^w = B_j H^{1/2}{\phi}_j^wW_j^w \in
\op S(1,G) \qquad\text{with values in $\ell^2$}
$$
In fact, since ${\phi}_j \in S(1,G)$ and $W_j \in S(h^{-1/2},
h|dw|^2)$ we find that ${\phi}_j^wW_j^w = ({\phi}_jW_j)^{w}$ modulo
$\op S(H^{1/2},G)$, and since $|{\phi}_jW_j| \le CH^{-1/2}$ we find from Leibniz'
rule that ${\phi}_jW_j \in S(H^{-1/2}, G)$.
Now $F \ge {\psi}(h)\id_N \gg hH^{-1}\id_N$ so by using
Proposition~\ref{cleanupest} in the Appendix and then integrating in~$t$ we find that
$$
\sum_j\w{F_j(t){\psi}_j^wu, {\psi}_j^wu} \le
C \sum_j\w{F_j(t){\phi}_j^wu,{\phi}_j^wu}  \qquad u \in
C^\infty_0(\br^{n+1}, \bc^N) $$ 
We obtain from ~\eqref{pouest} and~\eqref{Fjdef} that
\begin{equation*}
   {\psi}(h) \mn u^2 \le 2\sum_{j}^{}
 \w{F_j(t){\phi}_j^wu,{\phi}_j^wu}  \qquad u \in
C^\infty_0(\br^{n+1}, \bc^N) 
\end{equation*}
Thus, for any ${\varrho} > 0$ we obtain from~\eqref{newest} and
\eqref{Mest}--\eqref{apprest0} that 
\begin{equation*}
(1-C_0{\varrho})\sum_{j}^{} \w{F_j(t){\phi}_j^wu,{\phi}_j^wu} \le \sum_{j} 
 \im \w{{\phi}_j^wQ u, B_j(t) {\phi}_j^wu} 
+  C_{\varrho}h^2\mn{D_tu}^2  \qquad 0 < h \le h_{\varrho}
\end{equation*}
We have that $
\sum_{j}B_j{\phi}_j^w{\phi}_j^w \in S(1, G)$ is a scalar symmetric
operator uniformly in ~$t$.
When ${\varrho} = 1/2C_0$ we obtain the
estimate~\eqref{locest} with
$B^w = 4\sum_{j}B_j{\phi}_j^w{\phi}_j^w$, which
finishes the proof of Proposition~\ref{locestprop}.
\end{proof}

\appendix
\section{}

We shall first study the condition for the
one-dimensional model operator 
$$hD_t\id_N + i F(t) \qquad 0 \le F(t) \in C^\infty(\br) $$ 
to be of finite type of order ${\mu}$:
\begin{equation}\label{subcond01}
\left|{\Omega}_{\delta}(F)\right| \le C {\delta}^{\mu} \qquad 0 <
{\delta} \ll 1
\end{equation}
and we shall assume that $0 \notin {\Sigma}_\infty(P)$.
When $F(t)\notin C^\infty(\br)$ we may have any ${\mu} > 0$ in ~\eqref{subcond01}, for
example with $F(t) = |t|^{1/{\mu}}\id_N$. 
But when $F \in C^1_\rb$ the estimate cannot hold with ${\mu} > 1$,
and since it trivially holds for ${\mu}= 0$ 
the only interesting cases are $0 < {\mu} \le 1$.

When $0 \le F(t)$ is diagonalizable for any ~$t$ with eigenvalues
${\lambda}_j(t) \in C^\infty$, $j = 1, \dots, N$, then condition~\eqref{subcond01} is
equivalent to 
\begin{equation*}
 \left|{\Omega}_{\delta}({\lambda}_j)\right|\le C {\delta}^{\mu}
 \qquad \forall\,j \quad 0 < {\delta} \ll 1 
\end{equation*}
since ${\Omega}_{\delta}(F) = \bigcup_j
 {\Omega}_{\delta}({\lambda}_j)$.
Thus we shall start by studying the scalar case.

\begin{prop}\label{scalarsubcondlem}
Assume that $0 \le f(t) \in C^\infty(\br)$ such that $f(t) \ge c > 0$ when
$|t| \gg 1$, i.e., $0 \notin {\Sigma}_\infty(f)$. We find that $f$
satisfies ~\eqref{subcond01} with ${\mu} 
> 0$ if and only if ${\mu} \le 1/k$ for an even $k \ge 0$ so that
\begin{equation}\label{dersubcondsc}
 \sum_{j \le k}|\partial_t^jf(t)|  > 0 \qquad
 \forall\,t
\end{equation}
\end{prop}

Simple examples as $f(t) = e^{-t^2}$ show that the condition that
$0 \notin {\Sigma}_\infty(f)$ is necessary for the conclusion of 
Proposition~\ref{scalarsubcondlem}.

\begin{proof}
Assume that ~\eqref{dersubcondsc} does not hold with $k \le 1/{\mu}$, then there exists
$t_0$ such that $f^{(j)}(t_0) = 0$ for all integer $j \le 1/{\mu}$.
Then Taylor's formula gives that $f(t) \le c |t-t_0|^k$ and
$|{\Omega}_{\delta}(f)| \ge c {\delta}^{1/k}$ where $k = [1/{\mu}] + 1> 1/{\mu}$,
which contradicts condition~\eqref{subcond01}. 

Assume now that condition~\eqref{dersubcondsc} holds for some $k$, 
then $f^{-1}(0)$ consists of finitely many points. In fact, else there
would exist $t_0$ where $f$ vanishes of infinite order since $f(t)
\ne 0$ when $|t| \gg 1$. Also note that $\bigcap_{{\delta} >
  0}{\Omega}_{\delta}(f) = f^{-1}(0)$, in fact $f$
must have a positive infimum outside any neighborhood of $f^{-1}(0)$.
Thus, in order to estimate $|{\Omega}_{\delta}(f)|$ for
${\delta} \ll 1$ we only have to consider a small 
neighborhood ~${\omega}$ of $t_0 \in f^{-1}(0)$. Assume that
\begin{equation*}
 f(t_0) = f'(t_0) = \dots = f^{(j-1)}(t_0) =
 0\text{ and } f^{(j)} (t_0)\ne 0
\end{equation*}
for some $j \le k$.
Since $f \ge 0$ we
find that $j$ must be even and $ f^{(j)} (t_0) > 0$.
Taylor's formula gives as before
$
 f(t) \ge c|t -t_0|^j $ for  $ |t - t_0| \ll 1 
$
and thus we find that 
$$\left|{\Omega}_{\delta}(f) \bigcap {\omega}\right| \le C
{\delta}^{1/j} \le C{\delta}^{1/k} \qquad 0 < {\delta} \ll 1$$ 
if ${\omega}$ is a small neighborhood of $t_0$. Since $f^{-1}(0)$
consists of finitely many points we find that
\eqref{subcond01} is satisfied with ${\mu} = 1/k$ for an even $k$.
 \end{proof}

Thus, if $0 \le F \in  C^\infty(\br)$ is $C^\infty$ diagonalizable
system and $0 \notin {\Sigma}_\infty(P)$
then condition  ~\eqref{subcond01} is equivalently to 
\begin{equation}
 \sum_{j \le k}|\partial_t^j\w{F(t)u(t), u(t)}|/\mn {u(t)}^2  > 0 \qquad
 \forall\,t
\end{equation}
for any $0 \ne u(t) \in C^\infty(\br)$, since this holds for
diagonal matrices and is invariant.
This is true also in the general case by the following proposition.

\begin{prop}\label{subcondlem}
Assume that $0 \le F(t) \in C^\infty(\br)$ is an $N \times N$ system
such that $0 \notin {\Sigma}_\infty(F)$. 
We find that $F$ satisfies ~\eqref{subcond01} with ${\mu} > 0$ if
and only if ${\mu} \le 1/k$ for an even $k \ge 0$ so that
\begin{equation}\label{dersubcond}
 \sum_{j \le k}|\partial_t^j\w{F(t)u(t), u(t)}|/\mn{u(t)}^2  > 0 \qquad
 \forall\,t
\end{equation}
for any $0 \ne u(t) \in C^\infty(\br)$.
\end{prop}

Observe that since  $0 \notin {\Sigma}_\infty(F)$ it suffices to check
condition~\eqref{dersubcond} on a compact interval.

\begin{proof}
First we assume that~\eqref{subcond01} holds with ${\mu} > 0$,
let $u(t) \in C^\infty(\br, \bc^N)$ such that $|u(t)| \equiv 1$,
and $f(t) = \w{F(t)u(t),u(t)} \in C^\infty(\br)$. Then we
have $ {\Omega}_{\delta}(f) \subset {\Omega}_{\delta}(F)$ so
 ~\eqref{subcond01} gives
\begin{equation*}
 |{\Omega}_{\delta}(f)| \le |{\Omega}_{\delta}(F)| \le C{\delta}^{\mu}
 \qquad 0 < {\delta} \ll 1
\end{equation*}
The first part of the proof of Proposition~\ref{scalarsubcondlem} then gives
~\eqref{dersubcond} for some $k \le 1/{\mu}$.

For the proof of the sufficiency of ~\eqref{dersubcond} we need the
following simple lemma.

\begin{lem}\label{analem}
Assume that $F(t) = F^*(t) \in C^k(\br)$ is an $N \times N$ system with
eigenvalues ${\lambda}_j(t) \in \br$,  $j = 1,\dots, N$.
Then, for any $t_0 \in \br$, there exist analytic $v_{j}(t) \in \bc^N$, $j =
1,\dots, N$, so that  $\set{v_{j}(t_0)} $ is a base for $\bc^N$ and
\begin{equation}
\left|{\lambda}_j(t) - \w{F(t)v_{j}(t), v_{j}(t)}\right| \le
C|t-t_0|^k \qquad \text{for $|t - t_0| \le 1$}
\end{equation}
after a renumbering of the eigenvalues.
\end{lem}

By a well-known theorem of Rellich, the eigenvalues ${\lambda}(t)
\in C^1(\br)$  for symmetric
$F(t) \in C^1(\br)$, see~\cite[Theorem~II.6.8]{kato}.

\begin{proof}
It is no restriction to assume $t_0 = 0$. By Taylor's formula
\begin{equation*}
 F(t) = F_k(t) + R_k(t)
\end{equation*}
where $F_k$ and $R_k$ are symmetric, $F_k(t)$ is a polynomial of
degree $k-1$ and $R_k(t) = \Cal O(|t|^k)$. Since $F_k(t)$ is symmetric
and holomorphic, it has a base of normalized holomorphic eigenvectors $v_{j}(t)$
with real holomorphic eigenvalues $\wt {\lambda}_{j}(t)$
by~\cite[Theorem~II.6.1]{kato}.
Thus $\wt {\lambda}_j(t) = \w{F_k(t) v_j(t),v_j(t)}$ and 
by the minimax principle we may renumber the eigenvalues so that 
\begin{equation*}
 |{\lambda}_j(t) - \wt {\lambda}_{j}(t)| \le \mn{R_k(t)} \le C|t|^k
 \qquad \forall\,j 
\end{equation*}
Since 
\begin{equation*}
 | \w{(F(t)-F_k(t))v_{j}(t), v_{j}(t)}| =  |\w{R_k(t)v_{j}(t),
   v_{j}(t)}| \le C|t|^k \qquad \forall\,j
\end{equation*}
we obtain the result.
\end{proof}

Assume now that~\eqref{dersubcond} holds for some $k$.
As in the scalar case, we have that $k$ is even and
$\bigcap_{{\delta} >
  0}{\Omega}_{\delta}(F)= {\Sigma}_0(F) = |F|^{-1}(0)$.
Thus, for small ${\delta}$ we only have to consider a small
neighborhood of $t_0 \in {\Sigma}_0(F)$. Then by using
Lemma~\ref{analem} we have after renumbering that for each
eigenvalue ${\lambda}(t)$ of $F(t)$ there exists $v(t) \in C^\infty$ so
that $|v(t)| \ge c >0$ and
\begin{equation}\label{evap}
 |{\lambda}(t) - \w{F(t)v(t),v(t)}| \le C|t-t_0|^{k+1} \qquad
 \text{when $|t-t_0| \le c$}
\end{equation}
Now if ${\Sigma}_0(F) \ni t_j
\to t_0$ is an accumulation point, then after choosing a subsequence
we obtain that for some eigenvalue ${\lambda}_k$ we have
${\lambda}_k(t_j) = 0$, $\forall\, j$. Then ${\lambda}_k$ vanishes
of infinite order at $t_0$, contradicting ~\eqref{dersubcond} by ~\eqref{evap}.
Thus, we find that ${\Sigma}_0(F)$ is a finite collection of points.
By using~\eqref{dersubcond} with $u(t) = v(t)$ we find as in
the second part of the proof of Proposition~\ref{scalarsubcondlem} that 
\begin{equation*}
 \w{F(t)v(t),v(t)} \ge c |t-t_0|^j \qquad
 |t-t_0| \ll 1
\end{equation*}
for some even $j \le k$, which by ~\eqref{evap} gives that
\begin{equation*}
 {\lambda}(t) \ge c |t-t_0|^j -C |t-t_0|^{k+1} \ge c' |t-t_0|^j \qquad
 |t-t_0| \ll 1
\end{equation*}
Thus $|{\Omega}_{\delta}({\lambda})\bigcap {\omega}| \le
c{\delta}^{1/j}$ if ${\omega}$ for ${\delta} \ll 1$ if ${\omega}$ is
a small neighborhood of $t_0 \in {\Sigma}_0(F)$. Since
${\Omega}_{\delta}(F) = \bigcup_j {\Omega}_{\delta}({\lambda}_j)$,
where $\set{{\lambda}_j(t)}_j$ are the eigenvalues of $F(t)$,
we find by adding up that $|{\Omega}_{\delta}(F)| \le
C{\delta}^{1/k}$. Thus the largest
${\mu}$ satisfying ~\eqref{subcond01} must be $\ge 1/k$.
\end{proof}

Let $A(t) \in \lip(\br, \Cal L (L^2(\br^n))) $ be the $L^2(\br^n)$
bounded operators which are Lipschitzcontinuous in the parameter $t\in
\br$. This means that 
\begin{equation} \label{lipop}
A(s)-A(t)/s-t = B(s,t) \in \Cal L (L^2(\br^n)) \qquad \text{uniformly
  in $s$ and $t$}
\end{equation}
One example is $A(t) = a^w(t,x,D_x)$ where $a(t,x,{\xi})
\in \lip(\br, S(1,G))$ for a ${\sigma}$ temperate metric $G$
which is constant in $t$ such that $G/G^{{\sigma}} \le 1$.

\begin{lem}\label{liplem}
Assume that  $A(t) \in \lip(\br, \Cal L (L^2(\br^n))) $ and
${\phi}({\tau}) \in C^\infty(\br)$ such that ${\phi}'({\tau}) \in
C^\infty_0(\br)$. Then for ${\kappa} > 0$ we can estimate the
commutator 
\begin{equation*}
 \left\|\big[{\phi}(D_t/{\kappa}), A(t) \big] \right\|_{\Cal L
     (L^2(\br^{n+1}))} \le C{\kappa}^{-1} 
\end{equation*}
where the constant only depends on ${\phi}$ and the bound on
$A(t) $ in $\lip(\br, \Cal L (L^2(\br^n)))$.
\end{lem}

\begin{proof}
In the following, we shall denote by $A(t,x,y)$ the distribution kernel of
$A(t)$. Then we find from ~\eqref{lipop} that 
\begin{equation}\label{lipref}
 A(s,x,y) - A(t,x,y) = (s -t)B(s,t,x,y)
\end{equation}
where $B(s,t,x,y)$ is the kernel for $B(s,t)$ for $s$, $t\in \br$. Then
\begin{multline}\label{comref}
\left\langle \big[{\phi}(D_t/{\kappa}), A(t)\big]u,v \right\rangle \\
  =(2{\pi})^{-1} \int e^{i(t-s){\tau}} 
 {\phi}({\tau}/{\kappa})(A(s,x,y)-A(t,x,y))u(s,x)\overline{v(t,y)}\,
 d{\tau} ds dt dx dy 
\end{multline}
for $u$, $ v \in C_0^\infty(\br^{n+1})$, and by using \eqref{lipref}
we obtain that the commutator has kernel
\begin{multline*}
(2{\pi})^{-1} \int  e^{i(t-s){\tau}} {\phi}({\tau}/{\kappa})(s -t)B(s,t,x,y)
 \,d{\tau} \\=  {\kappa}^{-1} \int  e^{i(t-s){\tau}} {\rho}({\tau}/{\kappa})B(s,t,x,y)
 \,d{\tau} = \widehat {\rho}({\kappa}(s-t)) B(s,t,x,y)
\end{multline*}
in $\Cal D(\br^{2n+ 2})$, where ${\rho} \in C^\infty_0(\br)$. 
Thus, we may estimate ~\eqref{comref} by using Cauchy-Schwarz:
\begin{equation*}
 \int |\widehat {\varrho}({\kappa}s)\w{B(s + t,t)u(s+t),  {v(t)}}_{L^2(\br^n)}|\,
 dtds \le C{\kappa}^{-1}\mn u \mn v
\end{equation*}
where the norms are in $\Cal L(L^2(\br^{n+1}))$.
\end{proof}

We also need some results about the lower bounds of systems, and
we shall use the following version of the G\aa rding inequality for systems.
A convenient way for proving the inequality is to use the Wick
quantization of $a \in L^\infty(T^*\br^n)$ given by
\begin{equation*}
a^{Wick}(x,D_x)u(x) = \int_{T^*\br^n}a(y,{\eta})
{\Sigma}^w_{y,{\eta}}(x,D_x)u(x)\,dyd{\eta}\qquad u \in  \Cal S(\br^n)
\end{equation*}
using the rank one orthogonal projections
${\Sigma}^w_{y,{\eta}}(x,D_x)$ in $L^2(\br^2)$ with Weyl symbol
$${\Sigma}_{y,{\eta}}(x,{\xi}) =
{\pi}^{-n}\exp\left(-|x-y|^2- |{\xi}-{\eta}|^2\right)
$$
(see~\cite[Appendix~B]{de:suff} or~\cite[Section~4]{ln:coh}).
We find that
$a^{Wick}$: $ \Cal S(\br^n) \mapsto \Cal S'(\br^n)$
is symmetric on $\Cal S(\br^n)$ if $a$ is real-valued,
\begin{equation} \label{poswick}
a \ge 0 
\implies
\sw{a^{Wick}(x,D_x)u,u} \ge
0 \qquad u \in \Cal S(\br^n)
\end{equation}
and
$ 
\mn{a^{Wick}(x,D_x)}_{\Cal L(L^2(\br^n))} \le
\mn{a}_{L^\infty(T^*\br^n)}, 
$ 
which is the main advantage with the Wick quantization. If $a \in S(1,
h|dw|^2)$ we find that 
\begin{equation}\label{weylwick}
a^{Wick} = a^w + r^w
\end{equation}
where $r \in S(h, h|dw|^2)$.  For a
reference, see \cite[Proposition~4.2]{ln:coh}.

\begin{prop}\label{garding}
Let\/ $0 \le A  \in C^\infty_\rb(T^*\br^n)$ be an $N \times N$ system, 
then we find that
\begin{equation*}
  \w{A^w(x,hD)u,u} \ge - Ch\mn u^2 \qquad \forall\, u \in
  C^\infty_0(\br^n, \bc^N)
\end{equation*}
\end{prop}

This result is well known, see for example Theorem~18.6.14
in~\cite{ho:yellow}, but
we shall give a short and direct proof.

\begin{proof}
By making a $L^2$ preserving linear symplectic change of coordinates:
$(x,{\xi}) \mapsto (h^{1/2}x, h^{-1/2}{\xi})$ 
we may assume that $0 \le A\in S(1, h|dw|^2)$. Then
we find from~\eqref{weylwick} that $A^w = A^{Wick} + R^w$ where   $R \in S(h,
h|dw|^2)$. Since $A \ge 0$ we obtain from~\eqref{poswick} that 
\begin{equation*}
 \w{A^wu,u} \ge \w{R^wu,u} \ge - Ch\mn u^2 \qquad \forall\, u \in
  C^\infty_0(\br^n, \bc^N)
\end{equation*}
which gives the result.
\end{proof}

\begin{rem}\label{grem}
Assume that $A$ and $B$ are $N \times N$ matrices such that $\pm A \le B$:
Then we find
\begin{equation*}
 \left| \w{Au,v} \right| \le \frac{3}{2}\left(  \w{Bu,u}  +
   \w{Bv,v} \right) \qquad \forall\, u, \ v \in \bc^N
\end{equation*}
\end{rem}

In fact, since $B \pm A \ge 0$ we find by the Cauchy-Schwarz inequality that
\begin{equation*}
 2 \left| \w{(B \pm A)u,v}\right| \le   \w{(B \pm A)u,u} +
  \w{(B \pm A)v,v} \qquad \forall\,
 u,\ v \in \bc^N
\end{equation*}
and $ 2  \left| \w{B u,v}\right| \le  \w{Bu,u} +  \w{Bv,v}$.
The estimate can then be expanded to give the inequality, since
\begin{equation*}
   \left|\w{Au,u} \right|  \le   \w{Bu,u}  \qquad
\forall\, u \in \bc^N
\end{equation*}
by the assumption.

\begin{lem}\label{semidefest}
Assume that $0 \le F(t)\in C^2(\br)$ is an $N \times N$ system such
that $F'' \in L^\infty(\br)$. Then we have  
\begin{equation*}
|\w{F'(0)u,u}|^2 \le C \mn{F''}_{L^\infty}\w{F(0)u,u}\mn u^2 \qquad \forall\,
u \in \bc^N
\end{equation*}
\end{lem}

\begin{proof}
Take $u\in \bc^N$ with $|u|= 1$ and let $0 \le f(t) = \w{F(t)u,u} \in
C^2(\br)$. Then $|f''| \le \mn{F''}_{L^\infty}$ so Lemma~7.7.2 in \cite{ho:yellow} gives
\begin{equation*}
 |f'(0)|^2 = |\w{F'(0)u,u}|^2 \le  C \mn{F''}_{L^\infty}f(0) = C\mn{F''}_{L^\infty} \w{F(0)u,u}
\end{equation*}
which proves the result.
\end{proof}

\begin{lem}\label{semiest} 
Assume that $F\ge 0$ is an $N \times N$ matrix and that $A$ is a $L^2$ bounded scalar operator, then 
\begin{equation*}
 0 \le \w{FAu,Au} \le \mn A^2\w{Fu,u}
\end{equation*}
for any $u \in C^\infty_0(\br^n, \bc^N)$.
\end{lem}

\begin{proof}
Since $F \ge 0$ we can choose an orthonomal base for~ $\bc^N$ such that
$\w{Fu,u} = \sum_{j=1}^{N} f_j |u_j|^2$ for $u = (u_1, u_2, \dots) \in
\bc^N$, where $f_j \ge 0$ are the eigenvalues of~$F$. In this base
we find
\begin{equation*}
 0 \le \w{FAu,Au} = \sum_{j} f_j \mn{Au_j}^2 \le \mn A^2 \sum_{j} f_j
 \mn{u_j}^2 =  \mn A^2\w{Fu,u}
\end{equation*}
for $u \in C^\infty_0(\br^n, \bc^N)$.
\end{proof}

\begin{prop}\label{cleanupest}
Assume that $h/H \le F \in S(1,g)$ is an $N \times N$ system,
$\set{{\phi}_j}$ and
$\set{{\psi}_j} \in S(1, G)$ with values in $\ell^2$ such that
$\sum_{j} |{\phi}_j|^2 \ge c > 0$ and 
${\psi}_j$ is supported in a fixed $G$ neighborhood of $w_j \in \supp
{\phi}_j$, $\forall\,j$. Here $g = h|dw|^2$ and $G = H|dw|^2$ are
constant metrics, $0 < h \le H \le 1$.
If $F_j = F(w_j)$ we find for $H \ll 1$ that
\begin{equation} \label{erroreq}
\sum_{j}
\w{F_j{\psi}_j^w(x,D_x)u, {\psi}_j^w(x,D_x)u} \le C \sum_{j}
\w{F_j{\phi}_j^w(x,D_x)u, {\phi}_j^w(x,D_x)u}
\end{equation}
for any $u \in C^\infty_0(\br^{n}, \bc^N)$.
\end{prop}

\begin{proof}
Since ${\chi} = \sum_{j}|{\phi}_j|^2  \ge c > 0$ we find that
${\chi}^{-1} \in S(1, G)$. The calculus gives
$$
({\chi}^{-1})^w\sum_{j}\ol {\phi}_j^w
{\phi}_j^w = 1 + r^w
$$ 
where $r \in S(H, G)$
uniformly in ~$H$. Thus, the mapping 
$
u \mapsto ({\chi}^{-1})^w\sum_{j}\ol{\phi}_j^w{\phi}_j^wu
$ 
is a homeomorphism on $L^2(\br^n)$ for small enough ~$H$.
Now the constant metric $G =
H|dw|^2$ is trivially {\em strongly ${\sigma}$ temperate} according to
Definition~7.1 in~\cite{bc:sob}, so Theorem~7.6 in~\cite{bc:sob} 
gives $B \in S(1, G)$ such that 
$$B^w({\chi}^{-1})^w\sum_{j}\ol {\phi}_j^w
{\phi}_j^w = \sum_{j} B_j^w{\phi}_j^w   = 1
$$ 
where $B_j^w = B^w({\chi}^{-1})^w\ol {\phi}_j^w \in \op S(1,G)$
uniformly, which gives $1= \sum_{j}\ol{\phi}_j^w \ol B_j^w$ since
$(B_j^w)^* = \ol B_j^w$. Now we shall put 
$$
\wt{\Cal F}^w(x,D_x) =  \sum_{j}
\ol{\psi}_j^w(x,D_x) F_j{\psi}_j^w(x,D_x) 
$$
then 
\begin{equation}\label{calcref}
\wt {\Cal F}^w = \sum_{jk} \ol {\phi}_j^w \ol B_j^w\wt{\Cal F}^w
B_k^w{\phi}^w_k = \sum_{jkl} \ol{\phi}_j^w \ol B_j^w \ol{\psi}_l^w F_l
{\psi}_l^w B_k^w{\phi}^w_k
\end{equation}
Let  $C_{jkl}^w = \ol B_j^w \ol{\psi}_l^w
{\psi}_l^wB_k^w$, then we find from~\eqref{calcref} that 
\begin{equation*}
 \w{\wt {\Cal F}^wu,u} = \sum_{jkl}\w{ F_lC_{jkl}^w{\phi}_k^w u,{\phi}_j^w u}
\end{equation*}
Let $d_{jk}$ be the $H^{-1}|dw|^2$ distance between the
$G$ neighborhoods in which ${\psi}_j$ and ${\psi}_k$ are supported.
The usual calculus estimates (see~\cite[p.\ 168]{ho:yellow}
or~\cite[Th.\ 2.6]{bc:sob}) gives
that the $L^2$ operator norm of $C_{jkl}^w$ can be estimated by
\begin{equation*}
 \mn{C_{jkl}^w} \le C_N(1+ d_{jl} + d_{lk})^{-N}
\end{equation*}
for any $N$. 
We find by Taylor's formula, Lemma~\ref{semidefest} and the
Cauchy-Schwarz inequality that
\begin{multline*}
 |\w{(F_j - F_k)u,u}| \le
 C_1 |w_j - w_k| \w{F_ku,u}^{1/2}h^{1/2}\mn u \\+
 C_2 h |w_j - w_k|^2 \mn u^2 \le C\w{F_ku,u}(1 + d_{jk})^2
\end{multline*}
since $|w_j - w_k| \le C(d_{jk}+H^{-1/2})$ and  $h \le hH^{-1} \le
F_k$. Since $F_l \ge 0$ we obtain that   
\begin{equation*}
 2|\w{F_lu,v}| \le \w{F_lu,u}^{1/2}\w{F_lv,v}^{1/2} \le
 C\w{F_ju,u}^{1/2}\w{F_kv,v}^{1/2} (1+
 d_{jl})(1 + d_{lk}) 
\end{equation*}
and Lemma~~\ref{semiest} gives
\begin{equation*}
 \w{ F_kC_{jkl}^w{\phi}_k^w u, F_kC_{jkl}^w{\phi}_k^w u } \le
 \mn{C_{jkl}}^2\w{F_k{\phi}_k^wu,{\phi}_k^wu} 
\end{equation*}
Thus we find that
\begin{multline*}
 \sum_{jkl}\w{ F_lC_{jkl}^w{\phi}_k^w u,{\phi}_j^w u} \le C_N
  \sum_{jkl}(1+ d_{jl} + d_{lk})^{2-N}
  \w{F_k{\phi}_k^wu,{\phi}_k^wu}^{1/2}
  \w{F_j{\phi}_j^wu,{\phi}_j^wu}^{1/2}\\
\le C_N\sum_{jkl}(1+ d_{jl})^{1-N/2}
(1+ d_{lk})^{1-N/2} \left(\w{F_j{\phi}_j^wu,{\phi}_j^wu} + \w{F_k{\phi}_k^wu,{\phi}_k^wu}\right)
\end{multline*}
Since 
\begin{equation*}
 \sum_j(1 + d_{jk})^{-N} \le C \qquad \forall\, k
\end{equation*}
for $N$ large enough by \cite[p.\ 168]{ho:yellow}), we obtain the
estimate~\eqref{erroreq} and the result. 
\end{proof}

\bibliographystyle{amsplain}

\providecommand{\bysame}{\leavevmode\hbox to3em{\hrulefill}\thinspace}
\providecommand{\MR}{\relax\ifhmode\unskip\space\fi MR }
\providecommand{\MRhref}[2]{%
  \href{http://www.ams.org/mathscinet-getitem?mr=#1}{#2}
}
\providecommand{\href}[2]{#2}

\end{document}